\documentclass{mc}
\renewcommand\datereceived{}
\renewcommand\dateaccepted{}
\usepackage{mathrsfs}
\usepackage{caption}
\usepackage{subcaption}
\usepackage{multirow}
\usepackage{xcolor,sectsty}
\definecolor{astral}{RGB}{46,116,181}
\definecolor{darkslategray}{rgb}{0.18, 0.31, 0.31}
\definecolor{warmblack}{rgb}{0.0, 0.46, 0.36}
\subsectionfont{\color{astral}}
\subsubsectionfont{\color{astral}}
\sectionfont{\color{astral}}
\newtheorem{note}[]{Note}
\usepackage{hyperref}
\usepackage{rotating}
\usepackage{algorithm}
\usepackage[]{algpseudocode}

\begin{document}
	
	
\markboth{S. Santra and R. Behera}{Enhancing accuracy with an adaptive discretization for the non-local integro-PDEs involving initial time singularities
 }
\title{{\color{warmblack} Enhancing accuracy with an adaptive discretization for the non-local integro-partial differential equations involving initial time singularities
 }}
	
	

	%
	\author[Santra and Behera]{Sudarshan Santra\affil{1} and Ratikanta Behera\affil{1}\corrauth}
 
	\address{\affilnum{1}{Department of Computational and Data Sciences, Indian Institute of Science (IISc), Bangalore, 560012, India.}}
	
	\emails{{\tt sudarshans@iisc.ac.in} (S. Santra), {\tt ratikanta@iisc.ac.in} (R. Behera)}
	%

\begin{abstract}
This work aims to construct an efficient and highly accurate numerical method to address the time singularity at $t=0$ involved in a class of time-fractional parabolic integro-partial differential equations in one and two dimensions. The $L2$-$1_\sigma$ scheme is used to discretize the time-fractional operator, whereas a modified version of the composite trapezoidal approximation is employed to discretize the Volterra operator in time. Subsequently, it helps to convert the proposed model into a second-order boundary value problem in a semi-discrete form. The multi-dimensional Haar wavelets are then used for grid adaptation and efficient computations for the two-dimensional problem, whereas the standard second-order approximations are employed to approximate the spatial derivatives for the one-dimensional case. The stability analysis is carried out on an adaptive mesh in time. The convergence analysis leads to $\mathcal{O}(N^{-2}+M^{-2})$ accurate solution in the space-time domain for the one-dimensional problem having time singularity based on the $L^\infty$ norm for a suitable choice of the grading parameter. Furthermore, it provides $\mathcal{O}(N^{-2}+\mathcal{M}^{-3})$ accurate solution for the two-dimensional problem having unbounded time derivative at $t=0$. The analysis also highlights a higher order accuracy for a sufficiently smooth solution resides in $C^3(\overline{\Omega}_t)$ even if the mesh is discretized uniformly. The truncation error estimates for the time-fractional operator, integral operator, and spatial derivatives are presented. In addition, we have examined the impact of various parameters on the robustness and accuracy of the proposed method. Numerous tests are performed on several examples in support of the theoretical analysis. The advancement of the proposed methodology is demonstrated through the application of the time-fractional Fokker-Planck equation and the fractional-order viscoelastic dynamics having weakly singular kernels. It also confirms the superiority of the proposed method compared with existing approaches available in the literature.

\end{abstract}
	
	
	
\maketitle
	


\section{Introduction}\label{sec intro}
During the past few decades, the qualitative analysis of fractional differential equations 	(FDEs) and fractional-order integro-differential equations (FOIDEs) has gained popularity owing to their practical application in numerous fields of science and technology, including control theory \cite{SinghEulerwavelets2024}, neural networks in cryptography \cite{RoohiAdaptive2020}, options trading \cite{AnAspace2021}, and viscoelasticity \cite{BjorklundErrorestimates2024}, etc. However, the present work is more focused on the analysis of the time-fractional integro-partial differential equations (TFIPDEs), which have a widespread application in the time-fractional Fokker-Planck equation \cite{DengNumerical2007} and the fractional-order viscoelastic dynamics. In this scenario, the current state of the system depends on the current time as well as on its past history over a certain time span which makes the fractional-order system more accurate in simulating a physical process that incorporates initial and boundary data. Furthermore, the singular behavior of the solution of a fractional-order model not only differs from the classical integer-order systems but also poses a greater challenge for solving them. 

To discuss the complexity of numerical simulation to address time singularity, in this work, we consider the time-fractional integro-partial differential equations involving time singularity at $t=0$, which can be expressed in the following generalized form:
 \begin{equation}\label{Pdoc1_1}
    \left\{
    \begin{array}{ll}
    \partial_t^\alpha \mathcal{U}(\boldsymbol{x},t)+\mathcal{L} \mathcal{U}(\boldsymbol{x},t)+\mu\displaystyle{\int_0^t} \mathcal{K}(\boldsymbol{x},t-\xi)\mathcal{U}(\boldsymbol{x},\xi)d\xi=f(\boldsymbol{x},t),\\[8pt]
    \mbox{for}~(\boldsymbol{x},t)\in\Omega\times\Omega_t,~ \mbox{with}\\[4pt]
    \mathcal{U}(\boldsymbol{x},0)=g(\boldsymbol{x})~\mbox{for}~\boldsymbol{x}\in \overline{\Omega},\\
    \mathcal{U}(\boldsymbol{x},t)=h(\boldsymbol{x},t)~\mbox{for}~(\boldsymbol{x},t)\in \partial\Omega\times\overline{\Omega}_t,
    \end{array}\right.
\end{equation}
 where $\alpha\in(0,1)$. $\Omega\subset \mathbb{R}^d$ ($d=1,2$) is an open and bounded domain with $\partial\Omega$ denotes its boundary, and $\Omega_t=(0,T]$. $\partial_t^\alpha$ denotes the fractional Caputo operator \cite{MillerAnintroduction1993} of order $\alpha$ with respect to $t$. The operator $\mathcal{L}$ is defined as:
\begin{equation*}
        \mathcal{L} \mathcal{U}(\boldsymbol{x},t):=-p(\boldsymbol{x})\Delta \mathcal{U}(\boldsymbol{x},t)+\boldsymbol{q}\cdot\nabla \mathcal{U}(\boldsymbol{x},t)+r(\boldsymbol{x})\mathcal{U}(\boldsymbol{x},t),
\end{equation*} 
with $\boldsymbol{q}=(q_1(\boldsymbol{x}),q_2(\boldsymbol{x}))$, $\nabla \mathcal{U}:=(\mathcal{U}_x,\mathcal{U}_y)$, and $\Delta \mathcal{U}:=\mathcal{U}_{xx}+\mathcal{U}_{yy}$. $p,q_1,q_2,r,g,h,f$ are sufficiently smooth functions with $r(\boldsymbol{x})\geq 0$, $p(\boldsymbol{x})\geq p_0>0$ for $\boldsymbol{x}\in\overline{\Omega}$. In particular, for 1D problem, when $\Omega\subset\mathbb{R}$, we have $\mathcal{L} \mathcal{U}(x,t):=-p(x)\mathcal{U}_{xx}(x,t)+q(x)\mathcal{U}_{x}(x,t)+r(x)\mathcal{U}(x,t)$. The kernel $\mathcal{K}$ is smooth and is considered to be positive and real-valued. $\mu$ is a fixed positive constant. Further, it is assumed that the solution and its temporal derivatives satisfy the following regularity conditions:
\begin{equation}\label{Pdoc1_52}
    \Big\lvert\dfrac{\partial^j\mathcal{U}}{\partial t^j}(\cdot,t)\Big\rvert\leq C(1+t^{\alpha-j}),~\mbox{for}~j=0,1,2,3,
\end{equation}
for all $t\in\Omega_t$. $C$ is some generic constants, which can take different values at different places throughout the manuscript. Also, note that the solution is sufficiently smooth in its spatial variables.

If, we consider $\mu=0$, then (\ref{Pdoc1_1}) becomes time-fractional sub-diffusion equations, which have been extensively studied in the literature \cite{HanLinearized2023,GraciaConvergence2018}; for that problems, it is well known that given smooth and compatible data, nevertheless the solution will exhibit a weak singularity at the initial time $t=0$. Despite the presence of such singularities of the temporal derivatives in the solution, many researchers make a priori assumption that higher-order temporal derivatives of the solution are smooth on the closed domain, to deal with the numerical analysis of finite difference techniques for solving them using Taylor expansions in their truncation error analyses. For instance, see the works discussed in \cite{GaoAfinite2012,LiFast2022,AlikhanovAhigh-order2021}. However, this approach is generally impractical for fractional-order differential models, as the fractional derivatives are nonlocal and involve weakly singular kernels. In the present research, we construct a hybrid numerical scheme, which is a combination of the non-uniform $L2$-$1_\sigma$ discretization and the multi-dimensional wavelet approximation to address the time singularity of the proposed problem. The non-uniform mesh in time is shown to be more effective in enhancing the rate of accuracy compared to the uniform mesh.

For a comprehensive study of fractional-order problems, including FDEs and FOIDEs with initial time singularities, we refer the readers to the works of Chen {\em et al.} \cite{ChenAnanalysis2019}, Santra {\em et al.} \cite{SantraHigherorder2023}, Babaei and Salehi \cite{BabaeiAnefficient2021}, and Dehghan {\em et al.} \cite{DehghanALegendre2018}. It is important to note that semi-analytical approaches such as the Adomian decomposition method \cite{SantraAnovelapproach2022}, the homotopy analysis method \cite{AbbasbandyOnconvergence2013}, and the variational iteration method \cite{HamoudModified2018} tend to be less accurate. These methods often require a lot of iterations to reach the desired accuracy, which makes the process more time-consuming. Several attempts have been made to develop accurate, stable, and high-order approximations for the fractional derivatives. Among them, the $L1$ discretization is a good approximation, which has been widely used in \cite{GraciaConvergence2018,SantraAnovel2022}, for fractional problems having unbounded time derivative at $t=0$. The method provides first-order convergence on any subdomain away from the origin, whereas it produces a $\mathcal{O}(N^{-\alpha})$ accurate solution over the entire region, where $N$ denotes the number of mesh intervals towards the temporal direction. This phenomenon is highlighted in \cite{SantraAnalysis2022,SantraAnalysis2021}. Even though the non-uniform mesh is more effective than the uniform mesh in capturing the initial layer, it fails to occur with second-order accuracy \cite{StynesErroranalysis2017}. Furthermore, considering the higher regularity assumption for the solution of the fractional problems, the $L1$ discretization gives $\mathcal{O}(N^{-(2-\alpha)})$ accuracy \cite{GaoAfinite2012}. In an attempt to develop high-order accuracy, Gao {\em et al.} \cite{GaoAnew2014} developed an efficient approximation called $L1$-$2$ scheme, which gives $\mathcal{O}(N^{-(3-\alpha)})$ temporal accuracy at $t_n~ (n\geq 2)$ but at $t_1$, it produces $\mathcal{O}(N^{-(2-\alpha)})$ convergence rate. In contrast, Alikhanov \cite{AlikhanovAnew2015} constructed the $L2$-$1_\sigma$ scheme on a uniform mesh that provides $\mathcal{O}(N^{-(3-\alpha)})$ accurate solution for all $t_n~(n\geq 1)$ for which the solution resides in $C^3(\overline{\Omega}_t)$, where $C^3(\overline{\Omega}_t)$ denotes the space of all functions whose third-order temporal derivatives exist and are continuous on $\overline{\Omega}_t$. However, the present work focuses on constructing a hybrid approach based on the $L2$-$1_\sigma$ scheme and the multi-dimensional Haar wavelets, which is designed on a non-uniform mesh in time to address the time singularity of the proposed model for $\Omega\subset\mathbb{R}^2$.

Wavelets are numerical concepts that allow one to represent a function in terms of basis functions, and wavelet-based numerical methods take advantage of the fact that functions with localized regions of sharp transitions are well compressed using wavelet decomposition \cite{BeheraMultilevel2015,SchneiderWavelet2010}. This property allows local grid refinement up to an arbitrarily small scale without a drastic increase in the number of collocation points; thus, high-resolution computations can be performed only in regions where sharp transitions occur. It is worth mentioning that in the last few decades, the Haar wavelet method has become a valuable tool for solving differential as well as integro-differential equations, and many authors have verified it \cite{PervaizHaar2020,Siraj-ul-IslamAcomparative2010}. Specifically, the Haar wavelet is more popular owing to its beneficial properties such as orthogonality, simple applicability, and compact support. Also, note that the basis functions are piecewise constant, for which it requires less effort to maintain the accuracy in the numerical approximation of the solution and its derivatives. In addition to this, the convergence of the Haar wavelet approximation demands less regularity compared to other wavelets like Hermite \cite{FaheemAhigh2022}, Bernoulli \cite{KeshavarzBernoulli2014}, Chebyshev \cite{LiSolving2010} wavelets that require a high regularity assumption on the solution to achieve the desired accuracy. The major contributions of the present study are highlighted as follows:

\begin{itemize}
    \item We design an efficient hybrid numerical scheme that combines the $L2$-$1_\sigma$ discretization and the multi-dimensional Haar wavelets with non-uniform temporal mesh to address a class of TFIPDEs in one and two dimensions to resolve the time singularity effectively. 

    \item The stability analysis of the proposed problem is carried out on a non-uniform mesh based on the $L^\infty$ norm. It helps to estimate the main theoretical convergence of the proposed hybrid scheme based on the $L^2$ norm. It also highlights the convergence on a uniformly distributed mesh.

    \item We confirm that the proposed scheme leads to a $\mathcal{O}(N^{-2})$ accurate solution in time on a non-uniform mesh with a suitable choice of grading parameter with unbounded time derivatives at $t=0$. In this case, the uniform mesh reduces to $\mathcal{O}(N^{-\alpha})$ accuracy.

    \item We discuss the truncation error estimates corresponding to the time-fractional operator, Volterra integral operator, and the spatial derivatives (when $\Omega\subset\mathbb{R}$), along with impact of these approximations on the error bounds. It can be noticed that the error bound of the integral operator is dominated by the error bound of the time-fractional operator, providing a temporal accuracy of $\mathcal{O}(N^{-2})$ over the space-time domain.
    
    \item Each of the theoretical investigations is validated through numerous test examples of type (\ref{Pdoc1_1}) for both $\Omega\subset\mathbb{R}/\mathbb{R}^2$. The results are compared with the results obtained by the $L1$ scheme, which provides a $\mathcal{O}(N^{-(2-\alpha)})$ temporal accuracy for a suitable choice of the grading parameter. In contrast, the proposed scheme leads to $\mathcal{O}(N^{-2})$ accuracy. It also highlights a higher rate of accuracy for a sufficiently smooth solution $\mathcal{U}\in C^3(\overline{\Omega}_t)$ even if the mesh is discretized uniformly.  

\end{itemize}

\noindent The rest of the manuscript is organized as follows.
We start with the temporal semi-discretization of the proposed problem on a non-uniform mesh in time in Section \ref{sec_Numerical_scheme}. The discretization of the time-fractional operator and the integral operator in time are also discussed. The approximation of the spatial derivatives for $\Omega\subset\mathbb{R}/\mathbb{R}^2$ are also demonstrated in this section. Specifically, when $\Omega\subset\mathbb{R}^2$, the multi-dimensional Haar wavelet decomposition is employed to approximate the spatial derivatives. The approximation is based on two-dimensional collocation points, which are distributed uniformly over $\Omega$. The convergence of this multi-dimensional approach is also examined. Section \ref{sec_Stability} is focused on the stability of the proposed TFIPDEs based on the $L^\infty$ norm for a suitably chosen non-uniform mesh. The convergence analysis is carried out in Section \ref{sec_Convergence}. It provides the truncation error analysis of the fractional derivative, integral operator, and the spatial derivatives and discusses its impact on the error bounds through several remarks. Further, it shows the theoretical error bounds of the proposed hybrid approach, which is a combination of $L2$-$1_\sigma$ scheme and the multi-dimensional Haar wavelets when $\Omega\subset\mathbb{R}^2$, illustrating a $\mathcal{O}(N^{-2})$ temporal accuracy on a non-uniform mesh with a suitably chosen grading parameter for solution having time singularity. It also highlights the theoretical convergences for a sufficiently smooth solution resides in $C^3(\overline{\Omega}_t)$. Each of the theoretical arguments is validated through numerous test examples in Section \ref{sec_example}. Also discusses how the proposed method can be used with the time-fractional Fokker-Planck equation and the fractional-order viscoelastic dynamics with weakly singular kernels. It also confirms the superiority of the proposed methodology over existing approaches. The manuscript ends with concluding remarks in Section \ref{sec_concl} demonstrating the main contribution of the present work and possible extension for future investigation.

\section{Numerical methods} \label{sec_Numerical_scheme}
The solution of the proposed problem (\ref{Pdoc1_1}) has an unbounded time derivative at $t=0$. To address this issue, a graded mesh is used in the temporal direction. Designing a uniform mesh towards spatial direction is sufficient as the proposed problem remains smooth in its spatial variables.

\subsection{The temporal semi-discretization based on graded mesh}
Let $N\in\mathbb{N}$ be fixed, and define $t_n=T(n/N)^\nu$ for $n=0,1,\ldots, N$. Then, we have the graded mesh in time as $\varOmega_t^N:=\{t_n~:~n=0,1,\ldots, N\}$. $\nu\geq1$ is the grading parameter, in particular for $\nu\neq1$ will lead to a non-uniform mesh in time. However, for $\nu=1$ the mesh $\varOmega_t^N$ becomes uniform. Further, we define the step size $\tau_n=t_n-t_{n-1}$ for $n=1,2,\ldots, N$, which satisfies the following bounds:
\begin{equation}\label{Pdoc1_53}
\left\{
    \begin{array}{ll}
         \tau_1=TN^{-\nu},  \\
         \tau_n=T\Big(\dfrac{n}{N}\Big)^\nu-T\Big(\dfrac{n-1}{N}\Big)^\nu\leq CTN^{-\nu}(n-1)^{\nu-1},~n=2,3,\ldots,N.
    \end{array}\right.
\end{equation}
The Alikhanov $L2$-$1_\sigma$ scheme \cite{AlikhanovAnew2015} is used on the graded mesh to approximate the fractional operator $\partial_t^\alpha$ at $t=t_{n+\sigma}$, which can be described briefly as follows:
\begin{align}
    \nonumber\partial_t^\alpha \mathcal{U}(\cdot,t_{n+\sigma})\approx {}^{L2\mbox{-}1_\sigma}\mathcal{D}_N^\alpha\mathcal{U}^{n+\sigma}:=&\dfrac{1}{\Gamma(1-\alpha)}\Big( \sum_{j=1}^n \int_{t_{j-1}}^{t_j}(t_{n+\sigma}-\rho)^{-\alpha}\mathscr{P}'_{2,j}(\cdot,\rho)d\rho+\int_{t_n}^{t_{n+\sigma}}(t_{n+\sigma}-\rho)^{-\alpha}\mathscr{P}'_{1,j}(\cdot,\rho)d\rho\Big)\\[2pt]
    =&\Lambda_n^n\mathcal{U}^{n+1}-\Lambda_n^0\mathcal{U}^0-\sum_{j=1}^n (\Lambda_n^{j}-\Lambda_n^{j-1})\mathcal{U}^{j},~\mbox{for}~n=0,1,\ldots,N-1,\label{Pdoc1_2}
\end{align}
 where $\mathscr{P}_{2,j}(\cdot,\rho)$ denotes the quadratic polynomial that interpolates $\mathcal{U}(\cdot,\rho)$ at the points $t_{j-1},t_j$ and $t_{j+1}$. Whereas $\mathscr{P}_{1,j}(\cdot,\rho)$ is the linear polynomial that interpolates $\mathcal{U}(\cdot,\rho)$ at the points $t_{j},t_{j+1}$. Further, $t_{n+\sigma}=t_n+\sigma\tau_{n+1}$ for some $\sigma\in(0,1)$. The coefficients can be expressed as: $\Lambda_0^0=\tau_1^{-1}a_{0,0}$ and for $n=1,2,\ldots,N-1$, we have
 \begin{equation}\label{Pdoc1_22}
    \Lambda_n^{j}=\dfrac{1}{\tau_{j+1}}\begin{cases}
                    a_{n,0}-b_{n,0},  &j=0, \\
                    a_{n,j}+b_{n,j-1}-b_{n,j},~&1\leq j\leq n-1,\\
                    a_{n,n}+b_{n,n-1}, &j=n,
                    \end{cases}
\end{equation}
where $a_{n,n}=\dfrac{\sigma^{1-\alpha}}{\Gamma(2-\alpha)}\tau_{n+1}^{1-\alpha}$ for $n=0,1,\ldots,N-1$, and for $n=1,2,\ldots,N-1$ and $j=0,1,\ldots,n-1$, one has
\begin{align*}
    a_{n,j}=\dfrac{1}{\Gamma(1-\alpha)}\int_{t_j}^{t_{j+1}}(t_{n+\sigma}-\rho)^{-\alpha}d\rho,~~b_{n,j}=\dfrac{2}{(t_{j+2}-t_j)\Gamma(1-\alpha)}\int_{t_j}^{t_{j+1}}(t_{n+\sigma}-\rho)^{-\alpha}(\rho-t_{j+1/2})d\rho.
\end{align*}
The Volterra integral operator can be discretized as follows:
\begin{align}
    \nonumber\displaystyle{\int_{0}^{t_{n+\sigma}}}\mathcal{K}(\cdot,t_{n+\sigma}-\xi)\mathcal{U}(\cdot,\xi)\;d\xi=&\sum_{j=0}^{n-1}\displaystyle{\int_{t_j}^{t_{j+1}}}\mathcal{K}(\cdot,t_{n+\sigma}-\xi)\mathcal{U}(\cdot,\xi)\;d\xi+\displaystyle{\int_{t_n}^{t_{n+\sigma}}}\mathcal{K}(\cdot,t_{n+\sigma}-\xi)\mathcal{U}(\cdot,\xi)\;d\xi\\
    \nonumber=& \sum_{j=0}^{n-1}\dfrac{\tau_{j+1}}{2}\big[\mathcal{K}(\cdot,t_{n+\sigma}-t_{j+1})\mathcal{U}(\cdot,t_{j+1})+\mathcal{K}(\cdot,t_{n+\sigma}-t_j)\mathcal{U}(\cdot,t_j)\big]\\
    \nonumber&\hspace{3cm}+\dfrac{\sigma\tau_{n+1}}{2}\big[\mathcal{K}(\cdot,t_{n+\sigma}-t_n)\mathcal{U}(\cdot,t_n)+\mathcal{K}(\cdot,0)\mathcal{U}(\cdot,t_{n+\sigma})\big]\\
    \nonumber\approx&\sum_{j=0}^{n-1}\dfrac{\tau_{j+1}}{2}\big[\mathcal{K}(\cdot,t_{n+\sigma}-t_{j+1})\mathcal{U}^{j+1}+\mathcal{K}(\cdot,t_{n+\sigma}-t_j)\mathcal{U}^j\big]\\
    \nonumber&\hspace{3cm}+\dfrac{\sigma\tau_{n+1}}{2}\mathcal{K}(\cdot,t_{n+\sigma}-t_n)\mathcal{U}^n+\dfrac{\sigma\tau_{n+1}}{2}\mathcal{K}(\cdot,0)\big[\sigma\mathcal{U}^{n+1}+(1-\sigma)\mathcal{U}^n\big]\\
    :=&\mathscr{J}_N \mathcal{U}^{n+\sigma}.\label{Pdoc1_4}
\end{align}
Based on this discretization, one yields
\begin{equation}\label{Pdoc1_5}
    \left\{
    \begin{array}{ll}
    {}^{L2\mbox{-}1_\sigma}\mathcal{D}_N^\alpha\mathcal{U}(\cdot,t_{n+\sigma})+\mathcal{L}\mathcal{U}(\cdot,t_{n+\sigma})+\mu\mathscr{J}_N\mathcal{U}(\cdot,t_{n+\sigma})=f(\cdot,t_{n+\sigma})+{}^{(1)}\mathscr{R}^{n+\sigma} +{}^{(2)}\mathscr{R}^{n+\sigma},\\[4pt]
    \mathcal{U}(\cdot,t_{n+1})\rvert_{\partial\Omega}=h(\cdot,t_{n+1}),
    \end{array}\right.
\end{equation}
for $n=0,1,\ldots,N-1$ with $\mathcal{U}(\cdot,t_{0})=g(\cdot)$. The remainder terms are given by
\begin{equation}\label{Pdoc1_16}
    {}^{(1)}\mathscr{R}^{n+\sigma}=\Big(\partial_t^\alpha - {}^{L2\mbox{-}1_\sigma}\mathcal{D}_N^\alpha\Big)\mathcal{U}(\cdot,t_{n+\sigma}),~~{}^{(2)}\mathscr{R}^{n+\sigma}=\displaystyle{\int_{0}^{t_{n+\sigma}}}\mathcal{K}(\cdot,t_{n+\sigma}-\xi)\mathcal{U}(\cdot,\xi)\;d\xi-\mathscr{J}_N \mathcal{U}(\cdot,t_{n+\sigma}).
\end{equation}
Then, the temporal semi-discretization is obtained for $n=0,1,\ldots, N-1$, as:
\begin{equation}\label{Pdoc1_62}
    \left\{
    \begin{array}{ll}
    {}^{L2\mbox{-}1_\sigma}\mathcal{D}_N^\alpha\mathcal{U}(\cdot,t_{n+\sigma})+\mathcal{L}\mathcal{U}(\cdot,t_{n+\sigma})+\mu\mathscr{J}_N\mathcal{U}(\cdot,t_{n+\sigma})=f(\cdot,t_{n+\sigma}),\\[4pt]
    \mathcal{U}(\cdot,t_{n+1})\rvert_{\partial\Omega}=h(\cdot,t_{n+1}).
    \end{array}\right.
\end{equation}

\subsection{Approximation of the spatial derivatives for $\Omega\subset\mathbb{R}$}
In this section, we solve the semi-discrete problem (\ref{Pdoc1_62}) when $\Omega\subset\mathbb{R}$. Without loss of generality, let us take $\Omega=(0,L)$. The initial and boundary conditions are then given by 
\[\mathcal{U}(x,0)=g(x)~\mbox{for}~x\in [0,L],~~\mathcal{U}(0,t)=h_1(t)~\mbox{and}~\mathcal{U}(L,t)=h_2(t)~\mbox{for}~t\in\overline{\Omega}_t.\]
Further, note that in this case, we have $\mathcal{L} \mathcal{U}(x,t_{n+\sigma}):=-p(x)\mathcal{U}_{xx}(x,t_{n+\sigma})+q(x)\mathcal{U}_{x}(x,t_{n+\sigma})+r(x)\mathcal{U}(x,t_{n+\sigma})$. Let us define the mesh points in spatial direction as $x_m=m\Delta x$ for $m=0,1,\ldots,M$, where $\Delta x=L/M$ for some fixed $M\in\mathbb{N}$. Then. we have the fully discrete mesh as $\varOmega^{M,N}:=\{(x_m,t_n):~m=0,1,\ldots,M;~n=0,1,\ldots,N\}$. Hence, on the discrete domain, we have used the following approximation:
\begin{align}
    \nonumber\mathcal{U}_{xx}(x_m,t_{n+\sigma})&\approx\sigma\mathcal{U}_{xx}(x_m,t_{n+1})+(1-\sigma)\mathcal{U}_{xx}(x_m,t_n)\\
    &\approx\dfrac{\sigma}{(\Delta x)^2}\Big(\mathcal{U}_{m-1}^{n+1}-2\mathcal{U}_{m}^{n+1}+\mathcal{U}_{m+1}^{n+1}\Big)+\dfrac{1-\sigma}{(\Delta x)^2}\Big(\mathcal{U}_{m-1}^{n}-2\mathcal{U}_{m}^{n}+\mathcal{U}_{m+1}^{n}\Big)=: \delta_{\Delta x}^2\mathcal{U}_m^{n+\sigma},\label{Pdoc1_3}\\[6pt]
    \nonumber\mathcal{U}_{x}(x_m,t_{n+\sigma})&\approx\sigma\mathcal{U}_{x}(x_m,t_{n+1})+(1-\sigma)\mathcal{U}_{x}(x_m,t_n)\\
    &\approx\dfrac{\sigma}{2\Delta x}\Big(\mathcal{U}_{m+1}^{n+1}-\mathcal{U}_{m-1}^{n+1}\Big)+\dfrac{1-\sigma}{2\Delta x}\Big(\mathcal{U}_{m+1}^{n}-\mathcal{U}_{m-1}^{n}\Big)=: D_{\Delta x}^0\mathcal{U}_m^{n+\sigma},\label{Pdoc1_21}\\[6pt]
    \mathcal{U}(x_m,t_{n+\sigma})&\approx\sigma\mathcal{U}_m^{n+1}+(1-\sigma)\mathcal{U}_m^n.\label{Pdoc1_54}
\end{align}
Based on this approximation, equation (\ref{Pdoc1_5}) yields
\begin{equation}\label{Pdoc1_55}
    \left\{
    \begin{array}{ll}
    {}^{L2\mbox{-}1_\sigma}\mathcal{D}_N^\alpha\mathcal{U}(x_m,t_{n+\sigma})-p(x_m)\delta_{\Delta x}^2\mathcal{U}(x_m,t_{n+\sigma})+q(x_m)D_{\Delta x}^0\mathcal{U}(x_m,t_{n+\sigma})\\[4pt]
    \hspace{4cm}+r(x_m)\mathcal{U}(x_m,t_{n+\sigma})+\mu\mathscr{J}_N\mathcal{U}(x_m,t_{n+\sigma})=f(x_m,t_{n+\sigma})+\mathscr{R}_m^{n+\sigma},\\[4pt]
    \mbox{for}~m=1,2,\ldots,M-1;~n=0,1,\ldots,N-1,\\[4pt]
    \mathcal{U}(x_m,t_0)=g(x_m)~~\mbox{for}~m=0,1,\ldots,M,\\
    \mathcal{U}(x_0,t_{n+1})=h_1(t_{n+1})~\mbox{and}~\mathcal{U}(x_M,t_{n+1})=h_2(t_{n+1})~~\mbox{for}~n=0,1,\ldots,N-1.
    \end{array}\right.
\end{equation}
The remainder term $\mathscr{R}_m^{n+\sigma}$ is given by
\begin{equation}\label{Pdoc1_56}
    \mathscr{R}_m^{n+\sigma}={}^{(1)}\mathscr{R}_m^{n+\sigma}+{}^{(2)}\mathscr{R}_m^{n+\sigma}+{}^{(3)}\mathscr{R}_m^{n+\sigma}+{}^{(4)}\mathscr{R}_m^{n+\sigma},
\end{equation}
where ${}^{(1)}\mathscr{R}_m^{n+\sigma},~{}^{(2)}\mathscr{R}_m^{n+\sigma},~{}^{(3)}\mathscr{R}_m^{n+\sigma}$ and ${}^{(4)}\mathscr{R}_m^{n+\sigma}$ are occurred due to the approximation to the fractional operator, integral operator, and the second and first-order spatial derivatives, respectively. Hence, we have the fully discrete scheme for the one-dimensional case of (\ref{Pdoc1_1}) as:
\begin{equation}\label{Pdoc1_12}
    \left\{
    \begin{array}{ll}
        {}^{L2\mbox{-}1_\sigma}\mathcal{D}_N^\alpha\mathcal{U}_m^{n+\sigma}-p(x_m)\delta_{\Delta x}^2~\mathcal{U}_m^{n+\sigma}+q(x_m)D_{\Delta x}^0~\mathcal{U}_m^{n+\sigma}+r(x_m)\mathcal{U}_m^{n+\sigma}+\mu\mathscr{J}_N\mathcal{U}_m^{n+\sigma}=f(x_m,t_{n+\sigma}), \\[4pt]
	\mbox{for }m=1,2,\ldots,M-1;~n=0,1,\ldots,N-1,\\[4pt]
	\mathcal{U}_m^0=g(x_m)~~\mbox{for}~ m=0,1,\ldots,M,\\
	\mathcal{U}_0^{n+1}=h_1(t_{n+1})~\mbox{and}~\mathcal{U}_M^{n+1}=h_2(t_{n+1})~~\mbox{for}~ n=0,1,\ldots,N-1.
    \end{array}\right.
\end{equation}
A simple calculation yields the following implicit scheme:
\begin{equation}\label{Pdoc1_6}
    \left\{
    \begin{array}{ll}
	\mathscr{A}_{m-1}\mathcal{U}_{m-1}^{n+1}+\mathscr{B}_{m}\mathcal{U}_{m}^{n+1}+\mathscr{C}_{m+1}\mathcal{U}_{m+1}^{n+1}=\tilde{\mathscr{A}}_{m-1}\mathcal{U}_{m-1}^{n}+\tilde{\mathscr{B}}_{m}\mathcal{U}_{m}^{n}+\tilde{\mathscr{C}}_{m+1}\mathcal{U}_{m+1}^{n}+\mathscr{F}_m^n,\\[4pt]
	\mbox{for }m=1,2,\ldots,M-1;~n=0,1,\ldots,N-1,\\[4pt]
	\mathcal{U}_m^0=g(x_m)\;\; \mbox{for } m=0,1,\ldots,M,\\
	\mathcal{U}_0^{n+1}=h_1(t_{n+1})~\mbox{and}~\mathcal{U}_M^{n+1}=h_2(t_{n+1})~~\mbox{for}~ n=0,1,\ldots,N-1,
    \end{array}\right.
\end{equation}
where the coefficients are given by
\begin{equation*}
    \left\{
    \begin{array}{ll}
	\mathscr{A}_{m-1}=-\dfrac{q(x_m)\sigma}{2\Delta x}-\dfrac{p(x_m)\sigma}{(\Delta x)^2},\\ 
	\mathscr{B}_{m}=\Lambda_n^{n}+\dfrac{2p(x_m)\sigma}{(\Delta x)^2}+r(x_m)\sigma+\dfrac{\mu\sigma^2\tau_{n+1}}{2}\mathcal{K}(x_m,0),\\
	\mathscr{C}_{m+1}=\dfrac{q(x_m)\sigma}{2\Delta x}-\dfrac{p(x_m)\sigma}{(\Delta x)^2},\\
        \tilde{\mathscr{A}}_{m-1}=\dfrac{q(x_m)(1-\sigma)}{2\Delta x}+\dfrac{p(x_m)(1-\sigma)}{(\Delta x)^2},\\
	\tilde{\mathscr{B}}_{m}=-\dfrac{2p(x_m)(1-\sigma)}{(\Delta x)^2}-r(x_m)(1-\sigma)-\dfrac{\mu\tau_n}{2}\mathcal{K}(x_m,t_{n+\sigma}-t_n)-\dfrac{\mu\sigma\tau_{n+1}}{2}\mathcal{K}(x_m,t_{n+\sigma}-t_n)\\
        \hspace{1.1cm}-\dfrac{\mu\tau_{n+1}\sigma(1-\sigma)}{2}\mathcal{K}(x_m,0),\\
        \tilde{\mathscr{C}}_{m+1}=-\dfrac{q(x_m)(1-\sigma)}{2\Delta x}+\dfrac{p(x_m)(1-\sigma)}{(\Delta x)^2},
    \end{array}\right.
\end{equation*}
for $m=1,2,\ldots,M-1$; $n=0,1,\ldots,N-1$. Further, $\mathscr{F}_m^0=\Lambda_0^0\mathcal{U}^0+f(x_m,t_\sigma)$, $\mathscr{F}_m^1=\Lambda_1^0\mathcal{U}^0+(\Lambda_1^{1}-\Lambda_1^{0})\mathcal{U}^{1}-\dfrac{\mu\tau_1}{2}\mathcal{K}(x_m,t_{1+\sigma}-t_0)\mathcal{U}_m^0+f(x_m,t_{1+\sigma})$, and for $n=2,3,\ldots,N-1$, we have
\begin{align}
    \nonumber \mathscr{F}_m^n=&\Lambda_n^0\mathcal{U}_m^0+\sum_{j=1}^n (\Lambda_n^{j}-\Lambda_n^{j-1})\mathcal{U}_m^{j}-\dfrac{\mu\tau_n}{2}\mathcal{K}(x_m,t_{n+\sigma}-t_{n-1})\mathcal{U}_m^{n-1}\\
    &-\mu\sum_{j=0}^{n-2}\dfrac{\tau_{n+1}}{2}\Big[\mathcal{K}(x_m,t_{n+\sigma}-t_{j+1})\mathcal{U}_m^{j+1}+\mathcal{K}(x_m,t_{n+\sigma}-t_{j})\mathcal{U}_m^{j}\Big]+f(x_m,t_{n+\sigma}).\label{Pdoc1_20}
\end{align}
Notice that, at each time level $n$ for $n=0,1,\ldots,N-1$, the unknowns $\mathcal{U}_1^{n+1},\mathcal{U}_2^{n+1},\ldots,\mathcal{U}_{M-1}^{n+1}$ can be solved from the implicit scheme (\ref{Pdoc1_6}), which can be written as:
\[\mathbf{H}\boldsymbol{\mathcal{U}}^{n+1}=\mathbf{\tilde{H}}\boldsymbol{\mathcal{U}}^{n}+\pmb{\mathscr{F}}^n,\]
where $\boldsymbol{\mathcal{U}}^{n+1}=(\mathcal{U}_1^{n+1},\mathcal{U}_2^{n+1},\cdots,\mathcal{U}_{M-1}^{n+1})^T$, $\boldsymbol{\mathcal{U}}^{n}=(\mathcal{U}_1^{n},\mathcal{U}_2^{n},\cdots,\mathcal{U}_{M-1}^{n})^T$, $\pmb{\mathscr{F}}^n=(\mathscr{F}_1^n,\mathscr{F}_2^n,\cdots,\mathscr{F}_{M-1}^n)^T$, and the coefficient matrices $\mathbf{H},\mathbf{\tilde{H}}$ are defined as:
\begin{equation*}
    \mathbf{H}= \begin{pmatrix} 
        \mathscr{B}_1 & \mathscr{C}_2 &  &  & \\
        \mathscr{A}_1 & \mathscr{B}_2 & \mathscr{C}_3 &  &  \\
        & \ddots & \ddots & \ddots &\\
        &  & \mathscr{A}_{M-3} & \mathscr{B}_{M-2} & \mathscr{C}_{M-1}\\
        &  &  & \mathscr{A}_{M-2} & \mathscr{B}_{M-1}
    \end{pmatrix},~
    \mathbf{\tilde{H}}= \begin{pmatrix} 
	\tilde{\mathscr{B}}_1 & \tilde{\mathscr{C}}_2 &  &  & \\
	\tilde{\mathscr{A}}_1 & \tilde{\mathscr{B}}_2 & \tilde{\mathscr{C}}_3 &  &  \\
	& \ddots & \ddots & \ddots &\\
	&  & \tilde{\mathscr{A}}_{M-3} & \tilde{\mathscr{B}}_{M-2} & \tilde{\mathscr{C}}_{M-1}\\
	&  &  & \tilde{\mathscr{A}}_{M-2} & \tilde{\mathscr{B}}_{M-1}
    \end{pmatrix}.
\end{equation*}
The invertibility of the coefficient matrix $\mathbf{H}$ is ensured by the assumption that the coefficient function $p(x)$ satisfies the positivity condition $p(x)\geq p_0>0$. Additionally, the lower bound on the number of spatial partitions, denoted by $M$, is governed by the following nonrestrictive condition:
\begin{equation}\label{Nonrestrictive_assumption}
	\dfrac{L\|q\|_\infty}{2p_0}<M.
\end{equation} 
This condition is necessary to guarantee the invertibility of $\mathbf{H}$ in the proposed framework. The computational procedure for solving TFIPDEs (\ref{Pdoc1_1}) when $\Omega\subset\mathbb{R}$ is described in Algorithm \ref{Pdoc1_Alg1}.
 
\begin{algorithm}[ht]
{\small
\caption{ALGORITHM FOR THE SOLUTION OF (\ref{Pdoc1_1}) WHEN $\Omega\subset\mathbb{R}$}\label{Pdoc1_Alg1}
\begin{algorithmic}[1]
    \State \textbf{Input 1:} $\alpha$ \Comment{Order of the fractional operator}
    \State \textbf{Input 2:} $N$, $M$ \Comment{Number of mesh interval towards time and space}
    \State \textbf{Input 3:} $g,~h_1,~h_2$ \Comment{Initial and boundary conditions}
    \State Compute $\sigma=1-(\alpha/2)$ and the grading parameter $\nu$
    \State Compute $\Delta x=L/M$ and $x_m=m*\Delta x$ for $m=0,1,\ldots,M$
    \State Compute $t_n=T*(n/N)^\nu$ for $n=0,1,\ldots,N$
    \State Compute $p,q,r,f$, and $\mathcal{K}$  
    \State Compute $\Lambda_0^0$ and $\Lambda_n^{j}$ for $n=1,2,\ldots,N-1$, $j=0,1,\ldots,n$ given in (\ref{Pdoc1_2})
    \For{$n=0,1,\ldots,N-1$}
        \State Compute $\mathbf{H}$ and $\mathbf{\tilde{H}}$ 
        \State Compute the vector $\pmb{\mathscr{F}}^n$ using the formula given in (\ref{Pdoc1_20})
        \State Compute $\boldsymbol{\mathcal{U}}^{n+1}=\mathbf{H}^{-1}(\mathbf{\tilde{H}}\boldsymbol{\mathcal{U}}^{n}+\pmb{\mathscr{F}}^n)$
    \EndFor
    \State \textbf{Output:} Obtained discrete solution $\mathcal{U}_m^n$ for $m=1,2,\ldots,M-1,~n=1,2,\ldots,N$
\end{algorithmic}}
\end{algorithm}

\subsection{Approximation of the spatial derivatives for $\Omega\subset\mathbb{R}^2$}
In this section, we discuss the approximation scheme for the given integro-PDEs (\ref{Pdoc1_1}) when $\Omega\subset\mathbb{R}^2$. The multi-dimensional Haar wavelet decomposition is used to approximate the spatial derivatives. In this case, let us define the initial and boundary conditions in explicit form as:
\begin{align*}
    \mathcal{U}(x,y,0)&=g(x,y)~\mbox{for}~(x,y)\in\overline{\Omega},\\
    \mathcal{U}(0,y,t)&=h_1(y,t),~\mathcal{U}(1,y,t)=h_2(y,t)~\mbox{for}~(y,t)\in\overline{\Omega}_y\times\overline{\Omega}_t,\\
    \mathcal{U}(x,0,t)&=h_3(x,t),~\mathcal{U}(x,1,t)=h_4(x,t)~\mbox{for}~(x,t)\in\overline{\Omega}_x\times\overline{\Omega}_t,
\end{align*}
where $\Omega=\Omega_x\times\Omega_y$ with $\Omega_x=\Omega_y=(0,1)$. Using (\ref{Pdoc1_2}) and (\ref{Pdoc1_4}) into (\ref{Pdoc1_62}), one has the semi-discrete form of the proposed problem (\ref{Pdoc1_1}) as:
\begin{equation}\label{Pdoc1_25}
    \left\{
    \begin{array}{ll}
    \Big(\Lambda_n^{n}+\dfrac{\mu\sigma^2\tau_{n+1}}{2}\mathcal{K}(x,y,0)\Big)\mathcal{U}^{n+1}(x,y)+\mathcal{L}\mathcal{U}^{n+\sigma}(x,y)=F^n(x,y),\\[6pt]
    \mathcal{U}^{n+\sigma}(0,y)=h_1(y,t_{n+\sigma}),~\mathcal{U}^{n+\sigma}(1,y)=h_2(y,t_{n+\sigma}),\\
    \mathcal{U}^{n+\sigma}(x,0)=h_3(x,t_{n+\sigma}),~\mathcal{U}^{n+\sigma}(x,1)=h_4(x,t_{n+\sigma}),
    \end{array}\right.
\end{equation}
for $n=0,1,\ldots, N-1 $, where $\mathcal{L}\mathcal{U}^{n+\sigma}(x,y)$ is given by:
\begin{align*}
    \mathcal{L} \mathcal{U}^{n+\sigma}(x,y):=-p(x,y)\Delta \mathcal{U}^{n+\sigma}(x,y)+\boldsymbol{q}(x,y)\cdot\nabla \mathcal{U}^{n+\sigma}(x,y)+r(x,y)\mathcal{U}^{n+\sigma}(x,y),
\end{align*}
and $F^0(x,y)=\Lambda_n^0\mathcal{U}^{0}(x,y)=\Lambda_n^0g(x,y)$, and for $n=1,2,\ldots, N-1$, $F^n(x,y)$ can be expressed as:
\begin{align*}
     F^n(x,y)=&\Lambda_n^0g(x,y)+\sum_{j=1}^n (\Lambda_n^{j}-\Lambda_n^{j-1})\mathcal{U}^{j}(x,y)-\mu\Big(\dfrac{\sigma\tau_{n+1}}{2}\mathcal{K}(x,y,t_{n+\sigma}-t_n)+\dfrac{\sigma(1-\sigma)\tau_{n+1}}{2}\mathcal{K}(x,y,0)\Big)\mathcal{U}^n(x,y)\\
     &-\mu\sum_{j=0}^{n-1}\dfrac{\tau_{j+1}}{2}\big[\mathcal{K}(x,y,t_{n+\sigma}-t_{j+1})\mathcal{U}^{j+1}(x,y)+\mathcal{K}(x,y,t_{n+\sigma}-t_j)\mathcal{U}^j(x,y)\big]+f(x,y,t_{n+\sigma}).
\end{align*}
It is important to note that the discretization given in (\ref{Pdoc1_25}) results in a truncation error of the form:
\begin{equation}\label{Pdoc1_61}
    \widehat{\mathscr{R}}^{n+\sigma}={}^{(1)}\mathscr{R}^{n+\sigma}+{}^{(2)}\mathscr{R}^{n+\sigma},
\end{equation}
where ${}^{(1)}\mathscr{R}^{n+\sigma}$ and ${}^{(2)}\mathscr{R}^{n+\sigma}$ are defined in (\ref{Pdoc1_16}).
\subsubsection{Haar wavelets and its properties}\label{subsec_Haar}
The Haar wavelet \cite{HaarZur1910} consists of a sequence of square-shaped functions that together form a wavelet family. The mother wavelet function $\psi(x)$, and it's scaling function $\phi(x)$ are defined as:
\begin{equation*}
    \begin{array}{ll}
         \psi(x)=\begin{cases}
                 1,~&\mbox{for }t\in\Big[0,\dfrac{1}{2}\Big),\\[4pt]
                 -1,~&\mbox{for }t\in\Big[\dfrac{1}{2},1\Big),\\[4pt]
                 0,~&\mbox{elsewhere},
                 \end{cases}
        ~~\mbox{and}~~
        \phi(x)=\begin{cases}
                 1,~&\mbox{for }t\in[0,1),\\[4pt]
                 0,~&\mbox{elsewhere}.
                 \end{cases}
    \end{array}
\end{equation*}
The corresponding basis functions can be generated using the dilation parameter $j$ and translation parameter $k$ as $\{\phi_k^j(x)=2^{j/2}\phi(2^jx-k)\}_{j,k\in\mathbb{Z}}$, $\{\psi_k^j(x)=2^{j/2}\psi(2^jx-k)\}_{j,k\in\mathbb{Z}}$ to form an orthonormal subfamily of the Hilbert space $L^2(\mathbb{R})$. Let us choose $i=m+k+1$, where $m=2^j$ for $j=0,1,\ldots,J$ and $k=0,1,\ldots,m-1$. Here $J$ is called the maximum level of resolution. Then, for $x\in[0,1]$, the $i^{\mbox{th}}$ Haar wavelet is defined as:
\begin{equation}\label{Pdoc1_36}
    \begin{array}{ll}
         \psi_i(x)=\begin{cases}
                 1,~&\mbox{for }x\in\Big[\zeta_1(i),\zeta_2(i)\Big),\\[4pt]
                 -1,~&\mbox{for }x\in\Big[\zeta_2(i),\zeta_3(i)\Big),\\[4pt]
                 0,~&\mbox{elsewhere},
                 \end{cases}
        ~~i=2,3,\ldots,
    \end{array}
\end{equation}
where $\zeta_1(i)=\dfrac{k}{m},~\zeta_2(i)=\dfrac{k+1/2}{m},~\zeta_3(i)=\dfrac{k+1}{m}$.  Further, the $n^{\mbox{th}}$, $n\in\mathbb{N}$ integration of the Haar wavelets can be written as:
\begin{equation}\label{Pdoc1_38}
    \mathcal{R}_{n,1}(x)=\dfrac{x^n}{n!},\text{~ for all ~}x\in[0,1),
\end{equation} 
and for $i=2,3,\ldots$,
\begin{equation}\label{Pdoc1_39}
    \begin{array}{ll}
         \mathcal{R}_{n,i}(x)=\dfrac{1}{n!}\begin{cases}
                 0,~&0\leq x<\zeta_1(i),\\
                 (x-\zeta_1(i))^n,~&\zeta_1(i)\leq x<\zeta_2(i),\\
                 (x-\zeta_1(i))^n-2(x-\zeta_2(i))^n,~&\zeta_2(i)\leq x<\zeta_3(i),\\
                 (x-\zeta_1(i))^n-2(x-\zeta_2(i))^n+(x-\zeta_3(i))^n,~&\zeta_3(i)\leq x<1.
                 \end{cases}
    \end{array}
\end{equation}
The following lemma shows the upper bounds of the integrands of the Haar wavelets that help to obtain the required error bounds described in Theorem \ref{Pdoc1_thm6}.
\begin{lemma}\label{Pdoc1_lem5}
(See Theorem 1 of \cite{WichailukkanaAconvergence2016}) The $n^{\mbox{th}}$ $(n\in\mathbb{N})$ integration of the Haar wavelets defined in (\ref{Pdoc1_38}) and (\ref{Pdoc1_39}) satisfies the following bounds.
\begin{equation*}
    \left\{
        \begin{array}{l}
             \Big\lvert\mathcal{R}_{n,1}(t)\Big\rvert\leq \dfrac{1}{n!},~n\geq1,  \\[6pt]
             \Big\lvert\mathcal{R}_{1,i}(t)\Big\rvert\leq \dfrac{1}{2^{j+1}},~i\geq2,\\[6pt]
             \Big\lvert\mathcal{R}_{n,i}(t)\Big\rvert\leq \mathscr{C}(n)\Big(\dfrac{1}{2^{j+1}}\Big)^2,~n\geq2,~i\geq2,
        \end{array}\right.
\end{equation*}
where $\mathscr{C}(n)=\dfrac{8}{3(\lfloor (n+1)/2\rfloor !)^2}$.
\end{lemma}
A real-valued function $z(x,y)\in L^2(\overline{\Omega})$ can be decomposed by using two-dimensional Haar wavelets as:
\begin{align*}
  z(x,y)= D_{1,1}\phi(x)\phi(y)+\sum_{i_2=2}^{\infty}D_{1,i_2}\phi(x)\psi_{i_2}(y)+\sum_{i_1=2}^{\infty}D_{i_1,1}\psi_{i_1}(x)\phi(y)+\sum_{i_1=2}^{\infty}\sum_{i_2=2}^{\infty}D_{i_1,i_2}\psi_{i_1}(x)\psi_{i_2}(y).
\end{align*}
Its numerical approximation is given by
\begin{align}
  \nonumber z(x,y)\approx z_{M_1M_2}(x,y):= & D_{1,1}\phi(x)\phi(y)+\sum_{i_2=2}^{2M_2}D_{1,i_2}\phi(x)\psi_{i_2}(y)+\sum_{i_1=2}^{2M_1}D_{i_1,1}\psi_{i_1}(x)\phi(y)+\sum_{i_1=2}^{2M_1}\sum_{i_2=2}^{2M_2}D_{i_1,i_2}\psi_{i_1}(x)\psi_{i_2}(y)\\
  &=\boldsymbol{\mathcal{H}}(x)^T\mathbf{D}\boldsymbol{\mathcal{H}}(y),\label{Pdoc1_42}
\end{align}
where $\boldsymbol{\mathcal{H}}(x)=(\phi(x),\psi_2(x),\cdots,\psi_{2M_1}(x))^T$, $\boldsymbol{\mathcal{H}}(y)=(\phi(y),\psi_2(y),\cdots,\psi_{2M_2}(y))^T$, and $\mathbf{D}$ is the 
matrix of unknown coefficients of order $(2M_1\times 2M_2)$, given by 
\begin{equation}\label{Pdoc1_40}
    \mathbf{D}=
    \begin{pmatrix} 
	  D_{1,1} & D_{1,2} & \cdots   &D_{1,2M_2} \\
		D_{2,1} & D_{2,2} & \cdots   &D_{2,2M_2}  \\
		\vdots& \vdots & \ddots &\vdots\\
		D_{2M_1,1} & D_{2M_1,2} & \cdots   &D_{2M_1,2M_2}
	\end{pmatrix}.
\end{equation}
Here, we take
$i_1=m_1+k_1+1$, with $m_1=2^{j_1}$ for $j_1=0,1,\ldots,J_1$ and $k_1=0,1,\ldots,m_1-1$. 
$i_2=m_2+k_2+1$, with $m_2=2^{j_2}$ for $j_2=0,1,\ldots,J_2$ and $k_2=0,1,\ldots,m_2-1$.
$J_1$ and $J_2$ represent the maximum level of resolution in the direction of $x$ and $y$, respectively. 
Further, $M_1=2^{J_1},~M_2=2^{J_2}$. 
In order to find the unknown matrix, we use the two-dimensional collocation points 
$\{(x_{l_1},y_{l_2})\}_{l_1=1,l_2=1}^{2M_1,2M_2}$ defined as:
\begin{equation}\label{Pdoc1_33}
    x_{l_1}=\dfrac{l_1-1/2}{2M_1},~y_{l_2}=\dfrac{l_2-1/2}{2M_2},~l_1=1,2,\ldots,2M_1,~l_2=1,2,\ldots,2M_2.
\end{equation}

\subsubsection{Approximation of the spatial derivatives based on two-dimensional Haar wavelets}
Let us assume $\mathscr{Q}(x,y)=\mathcal{U}_{xxyy}^{n+\sigma}(x,y)\in L^2(\overline{\Omega})$. Then, it can be decomposed as a finite sum of the two-dimensional Haar wavelets as:
\begin{align}
  \mathcal{U}_{xxyy}^{n+\sigma}(x,y)\approx \boldsymbol{\mathcal{H}}(x)^T\mathbf{D}\boldsymbol{\mathcal{H}}(y).\label{Pdoc1_26}
\end{align}
Integrating \eqref{Pdoc1_26} with respect to $y$ twice and using the boundary conditions given in \eqref{Pdoc1_25}, we obtain:
\begin{align}
  \mathcal{U}_{xx}^{n+\sigma}(x,y)\approx \boldsymbol{\mathcal{H}}(x)^T\mathbf{D}(\boldsymbol{\mathcal{R}}_2(y)-y\boldsymbol{\mathcal{R}}_2(1))+(1-y)\dfrac{\partial^2h_3}{\partial x^2}(x,t_{n+\sigma})+y\dfrac{\partial^2h_4}{\partial x^2}(x,t_{n+\sigma}),\label{Pdoc1_27}
\end{align}
where $\boldsymbol{\mathcal{R}}_2(y):=(\mathcal{R}_{2,1}(y),\mathcal{R}_{2,2}(y),\cdots,\mathcal{R}_{2,2M_2}(y))^T$. Again, integrating \eqref{Pdoc1_26} with respect to $x$ twice and then, using the boundary conditions, we yield
\begin{align}
  \mathcal{U}_{yy}^{n+\sigma}(x,y)\approx (\boldsymbol{\mathcal{R}}_2(x)-x\boldsymbol{\mathcal{R}}_2(1))^T\mathbf{D}\boldsymbol{\mathcal{H}}(y)+(1-x)\dfrac{\partial^2h_1}{\partial y^2}(y,t_{n+\sigma})+x\dfrac{\partial^2h_2}{\partial y^2}(y,t_{n+\sigma}),\label{Pdoc1_28}
\end{align}
with $\boldsymbol{\mathcal{R}}_2(x):=(\mathcal{R}_{2,1}(x),\mathcal{R}_{2,2}(x),\cdots,\mathcal{R}_{2,2M_1}(x))^T$. Further, the integration of \eqref{Pdoc1_27} with respect to $x$ twice, and the integration of \eqref{Pdoc1_28} with respect to $y$ once yields the following:
\begin{align}
  \nonumber \mathcal{U}_{x}^{n+\sigma}(x,y)\approx& (\boldsymbol{\mathcal{R}}_1(x)-\boldsymbol{\mathcal{R}}_2(1))^T\mathbf{D}(\boldsymbol{\mathcal{R}}_2(y)-y\boldsymbol{\mathcal{R}}_2(1))-h_1(y,t_{n+\sigma})+h_2(y,t_{n+\sigma})\\
  \nonumber &+(1-y)\Big(\dfrac{\partial h_3}{\partial x}(x,t_{n+\sigma})-h_3(1,t_{n+\sigma})+h_3(0,t_{n+\sigma})\Big)\\
  &+y\Big(\dfrac{\partial h_4}{\partial x}(x,t_{n+\sigma})-h_4(1,t_{n+\sigma})+h_4(0,t_{n+\sigma})\Big),\label{Pdoc1_29}\\[6pt]
  \nonumber \mathcal{U}_{y}^{n+\sigma}(x,y)\approx& (\boldsymbol{\mathcal{R}}_2(x)-x\boldsymbol{\mathcal{R}}_2(1))^T\mathbf{D}(\boldsymbol{\mathcal{R}}_1(y)-\boldsymbol{\mathcal{R}}_2(1))-h_3(x,t_{n+\sigma})+h_4(x,t_{n+\sigma})\\
  \nonumber &+(1-x)\Big(\dfrac{\partial h_1}{\partial y}(y,t_{n+\sigma})-h_1(1,t_{n+\sigma})+h_1(0,t_{n+\sigma})\Big)\\
  &+x\Big(\dfrac{\partial h_2}{\partial y}(y,t_{n+\sigma})-h_2(1,t_{n+\sigma})+h_2(0,t_{n+\sigma})\Big),\label{Pdoc1_30}\\[6pt]
  \nonumber \mathcal{U}^{n+\sigma}(x,y)\approx& (\boldsymbol{\mathcal{R}}_2(x)-x\boldsymbol{\mathcal{R}}_2(1))^T\mathbf{D}(\boldsymbol{\mathcal{R}}_2(y)-y\boldsymbol{\mathcal{R}}_2(1))+(1-x)h_1(y,t_{n+\sigma})+xh_2(y,t_{n+\sigma})\\
  \nonumber &+(1-y)\big(h_3(x,t_{n+\sigma})-h_3(0,t_{n+\sigma})-xh_3(1,t_{n+\sigma})+xh_3(0,t_{n+\sigma})\big)\\
  &+y\big(h_4(x,t_{n+\sigma})-h_4(0,t_{n+\sigma})-xh_4(1,t_{n+\sigma})+xh_4(0,t_{n+\sigma})\big),\label{Pdoc1_31}
\end{align}
where 
$\boldsymbol{\mathcal{R}}_1(x):=(\mathcal{R}_{1,1}(x),\mathcal{R}_{1,2}(x),\cdots,\mathcal{R}_{1,2M_1}(x))^T,~\boldsymbol{\mathcal{R}}_1(y):=(\mathcal{R}_{1,1}(y),\mathcal{R}_{1,2}(y),\cdots,\mathcal{R}_{1,2M_2}(y))^T$. Equation (\ref{Pdoc1_31}) is further approximated as 
\begin{equation}\label{Pdoc1_57}
    \mathcal{U}^{n+\sigma}(x,y)\approx \sigma\mathcal{U}^{n+1}(x,y)+(1-\sigma)\mathcal{U}^n(x,y).
\end{equation}
Hence, we have
\begin{align}
    \nonumber \mathcal{U}^{n+1}(x,y)\approx\mathcal{U}_{M_1M_2}^{n+1}(x,y)=& \dfrac{1}{\sigma}\Big[(\boldsymbol{\mathcal{R}}_2(x)-x\boldsymbol{\mathcal{R}}_2(1))^T\mathbf{D}(\boldsymbol{\mathcal{R}}_2(y)-y\boldsymbol{\mathcal{R}}_2(1))+(1-x)h_1(y,t_{n+\sigma})+xh_2(y,t_{n+\sigma})\\
  \nonumber &+(1-y)\big(h_3(x,t_{n+\sigma})-h_3(0,t_{n+\sigma})-xh_3(1,t_{n+\sigma})+xh_3(0,t_{n+\sigma})\big)\\
  &+y\big(h_4(x,t_{n+\sigma})-h_4(0,t_{n+\sigma})-xh_4(1,t_{n+\sigma})+xh_4(0,t_{n+\sigma})\big)\Big]-\Big(\dfrac{1-\sigma}{\sigma}\Big)\mathcal{U}^n(x,y).\label{Pdoc1_32}
\end{align}
Substituting \eqref{Pdoc1_27}-\eqref{Pdoc1_32} into \eqref{Pdoc1_25}, one can have a linear system at each time level and that can be solved by using the collocation points defined in \eqref{Pdoc1_33}. Then, after substituting the wavelet coefficients into \eqref{Pdoc1_32}, we obtain the desired solution of (\ref{Pdoc1_1}) at each time level. The algorithm for numerical computation of the proposed wavelet-based finite difference approximation is described in Algorithm \ref{Pdoc1_Alg2}.

\subsubsection{Convergence of the wavelet approximation}
In this segment, our objective is to show the convergence of the wavelet approximation and to derive an estimate for the error bound based on the $L^2$-norm. The orthogonality property presented below will be taken into account in establishing this estimate.
 \begin{equation}\label{Pdoc1_41}
     \int_0^1\psi_{i_1}(x)\psi_{i_1'}(x)~dx=\begin{cases}
                                            \dfrac{1}{2^{j_1}},~&\mbox{if }i_1=i_1',\\
                                            0,~&\mbox{if }i_1\neq i_1'
                                            \end{cases},~
    \int_0^1\psi_{i_2}(y)\psi_{i_2'}(y)~dy=\begin{cases}
                                            \dfrac{1}{2^{j_2}},~&\mbox{if }i_2=i_2',\\
                                            0,~&\mbox{if }i_2\neq i_2'
                                            \end{cases}.
 \end{equation}
The unknown coefficients in the wavelet expansion of $\mathscr{Q}(x,y)$, as expressed in (\ref{Pdoc1_26}), can be determined by using the orthogonality properties given in (\ref{Pdoc1_41}) as: 
\[D_{i_1,i_2}:=\displaystyle{\int_0^1\int_0^1}\mathscr{Q}(x,y)\psi_{i_1}(x)\psi_{i_2}(y)dxdy\]
with $\psi_{i_1}(x)=\phi(x),~\psi_{i_2}(y)=\phi(y)$ when $i_1=1,~i_2=1$, respectively. The following lemma demonstrates the bounds for the wavelet coefficients that will be used later in Theorem \ref{Pdoc1_thm6} to obtain the error bounds for the wavelet-based numerical solutions for two-dimensional TFIPDEs having time singularity.  
\begin{lemma}\label{Pdoc1_lem6}
Let $\mathscr{Q}(x,y)\in L^2(\overline{\Omega})$ be continuous such that $\lvert \mathscr{Q}\rvert,~\lvert \mathscr{Q}_x\rvert,~\lvert \mathscr{Q}_y\rvert,~\lvert \mathscr{Q}_{xy}\rvert\leq\mathscr{M}$, for all $(x,y)\in\overline{\Omega}$, where $\mathscr{M}$ is some positive constant. Then, the wavelet coefficients $\{D_{i_1,i_2}\}_{i_1=1,i_2=1}^{\infty,\infty}$ satisfy the following bounds: 
\begin{equation*}
    \left\{
    \begin{array}{ll}
        \Big\lvert D_{1,1}\Big\rvert\leq \mathscr{M},  &\\[4pt]
        \Big\lvert D_{1,i_2}\Big\rvert\leq \dfrac{\mathscr{M}}{2^{j_2+1}},~&\mbox{for}~i_2\geq 2,\\[4pt]
        \Big\lvert D_{i_1,1}\Big\rvert\leq \dfrac{\mathscr{M}}{2^{j_1+1}},~&\mbox{for}~i_1\geq 2,\\[4pt]
        \Big\lvert D_{i_1,i_2}\Big\rvert\leq \dfrac{\mathscr{M}}{(2^{j_1+1})(2^{j_2+1})},~&\mbox{for}~i_1\geq 2,~i_2\geq 2,
    \end{array}\right.
\end{equation*}
\end{lemma}
\begin{proof}
Applying the mean value theorem for integrals, we have
\begin{align*}
    \nonumber \Big\lvert D_{1,1}\Big\rvert=&\Big\lvert\int_0^1\int_0^1\mathscr{Q}(x,y)\phi(x)\phi(y)dxdy\Big\rvert=\Big\lvert\int_0^1\int_0^1\mathscr{Q}(x,y)dxdy\Big\rvert=\Big\lvert\int_0^1\mathscr{Q}(\rho_1,y)dy\Big\rvert\\
    &=\Big\lvert \mathscr{Q}(\rho_1,\rho_2)\int_0^1dy\Big\rvert=\Big\lvert \mathscr{Q}(\rho_1,\rho_2)\Big\rvert\leq\mathscr{M},~\mbox{for some}~\rho_1,~\rho_2\in(0,1).
\end{align*}
For $i_2\geq2$, following (\ref{Pdoc1_36}), and the mean value theorem for integrals yields 
\begin{align*}
    \nonumber D_{1,i_2}=&\int_0^1\int_0^1\mathscr{Q}(x,y)\phi(x)\psi_{i_2}(y)dxdy=\int_0^1\int_0^1\mathscr{Q}(x,y)\psi_{i_2}(y)dxdy=\int_0^1\mathscr{Q}(\rho_3,y)\psi_{i_2}(y)dy\\
    \nonumber=&\int_{k_2/2^{j_2}}^{(k_2+\frac{1}{2})/2^{j_2}}\mathscr{Q}(\rho_3,y)~dy-\int_{(k_2+\frac{1}{2})/2^{j_2}}^{(k_2+1)/2^{j_2}}\mathscr{Q}(\rho_3,y)~dy=\dfrac{1}{2^{j_2+1}}\big(\mathscr{Q}(\rho_3,\rho_4)-\mathscr{Q}(\rho_3,\rho_5)\big)\\
    \nonumber=&\dfrac{1}{2^{j_2+1}}(\rho_4-\rho_5)\dfrac{\partial \mathscr{Q}}{\partial y}\Big\rvert_{x=\rho_3,y=\rho_6},
\end{align*}
for some $\rho_3\in(0,1)$, $\rho_4\in\Big(\dfrac{k_2}{2^{j_2}},\dfrac{k_2+\frac{1}{2}}{2^{j_2}}\Big)$, $\rho_5\in\Big(\dfrac{k_2+\frac{1}{2}}{2^{j_2}},\dfrac{k_2+1}{2^{j_2}}\Big)$, and $\rho_6\in\big(\min\{\rho_4,\rho_5\},\max\{\rho_4,\rho_5\}\big)$. Hence,
\begin{equation*}
    \Big\lvert D_{1,i_2}\Big\rvert=\Big\lvert\dfrac{1}{2^{j_2+1}}(\rho_4-\rho_5)\dfrac{\partial \mathscr{Q}}{\partial y}\Big\rvert_{x=\rho_3,y=\rho_6}\Big\rvert\leq \dfrac{1}{2^{j_2+1}}\lvert\rho_4-\rho_5\rvert\Big\lvert\dfrac{\partial \mathscr{Q}}{\partial y}\Big\rvert_{x=\rho_3,y=\rho_6}\Big\rvert \leq \dfrac{\mathscr{M}}{2^{j_2+1}}.
\end{equation*}
Using similar arguments, one can prove $\Big\lvert D_{i_1,1}\Big\rvert\leq \dfrac{\mathscr{M}}{2^{j_1+1}},~\mbox{for}~i_1\geq 2$. Finally, for $i_1\geq2,~i_2\geq2$, we have
\begin{align*}
    \nonumber D_{i_1,i_2}=&\int_0^1\int_0^1\mathscr{Q}(x,y)\psi_{i_1}(x)\psi_{i_2}(y)dxdy\\[6pt]
    =&\int_0^1\Big(\int_{k_1/2^{j_1}}^{(k_1+\frac{1}{2})/2^{j_1}}\mathscr{Q}(x,y)dx-\int_{(k_1+\frac{1}{2})/2^{j_1}}^{(k_1+1)/2^{j_1}}\mathscr{Q}(x,y)dx\Big)\psi_{i_2}(y)dy\\[6pt]
    \nonumber =&\dfrac{1}{2^{j_1+1}}\int_0^1\big(\mathscr{Q}(\rho_7,y)-\mathscr{Q}(\rho_8,y)\big)\psi_{i_2}(y)dy=\dfrac{(\rho_7-\rho_8)}{2^{j_1+1}}\int_0^1\dfrac{\partial \mathscr{Q}}{\partial x}(\rho_9,y)\psi_{i_2}(y)dy\\[6pt]
    \nonumber =&\dfrac{(\rho_7-\rho_8)}{2^{j_1+1}}\Big(\int_{k_2/2^{j_2}}^{(k_2+\frac{1}{2})/2^{j_2}}\dfrac{\partial \mathscr{Q}}{\partial x}(\rho_9,y)dy-\int_{(k_2+\frac{1}{2})/2^{j_2}}^{(k_2+1)/2^{j_2}}\dfrac{\partial \mathscr{Q}}{\partial x}(\rho_9,y)dy\Big)\\[6pt]
    \nonumber =&\dfrac{(\rho_7-\rho_8)}{(2^{j_1+1})(2^{j_2+1})}\Big(\dfrac{\partial \mathscr{Q}}{\partial x}(\rho_9,\rho_{10})-\dfrac{\partial \mathscr{Q}}{\partial x}(\rho_9,\rho_{11})\Big)=\dfrac{(\rho_7-\rho_8)(\rho_{10}-\rho_{11})}{(2^{j_1+1})(2^{j_2+1})}\dfrac{\partial^2\mathscr{Q}}{\partial y\partial x}\Big\rvert_{x=\rho_9,y=\rho_{12}},
\end{align*}
for some $\rho_7\in\Big(\dfrac{k_1}{2^{j_1}},\dfrac{k_1+\frac{1}{2}}{2^{j_1}}\Big)$, $\rho_8\in\Big(\dfrac{k_1+\frac{1}{2}}{2^{j_1}},\dfrac{k_1+1}{2^{j_1}}\Big)$, $\rho_9\in\big(\min\{\rho_7,\rho_8\},\max\{\rho_7,\rho_8\}\big)$, $\rho_{10}\in\Big(\dfrac{k_2}{2^{j_2}},\dfrac{k_2+\frac{1}{2}}{2^{j_2}}\Big)$, $\rho_{11}\in\Big(\dfrac{k_2+\frac{1}{2}}{2^{j_2}},\dfrac{k_2+1}{2^{j_2}}\Big)$, and $\rho_{12}\in\big(\min\{\rho_{10},\rho_{11}\},\max\{\rho_{10},\rho_{11}\}\big)$. Taking the modulus on both sides of the above equation, we obtain:
\begin{align*}
    \Big\lvert D_{i_1,i_2}\Big\rvert&=\Big\lvert\dfrac{(\rho_7-\rho_8)(\rho_{10}-\rho_{11})}{(2^{j_1+1})(2^{j_2+1})}\dfrac{\partial^2\mathscr{Q}}{\partial y\partial x}\Big\rvert_{x=\rho_9,y=\rho_{12}}\Big\rvert\leq\dfrac{\lvert\rho_7-\rho_8\rvert\lvert\rho_{10}-\rho_{11}\rvert}{(2^{j_1+1})(2^{j_2+1})}\Big\lvert\dfrac{\partial^2\mathscr{Q}}{\partial y\partial x}\Big\rvert_{x=\rho_9,y=\rho_{12}}\Big\rvert \\[6pt]
    &\leq \dfrac{\mathscr{M}}{(2^{j_1+1})(2^{j_2+1})}.
\end{align*}
This completes the proof.
\end{proof}
At each time level $t_{n+1},~n=0,1,\ldots,N-1$, the error and its $L^2$-norm based on multi-dimensional wavelet approximation can be defined as:
\[\Big\|\mathcal{U}^{n+1}(x,y)-\mathcal{U}_{M_1M_2}^{n+1}(x,y)\Big\|_{L^2(\overline{\Omega})}=\Bigg\{\int_0^1\int_0^1\big\lvert\mathcal{U}^{n+1}(x,y)-\mathcal{U}_{M_1M_2}^{n+1}(x,y)\big\rvert^2dxdy\Bigg\}^{1/2},\]
where
\begin{align}
    \nonumber \mathcal{U}^{n+1}(x,y)-\mathcal{U}_{M_1M_2}^{n+1}(x,y)=& \dfrac{1}{\sigma}\Big[\sum_{i_2=2M_2+1}^{\infty}D_{1,i_2}(\mathcal{R}_{2,1}(x)-x\mathcal{R}_{2,1}(1))(\mathcal{R}_{2,i_2}(y)-y\mathcal{R}_{2,i_2}(1))\\
    \nonumber &~+\sum_{i_1=2M_1+1}^{\infty}D_{i_1,1}(\mathcal{R}_{2,i_1}(x)-x\mathcal{R}_{2,i_1}(1))(\mathcal{R}_{2,1}(y)-y\mathcal{R}_{2,1}(1))\\
    &~+\sum_{i_1=2M_1+1}^{\infty}\sum_{i_2=2M_2+1}^{\infty}D_{i_1,i_2}(\mathcal{R}_{2,i_1}(x)-x\mathcal{R}_{2,i_1}(1))(\mathcal{R}_{2,i_2}(y)-y\mathcal{R}_{2,i_2}(1))\Big]+\mathcal{O}(N^{-2}).\label{Pdoc1_43}
\end{align}
The remainder term $\mathcal{O}(N^{-2})$ arises from the approximation presented in (\ref{Pdoc1_57}). A rigorous proof of this will be discussed later in Section \ref{sec_Convergence}. Below, we establish an upper bound for this term in the following theorem.
\begin{theorem}\label{Pdoc1_thm6}
Assume that $\mathscr{Q}(x,y)\in L^2(\overline{\Omega})$ be continuous such that $\lvert \mathscr{Q}\rvert,~\lvert\mathscr{Q}_x\rvert,~\lvert\mathscr{Q}_y\rvert,~\lvert\mathscr{Q}_{xy}\rvert\leq\mathscr{M}$, the numerical solution (\ref{Pdoc1_32}) obtained by the two-dimensional Haar wavelets converges. Furthermore, the $L^2$-norm of the error satisfies the following bound:
\[\Big\|\mathcal{U}^{n+1}(x,y)-\mathcal{U}_{M_1M_2}^{n+1}(x,y)\Big\|_{L^2(\overline{\Omega})}\leq C[\mathcal{M}^{-3}+N^{-2}],\]
where $\mathcal{M}=\min\{M_1,M_2\}$ with $M_1=2^{J_1},~M_2=2^{J_2}$; $J_1,~J_2$ are the maximum level of resolutions towards $x$ and $y$ direction, respectively. More precisely, $\Big\lvert\mathcal{U}^{n+1}(x,y)-\mathcal{U}_{M_1M_2}^{n+1}(x,y)\Big\rvert\rightarrow0$ as $M_1\rightarrow\infty,~M_2\rightarrow\infty$. 
\end{theorem}
\begin{proof}
\begin{align}
\Big\|\mathcal{U}^{n+1}(x,y)-\mathcal{U}_{M_1M_2}^{n+1}(x,y)\Big\|_{L^2}^2&=\int_0^1\int_0^1\big\lvert\mathcal{U}^{n+1}(x,y)-\mathcal{U}_{M_1M_2}^{n+1}(x,y)\big\rvert^2dxdy\leq\dfrac{1}{\sigma^2}\sum_{k=1}^6\mathcal{I}_k+\mathcal{O}(N^{-4}),\label{Pdoc1_45}
\end{align}
where
{\small
\begin{align*}
\mathcal{I}_1=&\sum_{i_2=2M_2+1}^{\infty}\sum_{i_2'=2M_2+1}^{\infty}\big\lvert D_{1,i_2}\big\rvert\big\lvert D_{1,i_2'}\big\rvert\int_0^1\big\lvert\mathcal{R}_{2,1}(x)-x\mathcal{R}_{2,1}(1)\big\lvert^2dx\int_0^1\big\lvert\mathcal{R}_{2,i_2}(y)-y\mathcal{R}_{2,i_2}(1)\big\rvert\big\lvert\mathcal{R}_{2,i_2'}(y)-y\mathcal{R}_{2,i_2'}(1)\big\rvert dy,\\
\mathcal{I}_2=&\sum_{i_1=2M_1+1}^{\infty}\sum_{i_1'=2M_1+1}^{\infty}\big\lvert D_{i_1,1}\big\rvert\big\lvert D_{i_1',1}\big\rvert\int_0^1\big\lvert\mathcal{R}_{2,i_1}(x)-x\mathcal{R}_{2,i_1}(1)\big\rvert\big\lvert\mathcal{R}_{2,i_1'}(x)-x\mathcal{R}_{2,i_1'}(1)\big\rvert dx\int_0^1\big\lvert\mathcal{R}_{2,1}(y)-y\mathcal{R}_{2,1}(1)\big\lvert^2dy,\\
\nonumber\mathcal{I}_3=&\sum_{i_1=2M_1+1}^{\infty}\sum_{i_2=2M_2+1}^{\infty}\sum_{i_1'=2M_1+1}^{\infty}\sum_{i_2'=2M_2+1}^{\infty}\big\lvert D_{i_1,i_2}\big\rvert\big\lvert D_{i_1',i_2'}\big\rvert\int_0^1\big\lvert\mathcal{R}_{2,i_1}(x)-x\mathcal{R}_{2,i_1}(1)\big\rvert\big\lvert\mathcal{R}_{2,i_1'}(x)-x\mathcal{R}_{2,i_1'}(1)\big\rvert dx\\
&~~~~~~~~~~~\times\int_0^1\big\lvert\mathcal{R}_{2,i_2}(y)-y\mathcal{R}_{2,i_2}(1)\big\rvert\big\lvert\mathcal{R}_{2,i_2'}(y)-y\mathcal{R}_{2,i_2'}(1)\big\rvert dy,\\
\mathcal{I}_4=& 2\sum_{i_2=2M_2+1}^{\infty}\sum_{i_1'=2M_1+1}^{\infty}\big\lvert D_{1,i_2}\big\rvert\big\lvert D_{i_1',1}\big\rvert\int_0^1\big\lvert\mathcal{R}_{2,1}(x)-x\mathcal{R}_{2,1}(1)\big\rvert\big\lvert\mathcal{R}_{2,i_1'}(x)-x\mathcal{R}_{2,i_1'}(1)\big\rvert dx\\
&~~~~~~~~~~~\times\int_0^1\big\lvert\mathcal{R}_{2,i_2}(y)-y\mathcal{R}_{2,i_2}(1)\big\rvert\big\lvert\mathcal{R}_{2,1}(y)-y\mathcal{R}_{2,1}(1)\big\rvert dy,\\
\nonumber\mathcal{I}_5=&2\sum_{i_2=2M_2+1}^{\infty}\sum_{i_1'=2M_1+1}^{\infty}\sum_{i_2'=2M_2+1}^{\infty}\big\lvert D_{1,i_2}\big\rvert\big\lvert D_{i_1',i_2'}\big\rvert\int_0^1\big\lvert\mathcal{R}_{2,1}(x)-x\mathcal{R}_{2,1}(1)\big\rvert\big\lvert\mathcal{R}_{2,i_1'}(x)-x\mathcal{R}_{2,i_1'}(1)\big\rvert dx\\
&~~~~\times\int_0^1\big\lvert\mathcal{R}_{2,i_2}(y)-y\mathcal{R}_{2,i_2}(1)\big\rvert\big\lvert\mathcal{R}_{2,i_2'}(y)-y\mathcal{R}_{2,i_2'}(1)\big\rvert dy,\\
\nonumber\mathcal{I}_6=&2\sum_{i_1=2M_1+1}^{\infty}\sum_{i_1'=2M_1+1}^{\infty}\sum_{i_2'=2M_2+1}^{\infty}\big\lvert D_{i_1,1}\big\rvert\big\lvert D_{i_1',i_2'}\big\rvert\int_0^1\big\lvert\mathcal{R}_{2,i_1}(x)-x\mathcal{R}_{2,i_1}(1)\big\rvert\big\lvert\mathcal{R}_{2,i_1'}(x)-x\mathcal{R}_{2,i_1'}(1)\big\rvert dx\\
&~~~~\times\int_0^1\big\lvert\mathcal{R}_{2,1}(y)-y\mathcal{R}_{2,1}(1)\big\rvert\big\lvert\mathcal{R}_{2,i_2'}(y)-y\mathcal{R}_{2,i_2'}(1)\big\rvert dy,
\end{align*}
} 
Applying Lemma \ref{Pdoc1_lem5} and Lemma \ref{Pdoc1_lem6}, one has
{\small
\begin{align}
    \nonumber\mathcal{I}_1\leq& \sum_{i_2=2M_2+1}^{\infty}\sum_{i_2'=2M_2+1}^{\infty}\dfrac{\mathscr{M}^2}{2^{j_2+1}2^{j_2'+1}}\Bigg[\int_0^1\Big\lvert\mathcal{R}_{2,1}(x)\Big\rvert^2dx+\int_0^1 x^2\Big\lvert\mathcal{R}_{2,1}(1)\Big\rvert^2dx\\
    \nonumber&~~~+\int_0^1 2\lvert x\rvert\Big\lvert\mathcal{R}_{2,1}(x)\Big\rvert\Big\lvert\mathcal{R}_{2,1}(1)\Big\rvert dx\Bigg] \Bigg[\int_0^1\Big\lvert\mathcal{R}_{2,i_2}(y)\Big\rvert\Big\lvert\mathcal{R}_{2,i_2'}(y)\Big\rvert dy+\int_0^1 \lvert y\rvert\Big\lvert\mathcal{R}_{2,i_2}(y)\Big\rvert\Big\lvert\mathcal{R}_{2,i_2'}(1)\Big\rvert dy\\
    \nonumber &~~~+\int_0^1 \lvert y\rvert\Big\lvert\mathcal{R}_{2,i_2}(1)\Big\rvert\Big\lvert\mathcal{R}_{2,i_2'}(y)\Big\rvert dy+\int_0^1 y^2\Big\lvert\mathcal{R}_{2,i_2}(1)\Big\rvert\Big\lvert\mathcal{R}_{2,i_2'}(1)\Big\rvert dy\Bigg]\\
    \nonumber\leq& \sum_{j_2=J_2+1}^{\infty}\sum_{j_2'=J_2+1}^{\infty}\dfrac{\mathscr{M}^2}{2^{j_2+1}2^{j_2'+1}}\Big[\dfrac{1}{(2!)^2}+\dfrac{1}{3(2!)^2}+\dfrac{1}{(2!)^2}\Big]\Big[\dfrac{[\mathscr{C}(2)]^2}{(2^{j_2+1})^2(2^{j_2'+1})^2}+\dfrac{[\mathscr{C}(2)]^2}{(2^{j_2+1})^2(2^{j_2'+1})^2}+\dfrac{[\mathscr{C}(2)]^2}{3(2^{j_2+1})^2(2^{j_2'+1})^2}\Big]\\
    \nonumber=& \dfrac{49\mathscr{M}^2[\mathscr{C}(2)]^2}{3\cdot3\cdot (2!)^2}\sum_{j_2=J_2+1}^{\infty}\dfrac{1}{(2^{j_2+1})^3}\sum_{j_2'=J_2+1}^{\infty}\dfrac{1}{(2^{j_2'+1})^3}= \dfrac{49\mathscr{M}^2[\mathscr{C}(2)]^2}{3\cdot3\cdot (2!)^2\cdot8\cdot8}\sum_{j_2=J_2+1}^{\infty}\dfrac{1}{8^{j_2}}\sum_{j_2'=J_2+1}^{\infty}\dfrac{1}{8^{j_2'}}\\
    =& \dfrac{49\mathscr{M}^2[\mathscr{C}(2)]^2}{3\cdot3\cdot (2!)^2\cdot8\cdot8\cdot8^{J_2+1}\cdot8^{J_2+1}}\sum_{j_2=0}^{\infty}\dfrac{1}{8^{j_2}}\sum_{j_2'=0}^{\infty}\dfrac{1}{8^{j_2'}}=\dfrac{\mathscr{M}^2[\mathscr{C}(2)]^2}{3\cdot3\cdot (2!)^2\cdot8\cdot8\cdot8^{J_2}\cdot8^{J_2}}
    \leq CM_2^{-6}\leq C\mathcal{M}^{-6},\label{Pdoc1_46}
\end{align}}
where $\mathcal{M}=\min\{M_1,M_2\}$. Proceeding in a similar way, one can obtain the following bounds:
\begin{equation}\label{Pdoc1_47}
\left\{
\begin{array}{ll}
    \mathcal{I}_2\leq&  CM_1^{-6}\leq C\mathcal{M}^{-6},\\
    \mathcal{I}_3\leq& CM_1^{-6}M_2^{-6}\leq C\mathcal{M}^{-12}\leq C\mathcal{M}^{-6},\\
    \mathcal{I}_4\leq& CM_1^{-3}M_2^{-3}\leq C\mathcal{M}^{-6},\\
    \mathcal{I}_5\leq& CM_1^{-3}M_2^{-6}\leq C\mathcal{M}^{-9}\leq C\mathcal{M}^{-6},\\
    \mathcal{I}_6\leq& CM_1^{-6}M_2^{-3}\leq C\mathcal{M}^{-9}\leq C\mathcal{M}^{-6}.
    \end{array}\right.
\end{equation}
Substituting (\ref{Pdoc1_46}) and (\ref{Pdoc1_47}) into (\ref{Pdoc1_45}), and then, taking the square root of both sides, one can obtain the desired error bound.
\end{proof}

\begin{algorithm}[ht]
{\small
\caption{ALGORITHM FOR THE SOLUTION OF (\ref{Pdoc1_1}) WHEN $\Omega\subset\mathbb{R}^2$}\label{Pdoc1_Alg2}
\begin{algorithmic}[1]
    \State \textbf{Input 1:} $\alpha$ \Comment{Order of the fractional operator}
    \State \textbf{Input 2:} $N$ \Comment{Number of mesh interval towards time}
    \State \textbf{Input 2:} $J_1,~J_2$ \Comment{Maximum level of resolution towards $x,~y$}
    \State \textbf{Input 4:} $g,~h_1,~h_2,~h_3,~h_4$ \Comment{Initial and boundary conditions}
    \State Compute $\sigma=1-(\alpha/2)$ and the grading parameter $\nu$
    \State Compute $t_n=T*(n/N)^\nu$ for $n=0,1,\ldots,N$
    \State Compute $M_1=2^{J_1},~M_2=2^{J_2}$
    \State Compute $\{(x_{l_1},y_{l_2})\}_{l_1=1,l_2=1}^{2M_1,2M_2}$ using (\ref{Pdoc1_33}) \Comment{The collocation points}
    \State Compute $p,~q_1,~q_2,~r,~f$ and the kernel $\mathcal{K}$ using collocation points  
    \State Compute $\Lambda_0^0$ and $\Lambda_n^{j}$ for $n=1,2,\ldots,N-1$, $j=0,1,\ldots,n$ given in (\ref{Pdoc1_2})
    \For{$n=0,1,\ldots,N-1$} 
        \State Compute the operational matrices corresponding to $\mathcal{U}_{xx}^{n+\sigma},~\mathcal{U}_{yy}^{n+\sigma}$ using (\ref{Pdoc1_27}), (\ref{Pdoc1_28})
        \State Compute the operational matrices corresponding to $\mathcal{U}_{x}^{n+\sigma},~\mathcal{U}_{y}^{n+\sigma}$ using (\ref{Pdoc1_29}), (\ref{Pdoc1_30})
        \State Compute the operational matrices corresponding to $\mathcal{U}^{n+\sigma},~\mathcal{U}^{n+1}$ using (\ref{Pdoc1_31}), (\ref{Pdoc1_32})
        \State Compute the right-hand side vector by utilizing the source term $f$
        \State Compute the unknown wavelet coefficients $\{D_{i_1,i_2}\}_{i_1=1,i_2=1}^{2M_1,2M_2}$ by solving the linear system
        \State Put the wavelet coefficients in to (\ref{Pdoc1_32}) to compute $\boldsymbol{\mathcal{U}}^{n+1}$
    \EndFor
    \State \textbf{Output:} Obtained the complete solution $\mathcal{U}^{n+1}$ for $n=0,1,\ldots,N-1$
\end{algorithmic}}
\end{algorithm}

\section{Stability analysis for TFIPDEs based on graded mesh} \label{sec_Stability}
In this section, we discuss the stability of the proposed scheme based on non-uniform mesh in time, i.e., $\nu\neq1$. Let $\{\mathcal{U}_m^n\}_{m=0,n=0}^{M,N}$ is the mesh function corresponding to a continuous function $\mathcal{U}:{\varOmega}^{M,N}\rightarrow \mathbb{R}$. Then we define the discrete maximum norm as
\[\|\mathcal{U}\|_\infty:=\max_{(x_m,t_n)\in{\varOmega}^{M,N}}\lvert\mathcal{U}(x_m,t_n)\rvert\;\;\mbox{and }\;\|\mathcal{U}^n\|_\infty:=\max_{0\leq m\leq M}\lvert\mathcal{U}_m^n\rvert.\] 
The stability result will be established later in Theorem \ref{Pdoc1_thm1} based on the abovementioned norm. The following technical lemma will be helpful in establishing the stability estimate.
\begin{lemma}\label{Pdoc1_lem1}
(See Lemma 4 of \cite{ChenError2019})
    Define the local mesh ratio $\eta_j:=\tau_{j+1}/\tau_j$ for $j=1,2,\ldots,N-1$. Let $1\geq\sigma\geq1-\alpha/2$. Then one has the following:
    \begin{enumerate}
        \item $\Lambda_n^0>\dfrac{1}{t_{n+\sigma}^\alpha\Gamma(1-\alpha)}>0$ for $n\geq0$.
        \item $(2\sigma-1)\Lambda_1^1-\sigma\Lambda_1^0>0$ and $\Lambda_n^1>\Lambda_n^0$ for $n\geq1$.
        \item If $\eta_{j-1}^2(\eta_{j-1}+1)\geq\dfrac{\eta_j}{\eta_j+1}$ for some $j\in\{2,3,\ldots,n\}$, with $n\geq2$, then $\Lambda_n^j>\Lambda_n^{j-1}$.
        \item If $\eta_{n-1}^2\Big(2-(1/\sigma)+\eta_n(\eta_n+2)\Big)\geq\dfrac{\eta_n(\eta_n+1)}{\eta_{n-1}+1}$ for some $n\in\{2,3,\ldots,N\}$, then $(2\sigma-1)\Lambda_n^n-\sigma\Lambda_n^{n-1}>0$.
    \end{enumerate}
\end{lemma}
\noindent The discrete problem (\ref{Pdoc1_12}) can be rewritten as:
\begin{equation}\label{Pdoc1_51}
    \left\{
    \begin{array}{ll}
	\Big[\Lambda_n^{n}+\dfrac{2p(x_m)\sigma}{(\Delta x)^2}+r(x_m)\sigma+\dfrac{\mu\sigma^2\tau_{n+1}}{2}\mathcal{K}(x_m,0)\Big]\mathcal{U}_m^{n+1}=\Lambda_n^0\mathcal{U}_m^0+\displaystyle{\sum_{j=1}^n} (\Lambda_n^{j}-\Lambda_n^{j-1})\mathcal{U}_m^{j}\\
        \hspace{2.5cm}+\Big[\dfrac{p(x_m)\sigma}{(\Delta x)^2}+\dfrac{q(x_m)\sigma}{2\Delta x}\Big]\mathcal{U}_{m-1}^{n+1}+\Big[\dfrac{p(x_m)\sigma}{(\Delta x)^2}-\dfrac{q(x_m)\sigma}{2\Delta x}\Big]\mathcal{U}_{m+1}^{n+1}\\[6pt]
        \hspace{2.5cm}+\Big[\dfrac{p(x_m)(1-\sigma)}{(\Delta x)^2}+\dfrac{q(x_m)(1-\sigma)}{2\Delta x}\Big]\mathcal{U}_{m-1}^{n}+\Big[\dfrac{p(x_m)(1-\sigma)}{(\Delta x)^2}-\dfrac{q(x_m)(1-\sigma)}{2\Delta x}\Big]\mathcal{U}_{m+1}^{n}\\[6pt]
        \hspace{2.5cm}-\Big[\dfrac{2p(x_m)(1-\sigma)}{(\Delta x)^2}+r(x_m)(1-\sigma)+\dfrac{\mu\tau_n}{2}\mathcal{K}(x_m,t_{n+\sigma}-t_n)+\dfrac{\mu\sigma\tau_{n+1}}{2}\mathcal{K}(x_m,t_{n+\sigma}-t_n)\\[6pt]
        \hspace{2.5cm}+\dfrac{\mu\tau_{n+1}\sigma(1-\sigma)}{2}\mathcal{K}(x_m,0)\Big]\mathcal{U}_m^n-\dfrac{\mu\tau_n}{2}\mathcal{K}(x_m,t_{n+\sigma}-t_{n-1})\mathcal{U}_m^{n-1}\\[6pt]
        \hspace{2.5cm}-\mu\displaystyle{\sum_{j=0}^{n-2}}\dfrac{\tau_{n+1}}{2}\Big[\mathcal{K}(x_m,t_{n+\sigma}-t_{j+1})\mathcal{U}_m^{j+1}+\mathcal{K}(x_m,t_{n+\sigma}-t_{j})\mathcal{U}_m^{j}\Big]+f(x_m,t_{n+\sigma}),\\[6pt]
	\mbox{for }m=1,2,\ldots,M-1;~n=0,1,\ldots,N-1,\\
	\mathcal{U}_m^0=g(x_m)\;\; \mbox{for } m=0,1,\ldots,M,~
	\mathcal{U}_0^{n+1}=h_1(t_{n+1}) \mbox{ and } \mathcal{U}_M^{n+1}=h_2(t_{n+1})\;\; \mbox{for } n=0,1,\ldots,N-1.
    \end{array}\right.
\end{equation}
The following theorem demonstrates the stability estimate of the proposed numerical approximation in favor of the TFIPDEs (\ref{Pdoc1_1}), which will be used later in Theorem \ref{Pdoc1_thm2} to show the main convergence result.
\begin{theorem}\label{Pdoc1_thm1}
    Let us assume that $1\geq\sigma\geq1-\alpha/2$. Based on the conditions given in Lemma \ref{Pdoc1_lem1}, the solution of the discrete problem (\ref{Pdoc1_12}) satisfies the following inequality:
    \[\big\|\mathcal{U}^{n+1}\big\|_\infty\leq \big\|\mathcal{U}^0\big\|_\infty+t_{n+\sigma}^\alpha\Gamma(1-\alpha)\max_{0\leq n\leq N-1}\big\|f^{n+\sigma}\big\|_\infty,~n=0,1,\ldots,N-1.\]
\end{theorem}
\begin{proof}
From Lemma \ref{Pdoc1_lem1}, we have $\Lambda_n^{n}>\Lambda_{n}^{n-1}>\cdots>\Lambda_n^{0}> \mathscr{C}>0$, for $n=0,1,\ldots,N-1$ with $\mathscr{C}=\dfrac{1}{t_{n+\sigma}^\alpha\Gamma(1-\alpha)}$. For any fixed $n\in\{0,1,\cdots,N-1\}$, choose an $m^*$ such that $\|w^{n+1}\|_\infty=\lvert w_{m^*}^{n+1}\rvert$. Then,  (\ref{Pdoc1_51}) yields 
\begin{equation*}
	\begin{array}{ll}
	\Big[\Lambda_n^{n}+\dfrac{2p(x_{m^*})\sigma}{(\Delta x)^2}+r(x_{m^*})\sigma+\dfrac{\mu\sigma^2\tau_{n+1}}{2}\mathcal{K}(x_{m^*},0)\Big]\mathcal{U}_{m^*}^{n+1}=\Lambda_n^0\mathcal{U}_{m^*}^0+\displaystyle{\sum_{j=1}^n} (\Lambda_n^{j}-\Lambda_n^{j-1})\mathcal{U}_{m^*}^{j}\\
        \hspace{2.5cm}+\Big[\dfrac{p(x_{m^*})\sigma}{(\Delta x)^2}+\dfrac{q(x_{m^*})\sigma}{2\Delta x}\Big]\mathcal{U}_{m^*-1}^{n+1}+\Big[\dfrac{p(x_{m^*})\sigma}{(\Delta x)^2}-\dfrac{q(x_{m^*})\sigma}{2\Delta x}\Big]\mathcal{U}_{m^*+1}^{n+1}\\[6pt]
        \hspace{2.5cm}+\Big[\dfrac{p(x_{m^*})(1-\sigma)}{(\Delta x)^2}+\dfrac{q(x_{m^*})(1-\sigma)}{2\Delta x}\Big]\mathcal{U}_{m^*-1}^{n}+\Big[\dfrac{p(x_{m^*})(1-\sigma)}{(\Delta x)^2}-\dfrac{q(x_{m^*})(1-\sigma)}{2\Delta x}\Big]\mathcal{U}_{m^*+1}^{n}\\[6pt]
        \hspace{2.5cm}-\Big[\dfrac{2p(x_{m^*})(1-\sigma)}{(\Delta x)^2}+r(x_{m^*})(1-\sigma)+\dfrac{\mu\tau_n}{2}\mathcal{K}(x_{m^*},t_{n+\sigma}-t_n)+\dfrac{\mu\sigma\tau_{n+1}}{2}\mathcal{K}(x_{m^*},t_{n+\sigma}-t_n)\\[6pt]
        \hspace{2.5cm}+\dfrac{\mu\tau_{n+1}\sigma(1-\sigma)}{2}\mathcal{K}(x_{m^*},0)\Big]\mathcal{U}_{m^*}^n-\dfrac{\mu\tau_n}{2}\mathcal{K}(x_{m^*},t_{n+\sigma}-t_{n-1})\mathcal{U}_{m^*}^{n-1}\\[6pt]
        \hspace{2.5cm}-\mu\displaystyle{\sum_{j=0}^{n-2}}\dfrac{\tau_{n+1}}{2}\Big[\mathcal{K}(x_{m^*},t_{n+\sigma}-t_{j+1})\mathcal{U}_{m^*}^{j+1}+\mathcal{K}(x_{m^*},t_{n+\sigma}-t_{j})\mathcal{U}_{m^*}^{j}\Big]+f(x_{m^*},t_{n+\sigma}).
	\end{array}
\end{equation*}
Keeping in mind that $1\geq\sigma\geq1-\alpha/2$, $r\geq 0$, and the kernel $\mathcal{K}\geq0$ together with the condition defined in (\ref{Nonrestrictive_assumption}) and the choice of $m^*$ yields
\begin{equation*}
    \begin{array}{ll}
        \Big[\Lambda_n^{n}+\dfrac{2p(x_{m^*})\sigma}{(\Delta x)^2}\Big]\big\|\mathcal{U}^{n+1}\big\|_\infty\leq \Lambda_n^{0}\big\|\mathcal{U}^{0}\big\|_\infty+\displaystyle{\sum_{j=1}^n} [\Lambda_n^{j}-\Lambda_{n}^{j-1}]\big\|\mathcal{U}^{j}\big\|_\infty+\Big[\dfrac{p(x_{m^*})\sigma}{(\Delta x)^2}+\dfrac{q(x_{m^*})\sigma}{2\Delta x}\Big]\big\|\mathcal{U}^{n+1}\big\|_\infty\\[8pt]
        \hspace{5cm}+\Big[\dfrac{p(x_{m^*})\sigma}{(\Delta x)^2}-\dfrac{q(x_{m^*})\sigma}{2\Delta x}\Big]\big\|\mathcal{U}^{n+1}\big\|_\infty+\Big[\dfrac{p(x_{m^*})(1-\sigma)}{(\Delta x)^2}+\dfrac{q(x_{m^*})(1-\sigma)}{2\Delta x}\Big]\big\|\mathcal{U}^{n}\big\|_\infty\\[8pt]
        \hspace{5cm}+\Big[\dfrac{p(x_{m^*})(1-\sigma)}{(\Delta x)^2}-\dfrac{q(x_{m^*})(1-\sigma)}{2\Delta x}\Big]\big\|\mathcal{U}^{n}\big\|_\infty-\Big[\dfrac{2p(x_{m^*})(1-\sigma)}{(\Delta x)^2}\Big]\big\|\mathcal{U}^n\big\|_\infty+\big\|f^{n+\sigma}\big\|_\infty.
    \end{array}
\end{equation*}
The given condition $\Lambda_n^{n}>\Lambda_{n}^{n-1}>\cdots>\Lambda_n^{0}\geq \mathscr{C}>0$ enables us to have
\begin{align}
    \nonumber \Lambda_n^{n}\big\|\mathcal{U}^{n+1}\big\|_\infty&\leq \displaystyle{\sum_{j=1}^n} [\Lambda_n^{j}-\Lambda_{n}^{j-1}]\big\|\mathcal{U}^{j}\big\|_\infty+\Lambda_n^{0}\Big[\big\|\mathcal{U}^{0}\big\|_\infty+t_{n+\sigma}^\alpha\Gamma(1-\alpha)\big\|f^{n+\sigma}\big\|_\infty\Big]\\
    &\leq \displaystyle{\sum_{j=1}^n} [\Lambda_n^{j}-\Lambda_{n}^{j-1}]\big\|\mathcal{U}^{j}\big\|_\infty+\Lambda_n^{0}\Big[\big\|\mathcal{U}^{0}\big\|_\infty+t_{n+\sigma}^\alpha\Gamma(1-\alpha)\max_{0\leq n\leq N-1}\big\|f^{n+\sigma}\big\|_\infty\Big]. \label{Pdoc1_14}
\end{align}
Let us take $\mathscr{D}=\big\|\mathcal{U}^0\big\|_\infty+t_{n+\sigma}^\alpha\Gamma(1-\alpha)\displaystyle{\max_{0\leq n\leq N-1}}\big\|f^{n+\sigma}\big\|_\infty$. We now use mathematical induction to prove this theorem. For $n=0$, the theorem is automatically satisfied from (\ref{Pdoc1_14}). Let us assume that the inequality holds for $n=1,2,\ldots,k-1$, i.e. $\big\|\mathcal{U}^{n+1}\big\|_\infty\leq \mathscr{D},~n=1,\ldots,k-1$. Now, from (\ref{Pdoc1_14}), we have 
\begin{equation*}
    \begin{array}{ll}
        \Lambda_k^{k}\big\|\mathcal{U}^{k+1}\big\|_\infty\leq \displaystyle{\sum_{j=1}^k}(\Lambda_k^{j}-\Lambda_{k}^{j-1})\big\|\mathcal{U}^j\big\|_\infty+\Lambda_k^{0}\mathscr{D} \leq \displaystyle{\sum_{j=1}^k}(\Lambda_k^{j}-\Lambda_{k}^{j-1})\mathscr{D}+\Lambda_k^{0}\mathscr{D} =\Lambda_k^{k}\mathscr{D},
    \end{array}
\end{equation*}
which implies
\[\big\|\mathcal{U}^{k+1}\big\|_\infty\leq \mathscr{D}=\big\|\mathcal{U}^0\big\|_\infty+t_{n+\sigma}^\alpha\Gamma(1-\alpha)\displaystyle{\max_{0\leq n\leq N-1}}\big\|f^{n+\sigma}\big\|_\infty.\]
Hence, the theorem is proved. 
\end{proof}

\section{Convergence analysis}\label{sec_Convergence}
In this section, first, we estimate the truncation error bounds for the approximation of the fractional operator, Volterra integral operator, and spatial derivatives. Subsequently, we achieve the main convergence results for the given TFIPDEs (\ref{Pdoc1_1}) for both the cases, i.e., $\Omega\subset\mathbb{R}$ and $\Omega\subset\mathbb{R}^2$. To begin the analysis, we start with an elementary interpolation error bound given in the following lemma.
\begin{lemma}\label{Pdoc1_lem3}
    For any $\phi(t)\in C^2(\Omega_t)$, and $n=0,1,\ldots,N-1$, one yields
    \[\lvert \phi(t_{n+\sigma})-\sigma\phi(t_{n+1})-(1-\sigma)\phi(t_n)\rvert\leq CN^{-2}.\]
\end{lemma}
\begin{proof}
Using the Taylor series expansion, we have
\begin{align*}
   \phi(t_{n+1})=&\phi(t_{n+\sigma})+(1-\sigma)\tau_{n+1}\phi'(t_{n+\sigma})+\dfrac{(1-\sigma)^2\tau_{n+1}^2}{2}\phi''(t_{n+\sigma})+\cdots,\\
   \phi(t_n)=&\phi(t_{n+\sigma})-\sigma\tau_{n+1}\phi'(t_{n+\sigma})+\dfrac{\sigma^2\tau_{n+1}^2}{2}\phi''(t_{n+\sigma})-\cdots.
\end{align*}
Consequently, it yields
\[\sigma\phi(t_{n+1})+(1-\sigma)\phi(t_n)=\phi(t_{n+\sigma})+\mathcal{O}(\tau_{n+1}^2).\]
Using the bounds given in (\ref{Pdoc1_53}), we obtain
\[\lvert \phi(t_{n+\sigma})-\sigma\phi(t_{n+1})-(1-\sigma)\phi(t_n)\rvert\leq CN^{-2}.\]
\end{proof}
The following lemma describes the truncation error bound for the discretization of the fractional operator based on non-uniform mesh for solution having time singularity.
\begin{lemma}\label{Pdoc1_lem2}
(See Eq. (3.6) of \cite{SinghAfully2024})
Let $\mathcal{U}(\cdot,t)\in C(\overline{\Omega}_t)\cap C^3(\Omega_t)$. For $n=0,1,\ldots, N-1$, the discretization error of the fractional Caputo operator on a temporally graded mesh satisfies the following bound:
\[\Big\lvert {}^{(1)}\mathscr{R}^{n+\sigma}\Big\rvert\leq Ct_{n+\sigma}^{-\alpha}N^{-\min\{\nu\alpha,3-\alpha\}}.\]    
\end{lemma}
Now, we discuss the truncation error bound for the integral operator based on the discretization given in (\ref{Pdoc1_4}) on a non-uniform mesh.
\begin{lemma}\label{Pdoc1_lem4}
    If the solution of the proposed problem meets the regularity condition defined in (\ref{Pdoc1_52}), then the discretization error corresponding to the Volterra integral operator on a temporally graded mesh satisfies the following bound:
    \[\Big\lvert {}^{(2)}\mathscr{R}^{n+\sigma}\Big\rvert\leq CN^{-\min\{2,\nu(\alpha+1)\}}~\mbox{for}~n=0,1,\ldots,N-1.\]
\end{lemma}
\begin{proof}
The discretization of the integral operator discussed in (\ref{Pdoc1_4}) yields the following:
\begin{align}
    \nonumber\Big\lvert {}^{(2)}\mathscr{R}^{n+\sigma}\Big\rvert=&\Bigg\lvert \displaystyle{\int_{0}^{t_{n+\sigma}}}\mathcal{K}(\cdot,t_{n+\sigma}-\xi)\mathcal{U}(\cdot,\xi)\;d\xi-\mathscr{J}_N \mathcal{U}(\cdot,t_{n+\sigma})\Bigg\rvert\\
    \nonumber\leq& \underbrace{\Bigg\lvert \displaystyle{\int_{0}^{t_{n}}}\mathcal{K}(\cdot,t_{n+\sigma}-\xi)\mathcal{U}(\cdot,\xi)\;d\xi-\sum_{j=0}^{n-1}\dfrac{\tau_{j+1}}{2}\big[\mathcal{K}(\cdot,t_{n+\sigma}-t_{j+1})\mathcal{U}^{j+1}+\mathcal{K}(\cdot,t_{n+\sigma}-t_j)\mathcal{U}^j\big]\Bigg\rvert}_{=:{}^{(2)}\widetilde{\mathscr{R}}^{n+\sigma}}\\
    &+\underbrace{\Bigg\lvert \displaystyle{\int_{t_n}^{t_{n+\sigma}}}\mathcal{K}(\cdot,t_{n+\sigma}-\xi)\mathcal{U}(\cdot,\xi)\;d\xi-\dfrac{\sigma\tau_{n+1}}{2}\mathcal{K}(\cdot,t_{n+\sigma}-t_n)\mathcal{U}^n-\dfrac{\sigma\tau_{n+1}}{2}\mathcal{K}(\cdot,0)\big[\sigma\mathcal{U}^{n+1}+(1-\sigma)\mathcal{U}^n\big]\Bigg\rvert}_{=:{}^{(2)}\widehat{\mathscr{R}}^{n+\sigma}}.\label{Pdoc1_58}
\end{align}
From Lemma 4.6 of \cite{SantraHigherorder2023}, we have
\begin{equation}\label{Pdoc1_59}
    \Big\lvert {}^{(2)}\widetilde{\mathscr{R}}^{n+\sigma}\Big\rvert\leq CN^{-\min\{2,\nu(\alpha+1)\}}.
\end{equation}
Now, using Lemma \ref{Pdoc1_lem3}, one can have
\begin{align}
    \nonumber \Big\lvert {}^{(2)}\widehat{\mathscr{R}}^{n+\sigma}\Big\rvert\leq& \Bigg\lvert \displaystyle{\int_{t_n}^{t_{n+\sigma}}}\mathcal{K}(\cdot,t_{n+\sigma}-\xi)\mathcal{U}(\cdot,\xi)\;d\xi-\Big(\dfrac{\sigma\tau_{n+1}}{2}\mathcal{K}(\cdot,t_{n+\sigma}-t_n)\mathcal{U}^n+\dfrac{\sigma\tau_{n+1}}{2}\mathcal{K}(\cdot,0)\mathcal{U}^{n+\sigma}\Big)\Bigg\rvert +\mathcal{O}(N^{-2})\\
    \nonumber=& \Bigg\lvert \displaystyle{\int_{t_n}^{t_{n+\sigma}}}\mathcal{K}(\cdot,t_{n+\sigma}-\xi)\mathcal{U}(\cdot,\xi)\;d\xi-\dfrac{1}{2}\int_{t_n}^{t_{n+\sigma}}\Big(\mathcal{K}(\cdot,t_{n+\sigma}-t_n)\mathcal{U}^n+\mathcal{K}(\cdot,0)\mathcal{U}^{n+\sigma}\Big)d\xi\Bigg\rvert +\mathcal{O}(N^{-2})\\
    \nonumber=& \dfrac{1}{2}\Bigg\lvert \displaystyle{\int_{t_n}^{t_{n+\sigma}}}\Big(\omega(\xi)-\omega(t_n)\Big)d\xi+\displaystyle{\int_{t_n}^{t_{n+\sigma}}}\Big(\omega(\xi)-\omega(t_{n+\sigma})\Big)d\xi\Bigg\rvert +\mathcal{O}(N^{-2}),
\end{align}
where the function $\omega(t)=\mathcal{K}(\cdot,t_{n+\sigma}-t)\mathcal{U}(\cdot,t)$. Using triangle inequality and the Lagrange mean value theorem, we reach at
\begin{align}
    \nonumber \Big\lvert {}^{(2)}\widehat{\mathscr{R}}^{n+\sigma}\Big\rvert\leq& \dfrac{1}{2}\lvert \xi-t_n\rvert\lvert\omega'(\xi_1)\rvert\int_{t_n}^{t_{n+\sigma}}d\xi+\dfrac{1}{2}\lvert \xi-t_{n+\sigma}\rvert\lvert\omega'(\xi_2)\rvert\int_{t_n}^{t_{n+\sigma}}d\xi+\mathcal{O}(N^{-2})\\
    \nonumber\leq& \dfrac{\sigma}{2}\tau_{n+1}^2\lvert\omega'(\xi_1)\rvert+\dfrac{\sigma}{2}\tau_{n+1}^2\lvert\omega'(\xi_2)\rvert+\mathcal{O}(N^{-2}),
\end{align}
where $\xi_1,\xi_2\in(t_n,t_{n+\sigma})$. Now, applying the bounds given in (\ref{Pdoc1_53}), we have
\begin{equation}\label{Pdoc1_60}
    \Big\lvert {}^{(2)}\widehat{\mathscr{R}}^{n+\sigma}\Big\rvert\leq CN^{-2}+\mathcal{O}(N^{-2})\leq CN^{-2}.
\end{equation}
Finally, the bounds given in (\ref{Pdoc1_59}) and (\ref{Pdoc1_60}) are used in (\ref{Pdoc1_58}) to obtain the desired result.
\end{proof}
The following lemma shows the discretization error of the spatial derivatives when $\Omega\subset\mathbb{R}$.
\begin{lemma}\label{Pdoc1_lem7}
    The discretization of the spatial derivatives, discussed in (\ref{Pdoc1_3}) and (\ref{Pdoc1_21}), satisfies the following error bounds:
    \[\Big\lvert{}^{(3)}\mathscr{R}_m^{n+\sigma}\Big\rvert\leq C(N^{-2}+M^{-2}),~~ \Big\lvert {}^{(4)}\mathscr{R}_m^{n+\sigma}\Big\rvert\leq C(N^{-2}+M^{-2}),\]
    for $n=0,1,\ldots,N-1$, $m=1,2,\ldots,M-1$.
\end{lemma}
\begin{proof}
    The spatial discretizations are the well-known second-order central difference approximations, and then the proof follows from Lemma \ref{Pdoc1_lem3}. 
\end{proof}
The error equation corresponding to TFIPDEs for one-dimensional case (i.e., for $\Omega\subset\mathbb{R}$) can be obtained by subtracting (\ref{Pdoc1_12}) from (\ref{Pdoc1_55}), and which can be expressed as follows:
\begin{equation}\label{Pdoc1_15}
    \left\{
    \begin{array}{ll}
	{}^{L2\mbox{-}1_\sigma}\mathcal{D}_N^\alpha e_m^{n+\sigma}-p(x_m)\delta_{\Delta x}^2~e_m^{n+\sigma}+q(x_m)D_{\Delta x}^0e_m^{n+\sigma}+ r(x_m)e_m^{n+\sigma}
	+\mu\mathscr{J}_N e_m^{n+\sigma}=\mathscr{R}_m^{n+\sigma}, \\[4pt]
	\mbox{for }m=1,2,\ldots,M-1;~n=0,1,\ldots,N-1,\\[4pt]
	e_m^0=0\;\; \mbox{for } m=0,1,\ldots,M,\\
	e_0^{n+1}=e_M^{n+1}=0\;\; \mbox{for } n=0,1,\ldots,N-1,
    \end{array}\right.
\end{equation}
where $\mathscr{R}_m^{n+\sigma}$ is the remainder term defined in (\ref{Pdoc1_56}), and $e_m^{n+\sigma}=\mathcal{U}(x_m,t_{n+\sigma})-\mathcal{U}_m^{n+\sigma}$.

\subsection{Error bounds}

\begin{theorem}\label{Pdoc1_thm2}
Assume that $1\geq\sigma\geq1-\alpha/2$ and let the solution of the given TFIPDE (\ref{Pdoc1_1}) satisfies the regularity condition given in (\ref{Pdoc1_52}). Then, for $\Omega\subset\mathbb{R}$, if $\{\mathcal{U}(x_m,t_n)\}_{m=0,n=0}^{M,N}$ and $\{\mathcal{U}_m^n\}_{m=0,n=0}^{M,N}$ be the exact and numerical solutions of (\ref{Pdoc1_1}) by using the discrete scheme (\ref{Pdoc1_12}), respectively. Then, one has the following error bound:
\[\|e\|_\infty\leq C\big(N^{-\min\{\nu\alpha,2\}}+M^{-2}\big).\]
\end{theorem}
\begin{proof}
The solution of the discrete problem (\ref{Pdoc1_15}) satisfies the stability estimate given in Theorem \ref{Pdoc1_thm1}. Hence, for $n=0,1,\ldots,N-1$, we have the following bound:
\begin{align*}
    \big\|e^{n+1}\big\|_\infty\leq& \big\|e^0\big\|_\infty+t_{n+\sigma}^\alpha\Gamma(1-\alpha)\max_{0\leq n\leq N-1}\Big\|\mathscr{R}^{n+\sigma}\Big\|_\infty\\[6pt]
    \leq& t_{n+\sigma}^\alpha\Gamma(1-\alpha)\max_{0\leq n\leq N-1}\Big\|{}^{(1)}\mathscr{R}^{n+\sigma}\Big\|_\infty+t_{n+\sigma}^\alpha\Gamma(1-\alpha)\max_{0\leq n\leq N-1}\Big\|{}^{(2)}\mathscr{R}^{n+\sigma}\Big\|_\infty\\
    &+t_{n+\sigma}^\alpha\Gamma(1-\alpha)\max_{0\leq n\leq N-1}\Big\|{}^{(3)}\mathscr{R}^{n+\sigma}\Big\|_\infty+t_{n+\sigma}^\alpha\Gamma(1-\alpha)\max_{0\leq n\leq N-1}\Big\|{}^{(4)}\mathscr{R}^{n+\sigma}\Big\|_\infty
\end{align*}
Using Lemmas \ref{Pdoc1_lem2}, \ref{Pdoc1_lem4}, and \ref{Pdoc1_lem7}, we get
\begin{align*}
    \big\|e^{n+1}\big\|_\infty\leq CN^{-\min\{\nu\alpha,3-\alpha\}}+Ct_{n+\sigma}^\alpha N^{-\min\{2,\nu(\alpha+1)\}}
    +Ct_{n+\sigma}^\alpha (N^{-2}+M^{-2})+Ct_{n+\sigma}^\alpha (N^{-2}+M^{-2}).
\end{align*}
Notice that $t_{n+\sigma}^\alpha\leq T$ for $\alpha\in(0,1)$ and any $n\in\{0,1,\ldots,N-1\}$. Hence, we obtain
\begin{align*}
    \|e\|_\infty=\max_{0\leq n\leq N-1}\big\|e^{n+1}\big\|_\infty\leq C\big(N^{-\min\{\nu\alpha,2\}}+M^{-2}\big).
\end{align*}
\end{proof}
\begin{remark}\label{Pdoc1_remark1}
In the realm of TFIPDEs with $\Omega\subset\mathbb{R}$, Theorem \ref{Pdoc1_thm2} demonstrates that the proposed methodology can achieve an optimal convergence rate of $\mathcal{O}(N^{-2})$ over time on a non-uniform mesh, provided the grading parameter is appropriately selected (specifically, $\nu\geq 2/\alpha$) for a solution exhibiting time singularity at $t=0$. The convergence remains unchanged when a uniform mesh is used, i.e., $\nu=1$. In this context, the accuracy reduces to $\mathcal{O}(N^{-\alpha})$.
\end{remark}
The following theorem demonstrates the error bounds for temporal semi-discretization of the proposed problem for the two-dimensional case (i.e., for $\Omega\subset\mathbb{R}^2$) based on a non-uniform mesh, which is graded towards the origin. The bounds will be used later in Theorem \ref{Pdoc1_thm7} to obtain the main convergence results when $\Omega\subset\mathbb{R}^2$.
\begin{theorem}\label{Pdoc1_thm5}
Assume that the solution of the proposed TFIPDE with $\Omega\subset\mathbb{R}^2$ satisfies the regularity condition (\ref{Pdoc1_52}) and let $1\geq\sigma\geq1-\alpha/2$. Then, for $n=0,1,\ldots, N-1$, the solution of the semi-discretized problem (\ref{Pdoc1_25}) satisfies the following error bound:
    \[\Big\lvert\mathcal{U}(x,y,t_{n+1})-\mathcal{U}^{n+1}(x,y)\Big\rvert\leq CN^{-\min\{\nu\alpha,2\}}.\]
\end{theorem}
\begin{proof}
From Lemmas \ref{Pdoc1_lem2} and \ref{Pdoc1_lem4}, it can be confirmed that the remainder term depicted in (\ref{Pdoc1_61}) satisfies the following bound:
\[\Big\lvert\widehat{\mathscr{R}}^{n+\sigma}\Big\rvert\leq \Big\lvert{}^{(1)}\mathscr{R}^{n+\sigma}\Big\rvert+\Big\lvert{}^{(2)}\mathscr{R}^{n+\sigma}\Big\rvert\leq Ct_{n+\sigma}^{-\alpha}N^{-\min\{\nu\alpha,3-\alpha\}}+CN^{-\min\{2,\nu(\alpha+1)\}}.\]
With the stability estimate given in Theorem \ref{Pdoc1_thm1} and the hypothesis presented in Theorem \ref{Pdoc1_thm2} together with the above inequality, one can obtain the desired bound.  
\end{proof}
The following theorem illustrates the convergence of the fully discrete solution to the TFIPDEs (\ref{Pdoc1_1}) in two dimensions (i,e., for $\Omega\subset\mathbb{R}^2$), which is based on $L2$-$1_\sigma$ scheme on a graded mesh in time and the multi-dimensional Haar wavelets.
\begin{theorem}\label{Pdoc1_thm7}
The solution to the given TFIPDEs (\ref{Pdoc1_1}) when $\Omega\subset\mathbb{R}^2$ satisfies the following error bound:
    \[\Big\|\mathcal{U}(x,y,t_{n+1})-\mathcal{U}_{M_1M_2}^{n+1}(x,y)\Big\|_{L^2(\overline{\Omega})}\leq C\big(N^{-\min\{\nu\alpha,2\}}+\mathcal{M}^{-3}\big),\]
where $\mathcal{M}=\min\{M_1,M_2\}$.
\end{theorem}
\begin{proof}
By using the triangle inequality and utilizing Theorems \ref{Pdoc1_thm6} and \ref{Pdoc1_thm5}, the required error bound is derived as follows:
\begin{align*}
    \Big\|\mathcal{U}(x,y,t_{n+1})-\mathcal{U}_{M_1M_2}^{n+1}(x,y)\Big\|_{L^2(\overline{\Omega})}&\leq \Big\|\mathcal{U}(x,y,t_{n+1})-\mathcal{U}^{n+1}(x,y)\Big\|_{L^2(\overline{\Omega})}+\Big\|\mathcal{U}^{n+1}(x,y)-\mathcal{U}_{M_1M_2}^{n+1}(x,y)\Big\|_{L^2(\overline{\Omega})}\\
    &\leq \Big\|\mathcal{U}(x,y,t_{n+1})-\mathcal{U}^{n+1}(x,y)\Big\|_{\infty}+\Big\|\mathcal{U}^{n+1}(x,y)-\mathcal{U}_{M_1M_2}^{n+1}(x,y)\Big\|_{L^2(\overline{\Omega})}\\
    &\leq C\big(N^{-\min\{\nu\alpha,2\}}+\mathcal{M}^{-3}\big).
\end{align*}
\end{proof}
\begin{remark}\label{Pdoc1_remark2}
From Theorem \ref{Pdoc1_thm7}, it is observed that the wavelet-based discrete $L2$-$1_\sigma$ scheme achieves an optimal rate of convergence of $\mathcal{O}(N^{-2}+\mathcal{M}^{-3})$ in the space-time domain for $\nu\geq2/\alpha$. In particular, it leads to second-order temporal accuracy for that specific choice of the grading parameter. However, on uniformly distributed temporal mesh, the accuracy of the proposed scheme reduces to $\mathcal{O}(N^{-\alpha})$. In this case, it is noted that the spatial domain $\Omega\subset\mathbb{R}^2$.
\end{remark}
\begin{remark}\label{Pdoc1_remark3}
({\bf Non-smooth solution having $\mathbf{\mathcal{U}\in C(\overline{\Omega}_t)\cap C^3(\Omega_t)}$}) From Theorems \ref{Pdoc1_thm2} $\&$ \ref{Pdoc1_thm7} and the discussion in Remarks \ref{Pdoc1_remark1} $\&$ \ref{Pdoc1_remark2}, it is evident that the proposed scheme achieves a temporal accuracy of $\mathcal{O}(N^{-2})$ for the given TFIPDEs with an unbounded time derivative at $t=0$ if one takes the grading parameter $\nu\geq2/\alpha$. In contrast, the $L1$ scheme achieves at most $\mathcal{O}(N^{-(2-\alpha)})$ temporal accuracy. However, in this case, both the schemes lead to $\mathcal{O}(N^{-\alpha})$ accurate solution on a uniform mesh. The application of the $L1$ scheme to address time singularity in a fractional-order problem can be found in \cite{SantraHigherorder2023,StynesErroranalysis2017}.
\end{remark}
\begin{remark}\label{Pdoc1_remark4}
({\bf Smooth solution having $\mathbf{\mathcal{U}\in C^3(\overline{\Omega}_t)}$}) 
Given that the solution to the proposed problem is adequately smooth, with $\mathcal{U}\in C^3(\overline{\Omega}_t)$, the proposed approach achieves $\mathcal{O}(N^{-2})$ accuracy in time, even under uniform mesh discretization.  The analysis of the $L2$-$1_\sigma$ scheme on a time-fractional sub-diffusion equation with sufficiently smooth solution can be found in \cite{AlikhanovAnew2015}. In contrast, the $L1$ scheme leads to $\mathcal{O}(N^{-(2-\alpha)})$ temporal accuracy, which varies for different values of $\alpha$. In particular, it produces almost 1st-order convergent rate when $\alpha$ tends to one, whereas the proposed scheme achieves 2nd-order accuracy globally for all $\alpha\in(0,1)$. A widespread application of the $L1$ scheme to deal with the fractional-order problems without time singularity can be found in \cite{GaoAfinite2012,LinFinite2007} and references therein.
\end{remark}


\section{Results and discussion}\label{sec_example}
\noindent To show the effectiveness and high accuracy of the proposed method, in this section, we consider numerous test examples of one and two-dimensional TFIPDEs having known and unknown exact solutions with bounded and unbounded time derivatives at $t=0$ for special cases. The results are compared with the results obtained by the $L1$ scheme. Several tests are performed and the results are shown in the shape of figures and tables. All the experiments are done with $\sigma=1-\alpha/2$. The domains are defined as: $\Omega_t=(0,1]$, $\Omega=(0,1)$ if $\Omega\subset\mathbb{R}$ and $\Omega=(0,1)\times(0,1)$ if $\Omega\subset\mathbb{R}^2$.
\begin{example}\label{example1}
Let $\alpha\in(0,1)$. Consider the following TFIPDE that has a known exact solution with $\mathcal{U}\in C(\overline{\Omega}_t)/C^4(\overline{\Omega}_t)$ for special cases.
\begin{equation*}
    \left\{
    \begin{array}{ll}
    \partial_t^\alpha \mathcal{U}(x,t)-(1+x)\mathcal{U}_{xx}+\mathcal{U}(x,t)+\displaystyle{\int_0^t} \mathcal{K}(x,t-\xi)\mathcal{U}(x,\xi)d\xi=f(x,t),~(x,t)\in\Omega\times\Omega_t, \\[4pt]
    \mbox{with initial and boundary conditions:}\\[4pt]
    \mathcal{U}(x,0)=g(x)\;\; \mbox{for}~ x\in \overline{\Omega},~
    \mathcal{U}(0,t)=0,~ \mathcal{U}(1,t)=0\;\; \mbox{for}~ t\in \overline{\Omega}_t,
    \end{array}\right.
\end{equation*}
\end{example}
\noindent where the kernel $\mathcal{K}(x,t-\xi)=(t-\xi)\sin x$ and $g(x)=\sin 2\pi x$. The source term $f(x,t)$ is given by
\begin{align*}
    f(x,t)=&-\Gamma(\alpha+1)\sin 2\pi x +4\pi^2(1-t^{\alpha})(1+x)\sin 2\pi x+(1-t^{\alpha})\sin 2\pi x\\
  &+\Big(\dfrac{t^2}{2}-\dfrac{t^{\alpha+2}}{(\alpha+1)(\alpha+2)}\Big)\sin x\sin 2\pi x.
\end{align*}
The exact solution for Example \ref{example1} is then given by $\mathcal{U}(x,t)=(1-t^\alpha)\sin 2\pi x$. It can be noticed that the analytical behaviour of the present problem has an unbounded time derivative at the initial time $t=0$ due to which an initial layer occurs in the neighbourhood of $t=0$. See the impact of the fractional operator on the sharpness of the layer in Figure \ref{Ex1a_fig_fig1}. If $\{\mathcal{U}(x_m,t_n)\}_{m=0,n=0}^{M,N}$ denotes the exact solution and $\{\mathcal{U}_m^n\}_{m=0,n=0}^{M,N}$ is the numerical solution of Example \ref{example1} by using the proposed scheme (\ref{Pdoc1_12}), then the computed error $\widehat{E}_{M,N}$ and the corresponding temporal rate of convergence $\widehat{P}_{M,N}$ are estimated as: 
\begin{align}
\widehat{E}_{M,N}=\displaystyle{\max_{(x_m,t_n)\in\varOmega^{M,N}}}\Big\lvert\mathcal{U}(x_m,t_n)-\mathcal{U}_m^n\Big\rvert,~~\widehat{P}_{M,N}=\log_2\Big(\dfrac{\widehat{E}_{M,N}}{\widehat{E}_{M,2N}}\Big).\label{Pdoc1_18}
\end{align}
From Table \ref{Ex1a_table1}, it can be confirmed that for a non-uniform mesh with a suitably chosen grading parameter (here we fix $\nu=2/\alpha$), the proposed scheme can lead to a higher-order accuracy of $\mathcal{O}(N^{-2})$, which is superior to $L1$ scheme that gives $\mathcal{O}(N^{-(2-\alpha)})$ accuracy, as theoretically proved in Theorem \ref{Pdoc1_thm2} and further discussion in Remarks \ref{Pdoc1_remark1} \& \ref{Pdoc1_remark3}. In contrast, it reduces to $\mathcal{O}(N^{-\alpha})$ accuracy on a uniformly distributed mesh (i.e., for $\nu=1$) for the solution having initial time singularity for both schemes. See the computational results in Table \ref{Ex1a_table2}.

If we take $g(x)=\sin \pi x$ and the source function $f(x,t)$ is given by:
\begin{align*}
    f(x,t)=&\dfrac{\Gamma(\alpha+5)}{24}t^4\sin \pi x +\pi^2(1+t^{\alpha+4})(1+x)\sin \pi x+(1+t^{\alpha+4})\sin \pi x\\
    &+\Big(\dfrac{t^2}{2}+\dfrac{t^{\alpha+6}}{(\alpha+5)(\alpha+6)}\Big)\sin x\sin \pi x.
\end{align*}
Then, the exact solution for Example \ref{example1} is $\mathcal{U}(x,t)=(1+t^{\alpha+4})\sin \pi x$, which lies in $C^{4}(\overline{\Omega}_t)$. The proposed scheme leads to $\mathcal{O}(N^{-2})$ accurate solution even if the mesh is discretized uniformly (i.e., for $\nu=1$), as theoretically discussed in Remark \ref{Pdoc1_remark4}. Also, see the computational efficiency in Tables \ref{Ex1b_table1} \& \ref{Ex1b_table2}. It also highlights that the proposed scheme gives second-order accuracy for any value of $\alpha$ whereas, the $L1$ scheme exhibits almost first-order accuracy for $\alpha$ tends to one (see Table \ref{Ex1b_table2}). The theoretical convergence can also be confirmed from Figure \ref{Ex1b_fig_fig1}(a), which illustrates the maximum error for different values of $\alpha$ and $M=N$. Further, see the log-log plot in Figure \ref{Ex1b_fig_fig1}(b), which proves that the computational result matches the theoretical rate of convergence.
	

\begin{figure}[ht]
	\begin{subfigure}{.5\textwidth}
		\centering
		\includegraphics[width=0.95\linewidth]{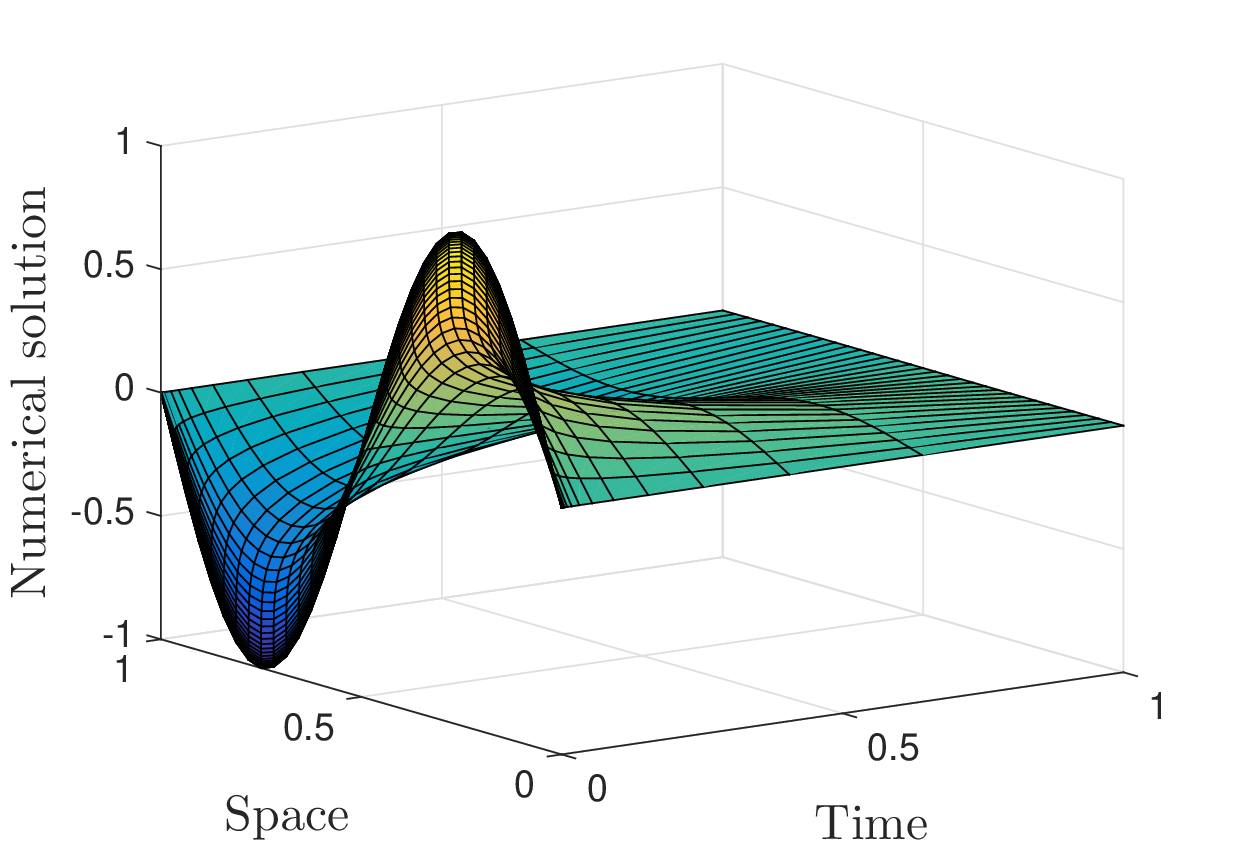}
		\caption{$\alpha=0.2$.}
		\label{Ex1a_Fig1}
	\end{subfigure}
	\begin{subfigure}{.5\textwidth}
		\centering
		\includegraphics[width=0.95\linewidth]{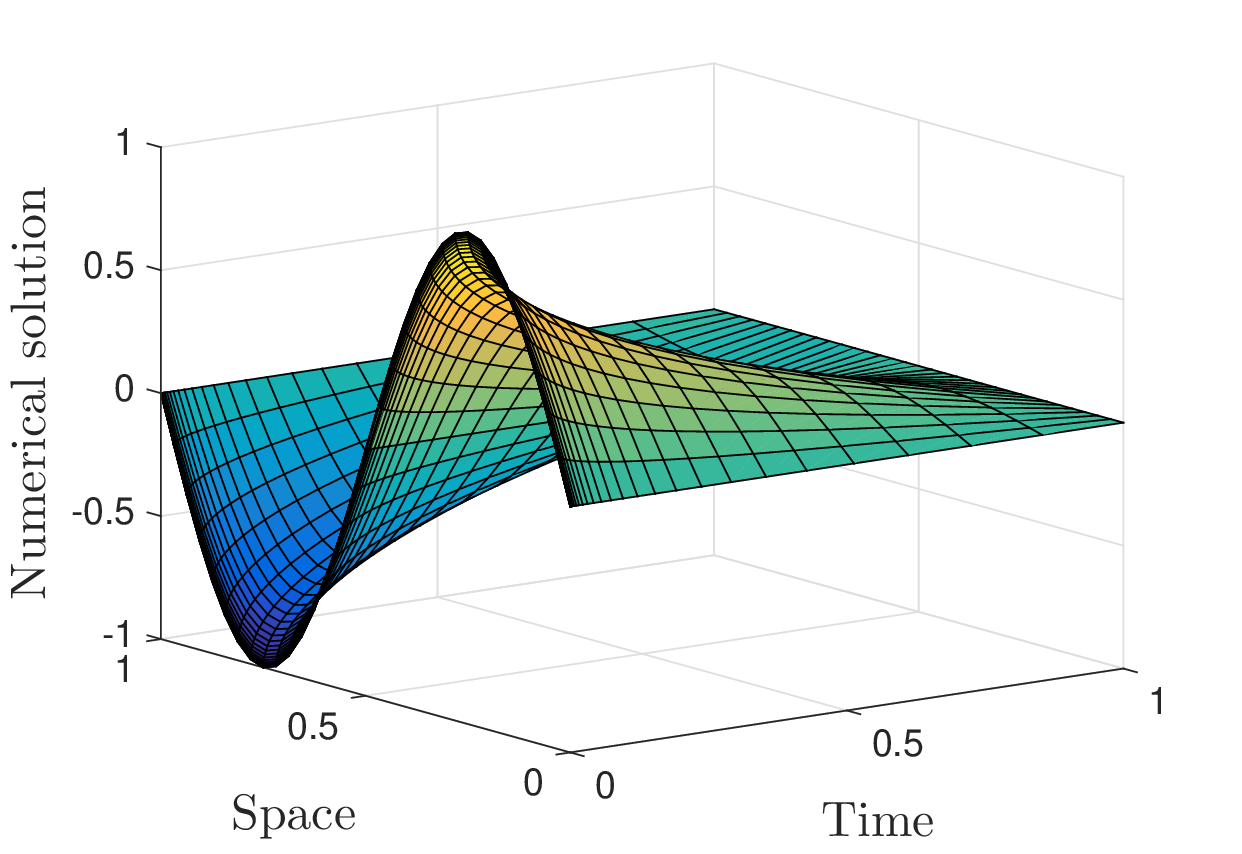}
		\caption{$\alpha=0.5$.}
		\label{Ex1a_Fig2}
	\end{subfigure}
	\caption{\small{Layer behaviour for Example \ref{example1} having solution with unbounded time derivatives at $t=0$ for $M=N=32$ on graded mesh with $\nu=(3-\alpha)/\alpha$.}}
	\label{Ex1a_fig_fig1}
\end{figure}


\begin{table}[ht]
	\caption {Maximum error (Max-Error) and rate of convergence (RoC) with fixed $M=1500$ for Example \ref{example1} with solution {\em having time singularity (graded mesh $\nu=2/\alpha$)}.}\label{Ex1a_table1}
	\begin{center}
		\begin{tabular}{llllllll}
			\hline
			$\alpha$ & $N$ & \multicolumn{2}{l}{$L2$-$1_\sigma$ scheme} &&& \multicolumn{2}{l}{$L1$ scheme} \\ \cline{3-8}
			&  & Max-Error & RoC &&& Max-Error & RoC \\ \hline
			\multirow{4}{*}{$0.2$} 
                &32    & 6.6577e-4  & 1.9444  &&& 4.1215e-4  & 1.4017 \\
			&64    & 1.7298e-4  & 1.9719  &&& 1.5599e-4  & 1.5489 \\ 	
			&128   & 4.4095e-5  & 1.9857  &&& 5.3313e-5  & 1.6024 \\ 
			&256   & 1.1134e-5  &         &&& 1.7558e-5  &        \\\hline
			\multirow{4}{*}{$0.4$} 
                &32    & 4.6167e-4  & 1.9797  &&& 5.5038e-4  & 1.3610 \\
			&64    & 1.1705e-4  & 1.9897  &&& 2.1427e-4  & 1.4626 \\ 	
			&128   & 2.9472e-5  & 1.9947  &&& 7.7747e-5  & 1.5057 \\ 
			&256   & 7.3949e-6  &         &&& 2.7379e-5  &        \\\hline
                \multirow{4}{*}{$0.8$} 
                &32    & 1.2055e-4  & 1.9957  &&& 4.7026e-4  & 1.0439 \\
			&64    & 3.0227e-5  & 1.9985  &&& 2.2808e-4  & 1.1156 \\ 	
			&128   & 7.5647e-6  & 1.9452  &&& 1.0526e-4  & 1.1482 \\ 
			&256   & 1.9644e-6  &         &&& 4.7492e-5  &        \\\hline
		\end{tabular}
	\end{center}
\end{table}



\begin{table}[ht]
	\caption {Maximum error (Max-Error) and rate of convergence (RoC) with fixed $M=1500$ for Example \ref{example1} with solution {\em having time singularity (uniform mesh $\nu=1$)}.}\label{Ex1a_table2}
	\begin{center}
		\begin{tabular}{llllllll}
			\hline
			$\alpha$ &  & \multicolumn{2}{l}{$L2$-$1_\sigma$ scheme} &&& \multicolumn{2}{l}{$L1$ scheme} \\ \cline{2-4} \cline{6-8}
			& $N$ & Max-Error & RoC && $N$ & Max-Error & RoC \\ \hline
			\multirow{4}{*}{$0.2$} 
                &4   & 6.4910e-2  & 0.2097  && 256  & 3.2807e-3  & 0.0193 \\
			&8   & 5.6130e-2  & 0.2083  && 512  & 3.2372e-3  & 0.0216 \\ 	
			&16  & 4.8584e-2  & 0.2089  && 1024 & 3.1891e-3  & 0.0241 \\ 
			&32  & 4.2035e-2  &         && 2048 & 3.1362e-3  &        \\\hline
			\multirow{4}{*}{$0.4$} 
			&4   & 7.9125e-2  & 0.4206  && 256  & 4.1287e-3  & 0.0953 \\ 	
			&8   & 5.9113e-2  & 0.4263  && 512  & 3.8647e-3  & 0.1134 \\ 
			&16  & 4.3990e-2  & 0.4350  && 1024 & 3.5724e-3  & 0.1334 \\
                &32  & 3.2538e-2  &         && 2048 & 3.2568e-3  &        \\\hline
                \multirow{4}{*}{$0.8$} 
			&4   & 3.2528e-2  & 0.9001  && 256  & 1.1142e-3  & 0.5693 \\ 	
			&8   & 1.7430e-2  & 0.9768  && 512  & 7.5088e-4  & 0.6466 \\ 
			&16  & 8.8561e-3  & 1.1075  && 1024 & 4.7964e-4  & 0.6875 \\
                &32  & 4.1100e-3  &         && 2048 & 2.9782e-4  &        \\\hline
		\end{tabular}
	\end{center}
\end{table}


\begin{figure}[ht]
	\begin{subfigure}{.5\textwidth}
		\centering
		\includegraphics[width=0.95\linewidth]{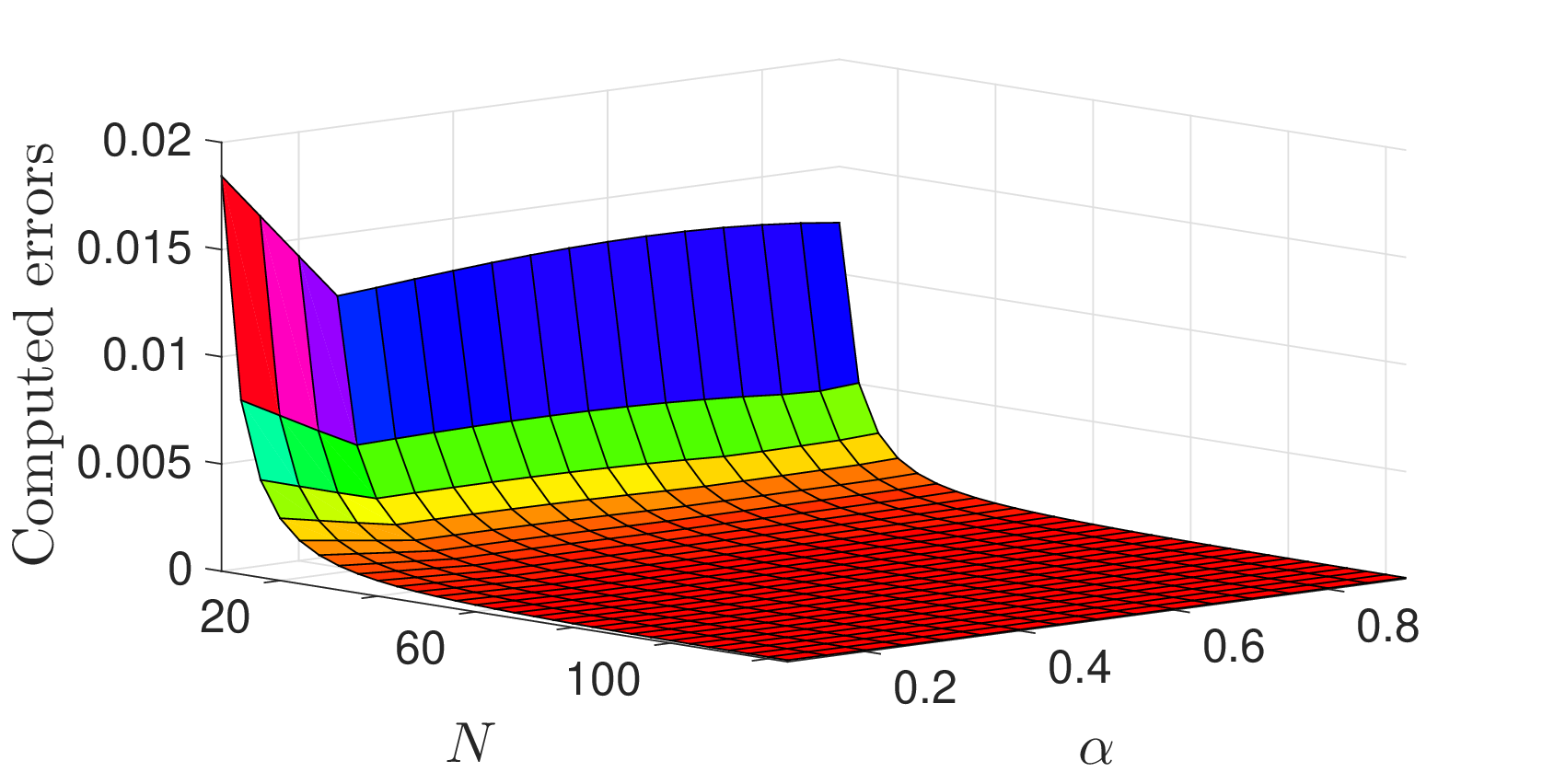}
		\caption{Computed errors.}
		\label{Ex1b_Fig1}
	\end{subfigure}
	\begin{subfigure}{.5\textwidth}
		\centering
		\includegraphics[width=0.95\linewidth]{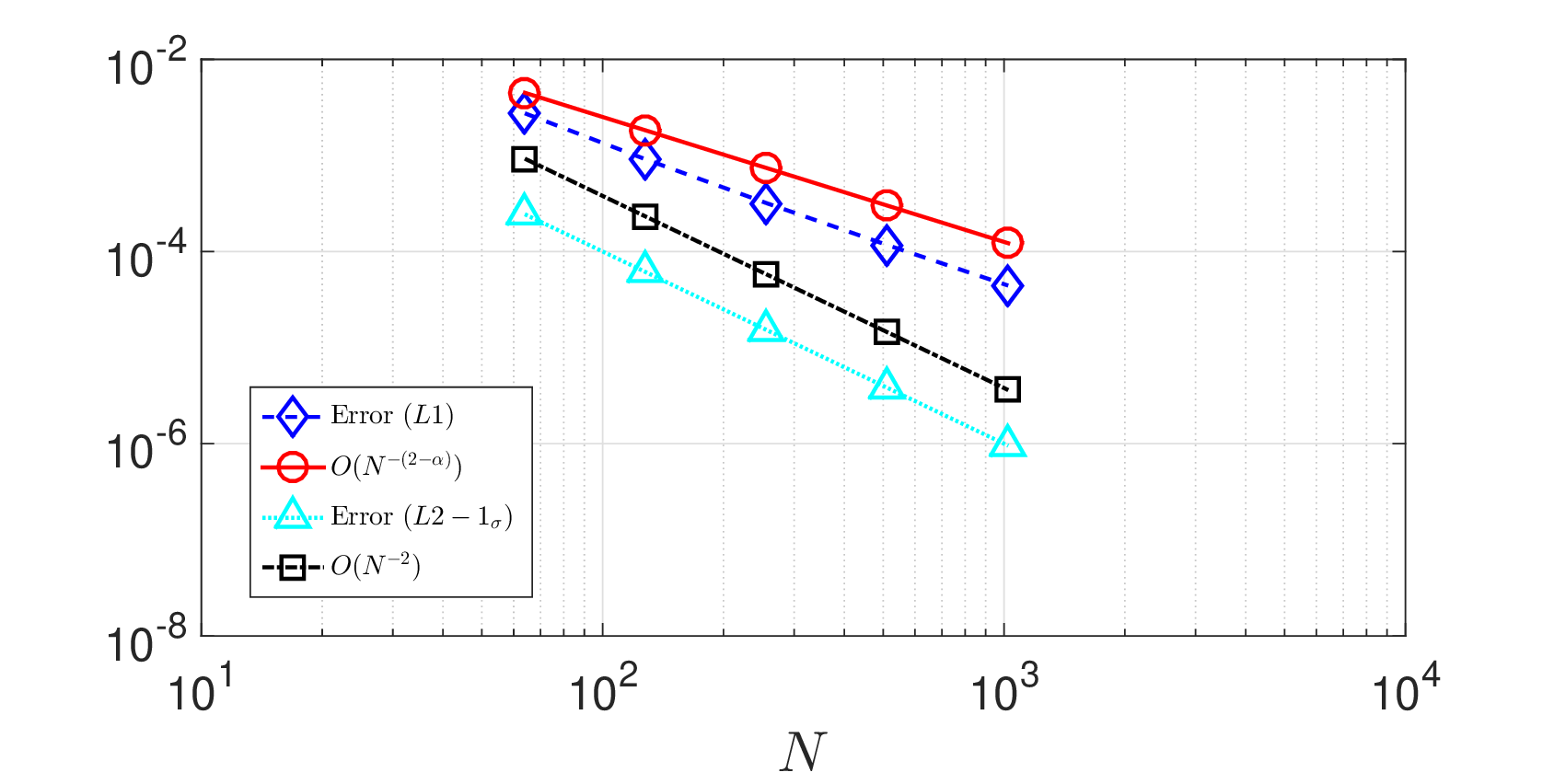}
		\caption{Log-log plot with $\alpha=0.7$.}
		\label{Ex1b_Fig2}
	\end{subfigure}
	\caption{\small{Error and log-log plots for Example \ref{example1}.}}
	\label{Ex1b_fig_fig1}
\end{figure}
	
 
\begin{table}[ht]
	\caption {Maximum error (Max-Error) and rate of convergence (RoC) based on $L2$-$1_\sigma$ scheme with fixed $M=1000$ for Example \ref{example1} having solution $\mathcal{U}\in C^4(\overline{\Omega}_t)$ {\em(uniform mesh $\nu=1$)}.}\label{Ex1b_table1}
	\begin{center}
		\begin{tabular}{lllllllll}
			\hline
			$N$ & \multicolumn{2}{l}{$\alpha=0.2$} && \multicolumn{2}{l}{$\alpha=0.5$}  && \multicolumn{2}{l}{$\alpha=0.8$} \\ \cline{2-9}
			 & Max-Error & RoC && Max-Error & RoC && Max-Error & RoC \\ \hline
			16   & 2.1106e-3  & 1.9787 && 5.1199e-3 & 1.9856 && 7.5374e-3 & 1.9952 \\ 
			32   & 5.3550e-4  & 2.0001 && 1.2928e-3 & 1.9987 && 1.8907e-3 & 2.0032 \\ 	
			64   & 1.3387e-4  & 2.0414 && 3.2351e-4 & 2.0171 && 4.7161e-4 & 2.0150 \\ 
			128  & 3.2521e-5  & 2.2008 && 7.9926e-5 & 2.0782 && 1.1668e-4 & 2.0551 \\ 
			256  & 7.0739e-6  &        && 1.8928e-5 &        && 2.8078e-5 &        \\ \hline
		\end{tabular}
	\end{center}
\end{table}


\begin{table}[ht]
	\caption {Maximum error (Max-Error) and rate of convergence (RoC) with fixed $M=1000$ for Example \ref{example1} having solution $\mathcal{U}\in C^4(\overline{\Omega}_t)$ {\em (uniform mesh $\nu=1$)}.}\label{Ex1b_table2}
	\begin{center}
		\begin{tabular}{llllllll}
			\hline
			$\alpha$ & $N$ & \multicolumn{2}{l}{$L2$-$1_\sigma$ scheme} &&& \multicolumn{2}{l}{$L1$ scheme} \\ \cline{3-8}
			&  & Max-Error & RoC &&& Max-Error & RoC \\ \hline
			\multirow{5}{*}{$0.3$} 
			&16   & 3.1525e-3  & 1.9803  &&& 8.3736e-4  & 1.5728 \\ 	
			&32   & 7.9896e-4  & 1.9978  &&& 2.8149e-4  & 1.5959 \\ 
			&64   & 2.0004e-4  & 2.0266  &&& 9.3123e-5  & 1.5874 \\ 
			&128  & 4.9095e-5  & 2.1291  &&& 3.0989e-5  & 1.5138 \\ 
			&256  & 1.1223e-5  &         &&& 1.0852e-5  &        \\ \hline
			\multirow{5}{*}{$0.9$} 
			&16   & 8.1421e-3  & 1.9975  &&& 2.0193e-2  & 1.0539 \\ 	
			&32   & 2.0391e-3  & 2.0038  &&& 9.7258e-3  & 1.0757 \\ 
			&64   & 5.0844e-4  & 2.0140  &&& 4.6143e-3  & 1.0869 \\ 
			&128  & 1.2588e-4  & 2.0505  &&& 2.1723e-3  & 1.0923 \\ 
			&256  & 3.0387e-5  &         &&& 1.0188e-3  &        \\ \hline
		\end{tabular}
	\end{center}
\end{table}


\begin{example}\label{example2}
Consider another test problem having a known exact solution with $\mathcal{U}\in C^3(\overline{\Omega}_t)$.
\begin{equation*}
    \left\{
    \begin{array}{ll}
    \partial_t^\alpha \mathcal{U}(x,t)-\mathcal{U}_{xx}+\displaystyle{\int_0^t} x(t-\xi)\mathcal{U}(x,\xi)d\xi=f(x,t),~(x,t)\in\Omega\times\Omega_t, \\[4pt]
    \mbox{with initial and boundary conditions:}\\[4pt]
    \mathcal{U}(x,0)=0\;\; \mbox{for}~ x\in \overline{\Omega},~
    \mathcal{U}(0,t)=t+t^{\alpha+3},~ \mathcal{U}(1,t)=0\;\; \mbox{for}~ t\in \overline{\Omega}_t,
    \end{array}\right.
\end{equation*}
\end{example}
\noindent where $\alpha\in(0,1)$. $f(x,t)$ is given by: 
\[f(x,t)=(1-x^2)\Big(\dfrac{t^{1-\alpha}}{\Gamma(2-\alpha)}+\dfrac{1}{6}\Gamma(\alpha+4)t^3\Big)+2(t+t^{\alpha+3})+x(1-x^2)\Big(\dfrac{t^3}{6}+\dfrac{t^{\alpha+5}}{(\alpha+4)(\alpha+5)}\Big).\] 
The exact solution for Example \ref{example2} is $\mathcal{U}(x,t)=(1-x^2)(t+t^{\alpha+3})$. Notice that the solution of the present problem lies in $C^{3}(\overline{\Omega}_t)$. Here, we calculate the error and the rate of convergence in a similar way as defined in 
(\ref{Pdoc1_18}). For the present scenario, it can be observed that the proposed scheme (\ref{Pdoc1_12}) not only gives higher order accuracy (see log-log plot in Figure \ref{Ex2_fig_fig1}(b)) but also the obtained errors are less (see maximum error plot in Figure \ref{Ex2_fig_fig1}(a)) compared to the $L1$ scheme even if the mesh is discretized uniformly. The computational efficiency can also be confirmed from Table \ref{Ex2_table1}, which demonstrates that the proposed scheme can lead to a higher-order accuracy in time possibly $\mathcal{O}(N^{-2})$ for a sufficiently smooth solution $\mathcal{U}\in C^3(\overline{\Omega}_t)$.


\begin{figure}[ht]
	\begin{subfigure}{.5\textwidth}
		\centering
		\includegraphics[width=0.95\linewidth]{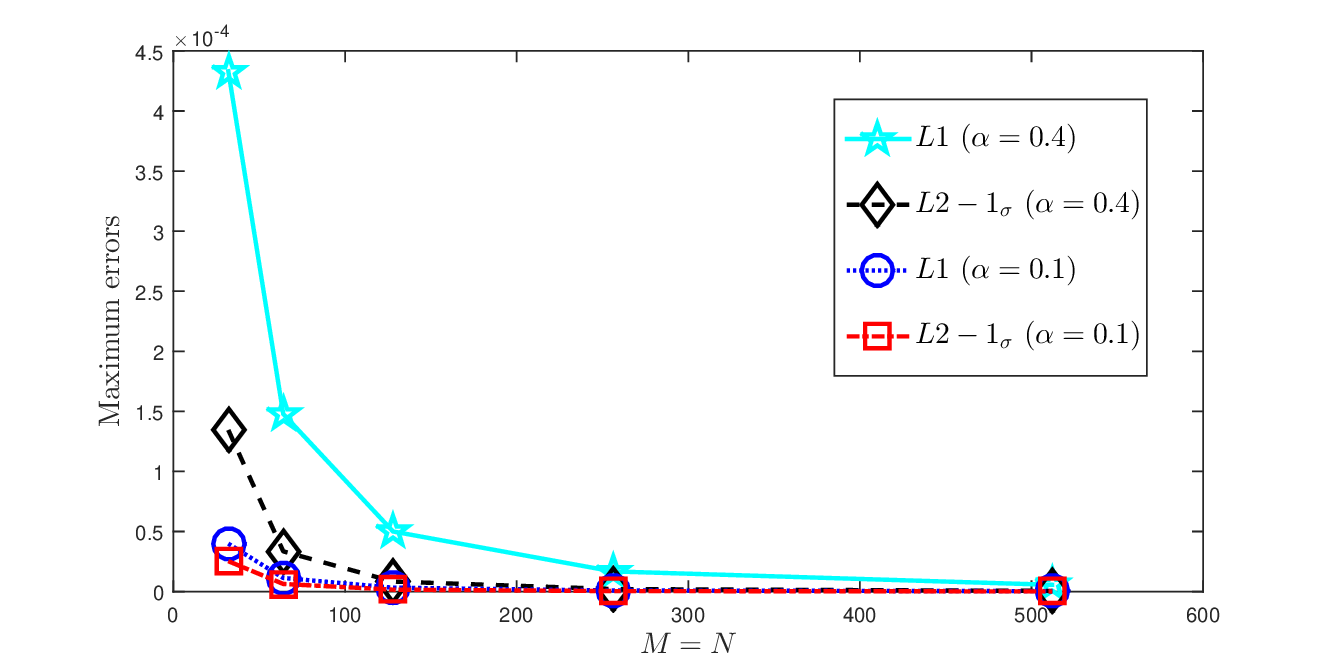}
		\caption{Maximum errors.}
		\label{Ex2_Fig1}
	\end{subfigure}
	\begin{subfigure}{.5\textwidth}
		\centering
		\includegraphics[width=0.95\linewidth]{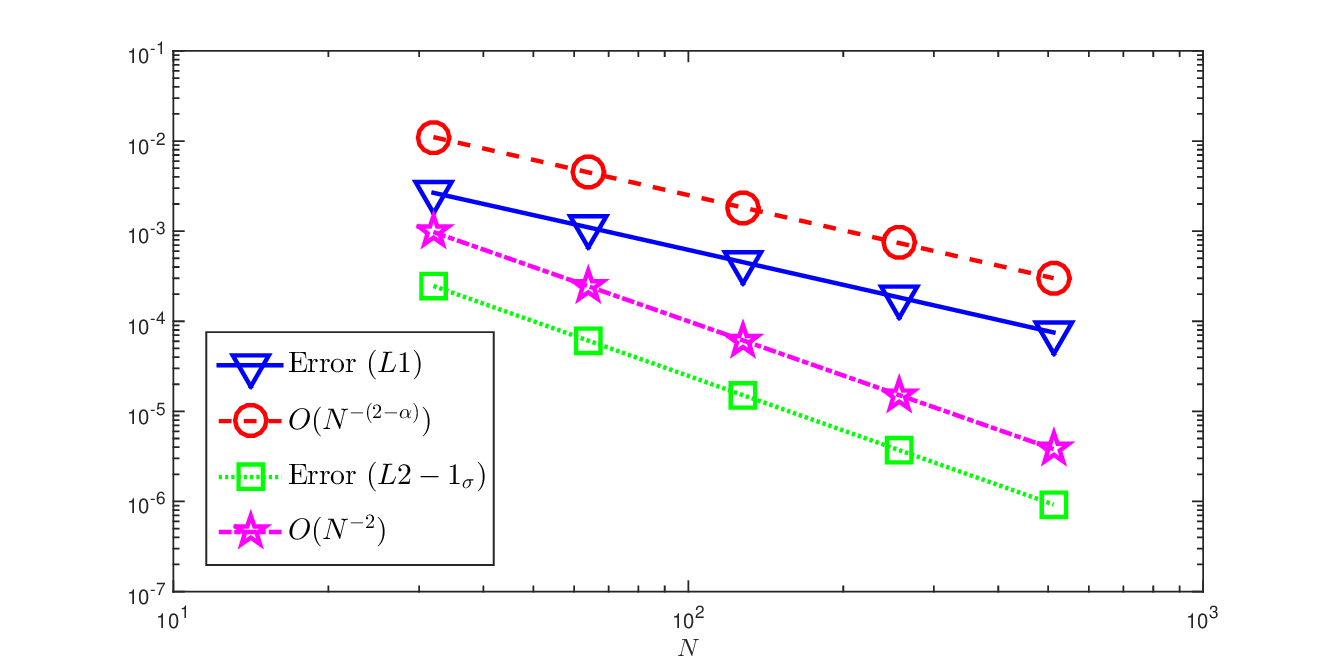}
		\caption{Log-log plot with $\alpha=0.7$.}
		\label{Ex2_Fig2}
	\end{subfigure}
	\caption{\small{Comparison between $L1$ and $L2$-$1_\sigma$ schemes for Example \ref{example2}.}}
	\label{Ex2_fig_fig1}
\end{figure}


\begin{table}[ht]
{\small
	\caption {Maximum error (Max-Error) and rate of convergence (RoC) with fixed $M=1000$ for Example \ref{example2} having solution $\mathcal{U}\in C^3(\overline{\Omega}_t)$ {\em (uniform mesh $\nu=1$)}.}\label{Ex2_table1}
	\begin{center}
		\begin{tabular}{llllllll}
			\hline
			$\alpha$ & $N$ & \multicolumn{2}{l}{$L2$-$1_\sigma$ scheme} &&& \multicolumn{2}{l}{$L1$ scheme} \\ \cline{3-8}
			&  & Max-Error & RoC &&& Max-Error & RoC \\\hline
			\multirow{5}{*}{$0.2$} 
			&32   & 6.0786e-5  & 2.0014  &&& 1.0164e-4  & 1.7219 \\ 
			&64   & 1.5182e-5  & 2.0020  &&& 3.0811e-5  & 1.7361 \\ 	
			&128  & 3.7901e-6  & 2.0018  &&& 9.2490e-6  & 1.7469 \\ 
			&256  & 9.4637e-7  & 2.0012  &&& 2.7556e-6  & 1.7555 \\ 
			&512  & 2.3639e-7  &         &&& 8.1611e-7  &        \\ \hline
			\multirow{5}{*}{$0.5$} 
			&32   & 1.7173e-4  & 2.0094  &&& 8.1771e-4  & 1.4612 \\ 
			&64   & 4.2654e-5  & 2.0086  &&& 2.9699e-4  & 1.4740 \\ 	
			&128  & 1.0600e-5  & 2.0071  &&& 1.0691e-4  & 1.4823 \\ 
			&256  & 2.6369e-6  & 2.0055  &&& 3.8264e-5  & 1.4879 \\ 
			&512  & 6.5669e-7  &         &&& 1.3643e-5  &        \\ \hline
			\multirow{5}{*}{$0.8$} 
			&32   & 2.8372e-4  & 2.0160  &&& 4.6527e-3  & 1.1799 \\ 
			&64   & 7.0147e-5  & 2.0159  &&& 2.0537e-3  & 1.1889 \\ 	
			&128  & 1.7345e-5  & 2.0149  &&& 9.0082e-4  & 1.1939 \\ 
			&256  & 4.2918e-6  & 2.0135  &&& 3.9378e-4  & 1.1966 \\ 
			&512  & 1.0629e-6  &         &&& 1.7181e-4  &        \\ \hline
		\end{tabular}
	\end{center}}
\end{table}


Below, we present another example of the type (\ref{Pdoc1_1}) for which the exact solution is unknown. The double mesh principle \cite{SantraAnalysis2021} is used to calculate the error and the order of convergence, which works as follows: 
Suppose $\mathcal{U}_m^n$ be the numerical solution for $m=0,1,\ldots, M$ and $n=0,1,\ldots, N$. Now consider a fine mesh $\big\{(x_{m},t_{n/2})\;\lvert\;m=0,1,\ldots, M \mbox{ and}\; n=0,1,\ldots, 2N\big\}$, and let $\mathcal{V}_{m}^{n/2}$ be the corresponding computed solution by using the same scheme. Then the error and the corresponding temporal rate of convergence are estimated as:
\begin{align}
    \tilde{E}_{M,N}:=\max_{(x_m,t_n)\in\varOmega^{M,N}}\Big\lvert\mathcal{U}_m^n-\mathcal{V}_m^n\Big\rvert,~\tilde{P}_{M,N}:=\log_2\Bigg(\dfrac{\tilde{E}_{M,N}}{\tilde{E}_{M,2N}}\Bigg).\label{Pdoc1_19}
\end{align}
\begin{example}\label{example3}
Consider the following TFIPDE with an unknown exact solution.
\begin{equation*}
\left\{
    \begin{array}{ll}
    \partial_t^\alpha \mathcal{U}(x,t)-\mathcal{U}_{xx}+x\mathcal{U}(x,t)+\displaystyle{\int_0^t} \mathcal{K}(x,t-\xi)\mathcal{U}(x,\xi)d\xi=xt^{\alpha},~(x,t)\in\Omega\times\Omega_t, \\[4pt]
    \mbox{with initial and boundary conditions:}\\[4pt]
    \mathcal{U}(x,0)=0\;\; \mbox{for}~ x\in \overline{\Omega},~
    \mathcal{U}(0,t)=\mathcal{U}(1,t)=0\;\; \mbox{for}~ t\in \overline{\Omega}_t,
    \end{array}\right.
\end{equation*}
\end{example}
\noindent where $\alpha\in(0,1)$, and the kernel $\mathcal{K}=e^{x(t-\xi)},~\xi\in[0,t],~(x,t)\in\overline{\Omega}\times\overline{\Omega}_t$.
The nature of the solution to this example is displayed graphically in Figure \ref{Ex3_fig_fig1}. It clearly indicates that the solution of such fractional-order integro-partial differential equation has a mild singularity at $t=0$ for which the graded mesh-based $L2$-$1_\sigma$ scheme will be more effective than the uniform mesh. Since the exact solution is unknown, we calculate the error and the rate of convergence by using the formula depicted in (\ref{Pdoc1_19}), and the data are tabulated in Table \ref{Ex3_table1}. It can be confirmed that the proposed scheme leads to a higher-order accuracy possibly $\mathcal{O}(N^{-2})$ on a non-uniform mesh with a suitably chosen grading parameter, in this case, it is $\nu\geq2/\alpha$. See the theoretical arguments given in Theorem \ref{Pdoc1_thm2} and further elaboration in Remark \ref{Pdoc1_remark1}. In contrast, it reduces to $\mathcal{O}(N^{-\alpha})$ accurate solution on a uniformly distributed mesh, but, computationally, we are getting a better approximation for higher values of $\alpha$.


 \begin{figure}[ht]
	\begin{subfigure}{.5\textwidth}
		\centering
		\includegraphics[width=0.95\linewidth]{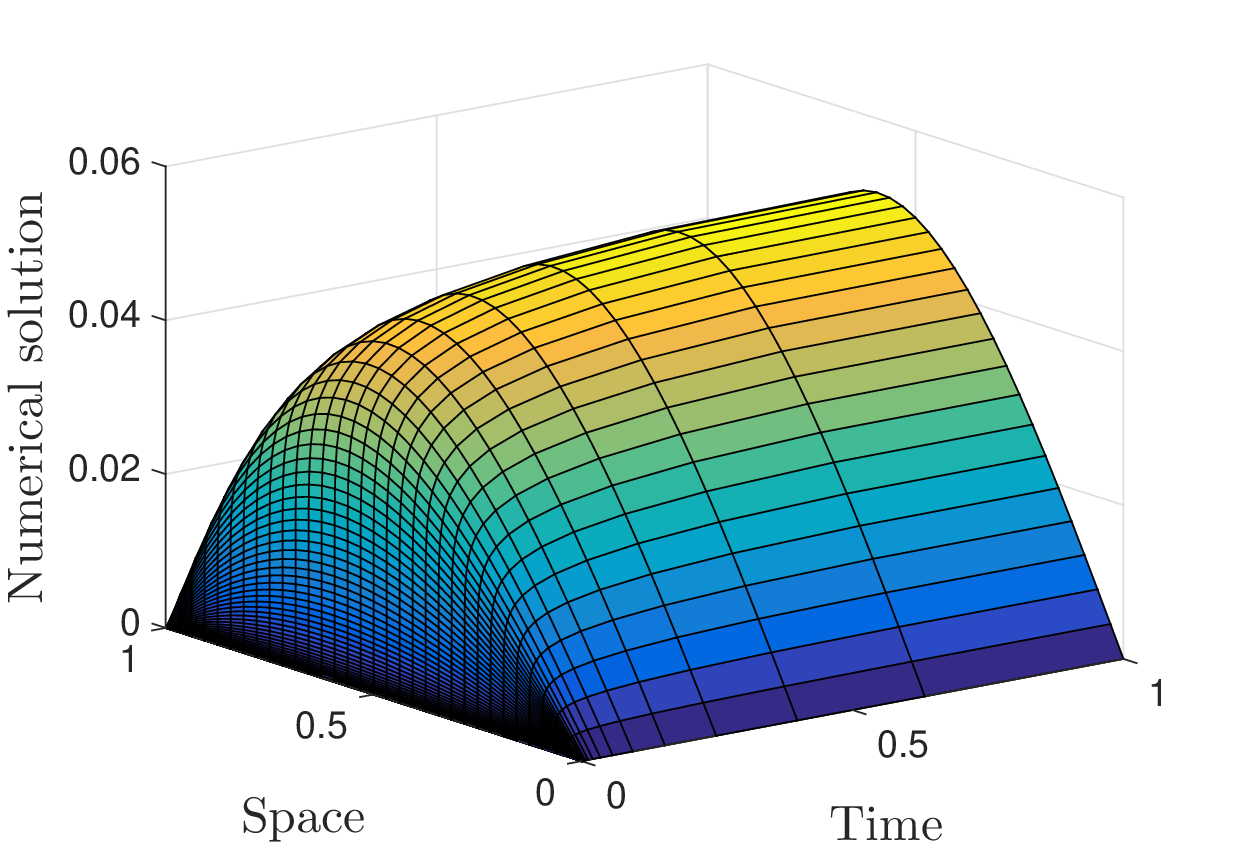}
		\caption{$\alpha=0.1$.}
		\label{Ex3_Fig1}
	\end{subfigure}
	\begin{subfigure}{.5\textwidth}
		\centering
		\includegraphics[width=0.95\linewidth]{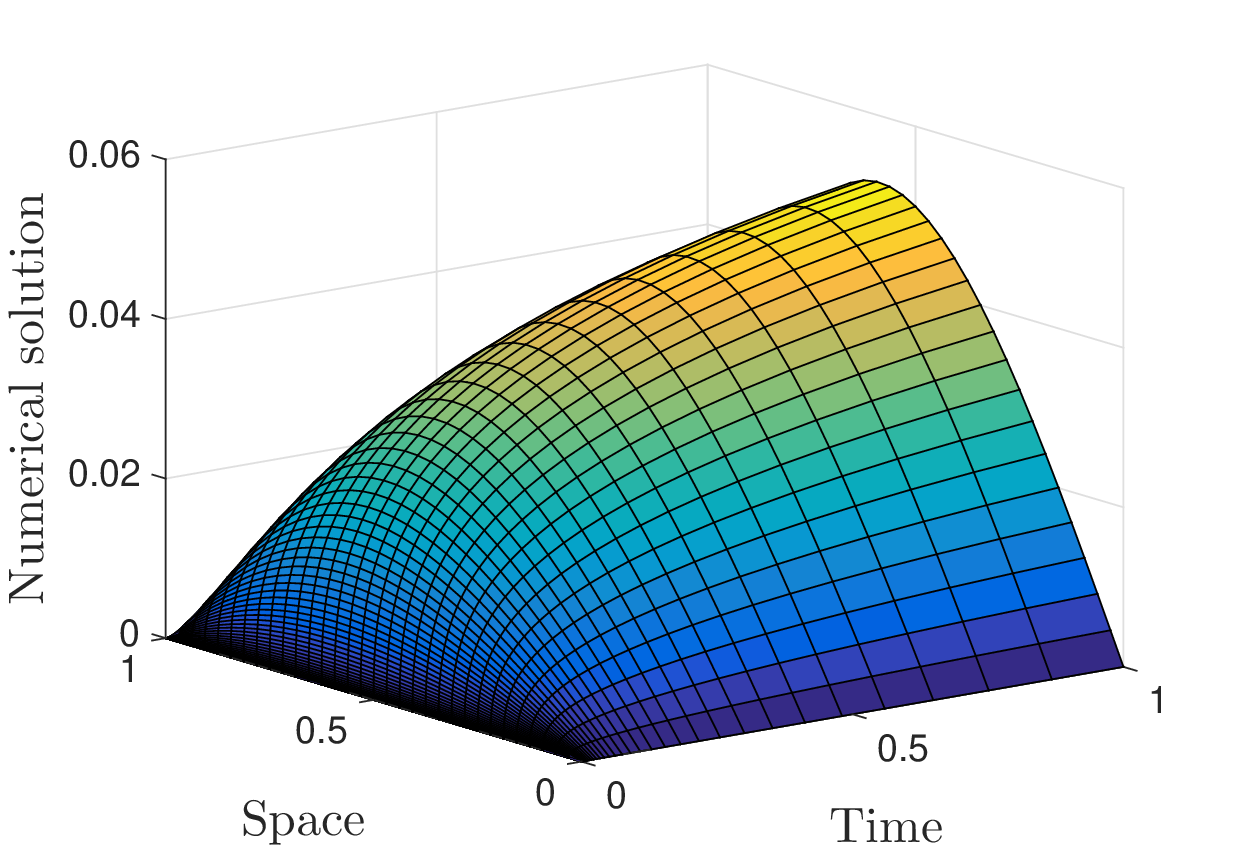}
		\caption{$\alpha=0.3$.}
		\label{Ex3_Fig2}
	\end{subfigure}
	\caption{\small{Layer behaviour for Example \ref{example3} with $M=32,~N=64$.}}
	\label{Ex3_fig_fig1}
\end{figure}
    

\begin{table}[ht]
	\caption {Maximum error (Max-Error) and rate of convergence (RoC) with fixed $M=1000$ based on $L2$-$1_\sigma$ scheme for Example \ref{example3}.}\label{Ex3_table1}
	\begin{center}
		\begin{tabular}{llllllllllll}
			\hline
			$\alpha$ & $N$ & \multicolumn{6}{c}{Graded mesh} &&& \multicolumn{2}{c}{Uniform mesh} \\ \cline{3-8} \cline{11-12}
                         &     & \multicolumn{2}{c}{$\nu=2/\alpha$} &&& \multicolumn{2}{c}{$\nu=(3-\alpha)/\alpha$} &&& \multicolumn{2}{c}{$\nu=1$}\\\cline{3-8} \cline{11-12}
			&  & Max-Error & RoC &&& Max-Error & RoC &&& Max-Error & RoC \\ \hline
			\multirow{4}{*}{$0.2$}  
			&32  & 4.0357e-5  & 1.9063  &&& 7.3962e-5  & 1.8646 &&& 1.6517e-3  & 0.3105 \\ 	
			&64  & 1.0766e-5  & 1.9512  &&& 2.0309e-5  & 1.9289 &&& 1.3318e-3  & 0.3193 \\ 
			&128 & 2.7842e-6  & 1.9746  &&& 5.3337e-6  & 1.9631 &&& 1.0674e-3  & 0.3303 \\ 
                &256 & 7.0839e-7  &         &&& 1.3680e-6  &        &&& 8.4893e-4  &        \\ \hline
			\multirow{4}{*}{$0.4$} 
			&32  & 2.6764e-5  & 1.9634  &&& 4.3962e-5  & 1.9484 &&& 7.0826e-4  & 0.7889 \\ 	
			&64  & 6.8628e-6  & 1.9808  &&& 1.1391e-5  & 1.9734 &&& 4.0992e-4  & 0.8172 \\ 
			&128 & 1.7386e-6  & 1.9897  &&& 2.9008e-6  & 1.9859 &&& 2.3264e-4  & 0.8325 \\ 
                &256 & 4.3778e-7  &         &&& 7.3234e-7  &        &&& 1.3064e-4  &        \\ \hline
			\multirow{4}{*}{$0.6$} 
			&32  & 1.6741e-5  & 1.9812  &&& 2.3719e-5  & 1.9757 &&& 1.2291e-4  & 1.2085 \\ 	
			&64  & 4.2403e-6  & 1.9897  &&& 6.0307e-6  & 1.9870 &&& 5.3184e-5  & 1.1834 \\ 
			&128 & 1.0677e-6  & 1.9940  &&& 1.5213e-6  & 1.9926 &&& 2.3417e-5  & 1.1753 \\ 
                &256 & 2.6804e-7  &         &&& 3.8228e-7  &        &&& 1.0369e-5  &        \\ \hline
		\end{tabular}
	\end{center}
\end{table}

The following example shows the strong reliability of the wavelet-based numerical approximation combined with graded mesh-based $L2$-$1_\sigma$ scheme in the context of the given TFIPDEs (\ref{Pdoc1_1}) with $\Omega\subset\mathbb{R}^2$, whose solution has an unbounded time derivative at $t=0$. Also, note that the convergence is unaltered for a sufficiently smooth solution even if the mesh is discretized uniformly.  The maximum error $\mathscr{E}_N^{M_1M_2}$ ($L^\infty$-error) and the $L^2$-error $\widetilde{\mathscr{E}}_N^{M_1M_2}$ can be estimated as:
\begin{equation}\label{Pdoc1_49}
\left\{
\begin{array}{ll}
    \mathscr{E}_N^{M_1M_2}=&\displaystyle{\max_{1\leq n\leq N}\max_{1\leq l_1\leq 2M_1}\max_{1\leq l_2\leq 2M_2}}\Big\lvert\mathcal{U}(x_{l_1},y_{l_2},t_n)-\mathcal{U}_{M_1M_2}^n(x_{l_1},y_{l_2})\Big\rvert,\\[6pt]
    \widetilde{\mathscr{E}}_N^{M_1M_2}=&\Bigg[\dfrac{1}{2M_1\times 2M_2\times (N+1)}\displaystyle{\sum_{l_1=1}^{2M_1}\sum_{l_2=1}^{2M_2}\sum_{n=0}^{N}}\Big\lvert\mathcal{U}(x_{l_1},y_{l_2},t_n)-\mathcal{U}_{M_1M_2}^n(x_{l_1},y_{l_2})\Big\rvert^2\Bigg]^{1/2}.
\end{array}\right.    
\end{equation} 
The corresponding temporal rate of convergence can be calculated by the usual way as:
\begin{align}
    \mathscr{R}_N^{M_1M_2}=\log_2\Bigg(\dfrac{\mathscr{E}_N^{M_1M_2}}{\mathscr{E}_{2N}^{M_1M_2}}\Bigg),~\widetilde{\mathscr{R}}_N^{M_1M_2}=\log_2\Bigg(\dfrac{\widetilde{\mathscr{E}}_N^{M_1M_2}}{\widetilde{\mathscr{E}}_{2N}^{M_1M_2}}\Bigg).\label{Pdoc1_50}
\end{align} 
 \begin{example}\label{example4}
Consider the following two-dimensional TFIPDE having known exact solution with bounded/unbounded time derivative at $t=0$.
\begin{equation*}
    \left\{
    \begin{array}{ll}
    \partial_t^\alpha \mathcal{U}(x,y,t)-\mathcal{U}_{xx}-\mathcal{U}_{yy}+\mathcal{U}_{x}+\displaystyle{\int_0^t} xy(t-\xi)\mathcal{U}(x,y,\xi)d\xi=f(x,y,t),~(x,y,t)\in\Omega\times\Omega_t, \\[4pt]
    \mbox{with initial and boundary conditions:}\\[4pt]
    \mathcal{U}(x,y,0)=xy(x-1)(y-1)\;\; \mbox{for}~ (x,y)\in \overline{\Omega},\\
    \mathcal{U}(0,y,t)=\mathcal{U}(1,y,t)=0\;\; \mbox{for}~ (y,t)\in [0,1]\times\overline{\Omega}_t,\\
    \mathcal{U}(x,0,t)=\mathcal{U}(x,1,t)=0\;\; \mbox{for}~ (x,t)\in [0,1]\times\overline{\Omega}_t,
    \end{array}\right.
\end{equation*}
\end{example}  
\noindent where $\alpha\in(0,1)$. The source function $f$ is given by:
 \begin{align*}
     f(x,y,t)=&\Gamma(\alpha+1)xy(x-1)(y-1)-2(1+t^{\alpha})(x^2+y^2-x-y)\\
     &+(1+t^{\alpha})(2x-1)(y^2-y)+\Big(\dfrac{t^2}{2}+\dfrac{t^{\alpha+2}}{(\alpha+1)(\alpha+2)}\Big)x^2y^2(x-1)(y-1).
\end{align*}
The exact solution for Example \ref{example4} is then $\mathcal{U}=(1+t^{\alpha})xy(x-1)(y-1)$. To solve the model, the graded mesh-based $L2$-$1_\sigma$ scheme is used to make it in a semi-discrete form as described in (\ref{Pdoc1_25}). Then, the two-dimensional Haar wavelet is applied to approximate the solution at each time level. Notice that the solution has an unbounded time derivative at $t=0$, which reduces the convergence rate on a uniformly distributed temporal mesh (see Table \ref{Ex4a_table2}). In contrast, the non-uniform mesh with a suitably chosen grading parameter ($\nu\geq2/\alpha$) leads to a higher-order accuracy, possibly $\mathcal{O}(N^{-2})$, which is evident from Table \ref{Ex4a_table1} illustrating the error and the rate of convergence based on $L^2$- \& $L^\infty$-norm. Also, see the theoretical explanation in Theorem \ref{Pdoc1_thm7} and further discussion in Remarks \ref{Pdoc1_remark2} \& \ref{Pdoc1_remark3}.
 
If the source function is chosen as:
\begin{align*}
    f(x,y,t)=&\dfrac{1}{6}\Gamma(\alpha+4)xy(x-1)(y-1)t^3-2(1+t^{\alpha+3})(x^2+y^2-x-y)\\
     &+(1+t^{\alpha+3})(2x-1)(y^2-y)+\Big(\dfrac{t^2}{2}+\dfrac{t^{\alpha+5}}{(\alpha+4)(\alpha+5)}\Big)x^2y^2(x-1)(y-1).
\end{align*}
Then, the exact solution for Example \ref{example4} is given by $\mathcal{U}=(1+t^{\alpha+3})xy(x-1)(y-1)$, which lies in $C^{3}(\overline{\Omega}_t)$. In this scenario, the proposed scheme can lead to second-order accuracy even if the temporal mesh is discretized uniformly, as theoretically discussed in Remark \ref{Pdoc1_remark4}. See the computational results in Table \ref{Ex4b_table1}.


\begin{table}[ht]
\caption {Error ($L^2$ \& $L^\infty$) and rate of convergence (RoC) based on $L2$-$1_\sigma$ scheme with $2M_1=2M_2=32$ for Example \ref{example4} with solution {\em having time singularity (graded mesh $\nu=2/\alpha,~(3-\alpha)/\alpha$)}.}\label{Ex4a_table1}
	\begin{center}
		\begin{tabular}{llllllllllllll}
			\hline
			$\alpha$ & $N$ & \multicolumn{5}{c}{$\nu=2/\alpha$} &&& \multicolumn{5}{c}{$\nu=(3-\alpha)/\alpha$} \\ \cline{3-7}\cline{10-14}
                 &  & $L^2$-error & RoC    && $L^\infty$-error & RoC &&& $L^2$-error & RoC    && $L^\infty$-error & RoC\\ \hline
			\multirow{4}{*}{$0.25$}  
			&5  & 3.3418e-4 & 1.6510 && 1.0641e-3 & 1.7221 &&& 4.4063e-4  & 1.6093  && 1.6585e-3  & 1.5882 \\ 	
			&10 & 1.0641e-4 & 1.7902 && 3.2253e-4 & 1.8575 &&& 1.4443e-4  & 1.7757  && 5.5160e-4  & 1.7900 \\ 
			&20 & 3.0767e-5 & 1.8740 && 8.9004e-5 & 1.9247 &&& 4.2179e-5  & 1.8723  && 1.5951e-4  & 1.8908 \\ 
                &40 & 8.3936e-6 &        && 2.3443e-5 &        &&& 1.1520e-5  &         && 4.3013e-5  &        \\ \hline
			\multirow{4}{*}{$0.55$} 
			&5  & 2.4785e-4 & 1.8876 && 7.0088e-4 & 1.9165 &&& 3.0452e-4  & 1.8913  && 9.9772e-4  & 1.8840 \\ 	
			&10 & 6.6985e-5 & 1.9218 && 1.8566e-4 & 1.9547 &&& 8.2084e-5  & 1.9355  && 2.7031e-4  & 1.9396 \\ 
			&20 & 1.7679e-5 & 1.9390 && 4.7896e-5 & 1.9742 &&& 2.1459e-5  & 1.9566  && 7.0464e-5  & 1.9668 \\ 
                &40 & 4.6108e-6 &        && 1.2190e-5 &        &&& 5.5286e-6  &         && 1.8026e-5  &        \\ \hline
			\multirow{4}{*}{$0.85$} 
			&5  & 7.7559e-5 & 1.9752 && 2.1285e-4 & 1.9651 &&& 8.3332e-5  & 1.9846 && 2.4252e-4  & 1.9602 \\ 	
			&10 & 1.9726e-5 & 1.9495 && 5.4515e-5 & 1.9795 &&& 2.1057e-5  & 1.9686 && 6.2326e-5  & 1.9777 \\ 
			&20 & 5.1070e-6 & 1.9271 && 1.3823e-5 & 1.9876 &&& 5.3802e-6  & 1.9578 && 1.5824e-5  & 1.9865 \\ 
                &40 & 1.3429e-6 &        && 3.4856e-6 &        &&& 1.3850e-6  &        && 3.9934e-6  &        \\ \hline
		\end{tabular}
	\end{center}
\end{table}

\begin{table}[ht]
\caption {Error ($L^2$ \& $L^\infty$) and rate of convergence (RoC) based on $L2$-$1_\sigma$ scheme with $2M_1=2M_2=32$ for Example \ref{example4} with solution {\em having time singularity (uniform mesh $\nu=1$)}.}\label{Ex4a_table2}
	\begin{center}
		\begin{tabular}{llllllll}
			\hline
			$\alpha$ & $N$ & $L^2$-error & RoC &&& $L^\infty$-error & RoC \\ \hline
			\multirow{4}{*}{$0.1$}  
			&5  & 4.4966e-4  & 0.5668  &&& 2.0422e-3  & 0.1313 \\ 	
			&10 & 3.0358e-4  & 0.5908  &&& 1.8645e-3  & 0.1261 \\ 
			&20 & 2.0157e-4  & 0.6069  &&& 1.7084e-3  & 0.1261 \\ 
                &40 & 1.3235e-4  &         &&& 1.5654e-3  &        \\ \hline
			\multirow{4}{*}{$0.3$} 
			&5  & 7.6773e-4  & 0.8322  &&& 3.4665e-3  & 0.3990 \\ 	
			&10 & 4.3122e-4  & 0.8805  &&& 2.6289e-3  & 0.4196 \\ 
			&20 & 2.3423e-4  & 0.9217  &&& 1.9653e-3  & 0.4458 \\ 
                &40 & 1.2365e-4  &         &&& 1.4430e-3  &        \\ \hline
			\multirow{4}{*}{$0.5$} 
			&5  & 6.3659e-4  & 1.1704 &&& 2.8479e-3   & 0.74083 \\ 	
			&10 & 2.8284e-4  & 1.2899 &&& 1.7042e-3   & 0.83734 \\ 
			&20 & 1.1567e-4  & 1.4477 &&& 9.5379e-4   & 0.97652 \\ 
                &40 & 4.2408e-5  &        &&& 4.8472e-4   &         \\ \hline
		\end{tabular}
	\end{center}
\end{table}

	
\begin{table}[ht]
	\caption {Error ($L^2$ \& $L^\infty$) and rate of convergence (RoC) based on $L2$-$1_\sigma$ scheme with $2M_1=2M_2=32$ for Example \ref{example4} having solution $\mathcal{U}\in C^3(\overline{\Omega}_t)$ {\em (uniform mesh $\nu=1$)}.}\label{Ex4b_table1}
	\begin{center}
		\begin{tabular}{llllllll}
			\hline
			$\alpha$ & $N$ & $L^2$-error & RoC &&& $L^\infty$-error & RoC \\ \hline
			\multirow{4}{*}{$0.4$}  
			&5  & 4.0744e-4 & 2.012 &&& 1.3910e-3 & 1.962 \\ 	
			&10 & 1.0101e-4 & 2.010 &&& 3.5696e-4 & 1.984 \\ 
			&20 & 2.5082e-5 & 2.007 &&& 9.0253e-5 & 1.993 \\ 
                &40 & 6.2403e-6 &       &&& 2.2668e-5 &       \\ \hline
			\multirow{4}{*}{$0.8$} 
			&5  & 7.5992e-4 & 2.060 &&& 2.7544e-3 & 1.980 \\ 	
			&10 & 1.8226e-4 & 2.038 &&& 6.9799e-4 & 1.995 \\ 
			&20 & 4.4364e-5 & 2.024 &&& 1.7513e-4 & 2.000 \\ 
                &40 & 1.0910e-5 &       &&& 4.3772e-5 &       \\ \hline
		\end{tabular}
	\end{center}
\end{table}

\section{Application}
In this section, we present two different fractional-order models in terms of practical implementation for special cases. The parameters and the initial and boundary conditions associated with the proposed problem (\ref{Pdoc1_1}) are chosen in such a way that it boils down to the following physical models: the {\em time-fractional Fokker-Planck equation} and the {\em time-fractional viscoelastic model}. 

\subsection{Fractional-order Fokker-Planck dynamics: Diffusion as a limit of random walk}
The stochastic motion of Brownian particles in a potential $U(x)=V(x)-Fx$, where $V(x)=V(x+L)$ represents a periodic substrate potential with period $L$ and $F$ is the constant bias, is well established in the literature. The Fokker-Planck equation describing the overdamped Brownian motion in this potential can be extended to account for anomalous transport. This extension is referred to as the {\em fractional Fokker-Planck} equation \cite{HeinsaluFractional2006}, and it is expressed as:
\[\dfrac{\partial}{\partial t}P(x,t)=\mathcal{J}^{1-\alpha}\Big[\dfrac{\partial}{\partial x}\dfrac{U'(x)}{\eta_\alpha}+K_\alpha\dfrac{\partial^2}{\partial x^2}\Big]P(x,t),\]
where $\alpha\in(0,1)$. $P(x,t)$ is the probability density function (pdf) and $K_\alpha$ denotes the anomalous diffusion coefficient with physical dimension [$m^2~s^{-\alpha}$]. The generalized friction coefficient, $\eta_\alpha$, has dimension [$kg~s^{\alpha-2}$]. $\mathcal{J}^{1-\alpha}$ denotes the Riemann-Liouville fractional integral operator \cite{MillerAnintroduction1993} of order $1-\alpha$ with respect to time. An equivalent form of the fractional Fokker-Planck equation can be expressed using the Caputo fractional derivative, as demonstrated in the following example.
\begin{example}\label{example5}
Consider the following time-fractional Fokker-Planck equation \cite{DengNumerical2007}:
\begin{equation*}
    \left\{
    \begin{array}{ll}
         \partial_t^\alpha P(x,t)-K_\alpha \dfrac{\partial^2}{\partial x^2}P(x,t)-\dfrac{U'(x)}{\eta_\alpha}\dfrac{\partial}{\partial x}P(x,t)-\dfrac{U''(x)}{\eta_\alpha}P(x,t)=0,~(x,t)\in(1,11)\times(0,1],\\[6pt]
         P(x,0)=\psi(x),~x\in[1,11],\\
         P(1,t)=\phi_1(t),~P(11,t)=\phi_2(t),~t\in[0,1],
    \end{array}\right.
\end{equation*}
\end{example}
where $\alpha\in(0,1)$. The parameters are given by: $U(x)=\cos x-Fx,~F=6,~\eta_\alpha=6,~K_\alpha=2,~\psi(x)=0.10,~\phi_1(t)=0.10,~\phi_2(t)=0.10$. The accuracy of the proposed methodology can be verified by estimating error and the rate of convergence using the formula provided in (\ref{Pdoc1_19}). It can be noted that the solution of this time-fractional Fokker-Planck equation does not violet the singular behaviour of the fractional operator as illustrated graphically in Figure \ref{Ex5_fig_fig1}(a). A cross-sectional view is displayed in Figure \ref{Ex5_fig_fig1}(b), which represents the evolution of the pdf $P(x,t)$ when $t=0.1,0.4,0.9$. The graphical representation ensures that the solution to the given time-fractional Fokker-Planck equation has unbounded time derivative at $t=0$ for which the present method leads to a $\mathcal{O}(N^{-2})$ temporal rate of accuracy on a suitably chosen non-uniform mesh as it is confirmed in Table \ref{Ex5_table1}. Also, see the theoretical explanation in Theorem \ref{Pdoc1_thm2}. On the contrary, it reduces to $\mathcal{O}(N^{-\alpha})$ on a uniform mesh. 


 \begin{figure}[ht]
	\begin{subfigure}{.5\textwidth}
		\centering
		\includegraphics[width=0.95\linewidth]{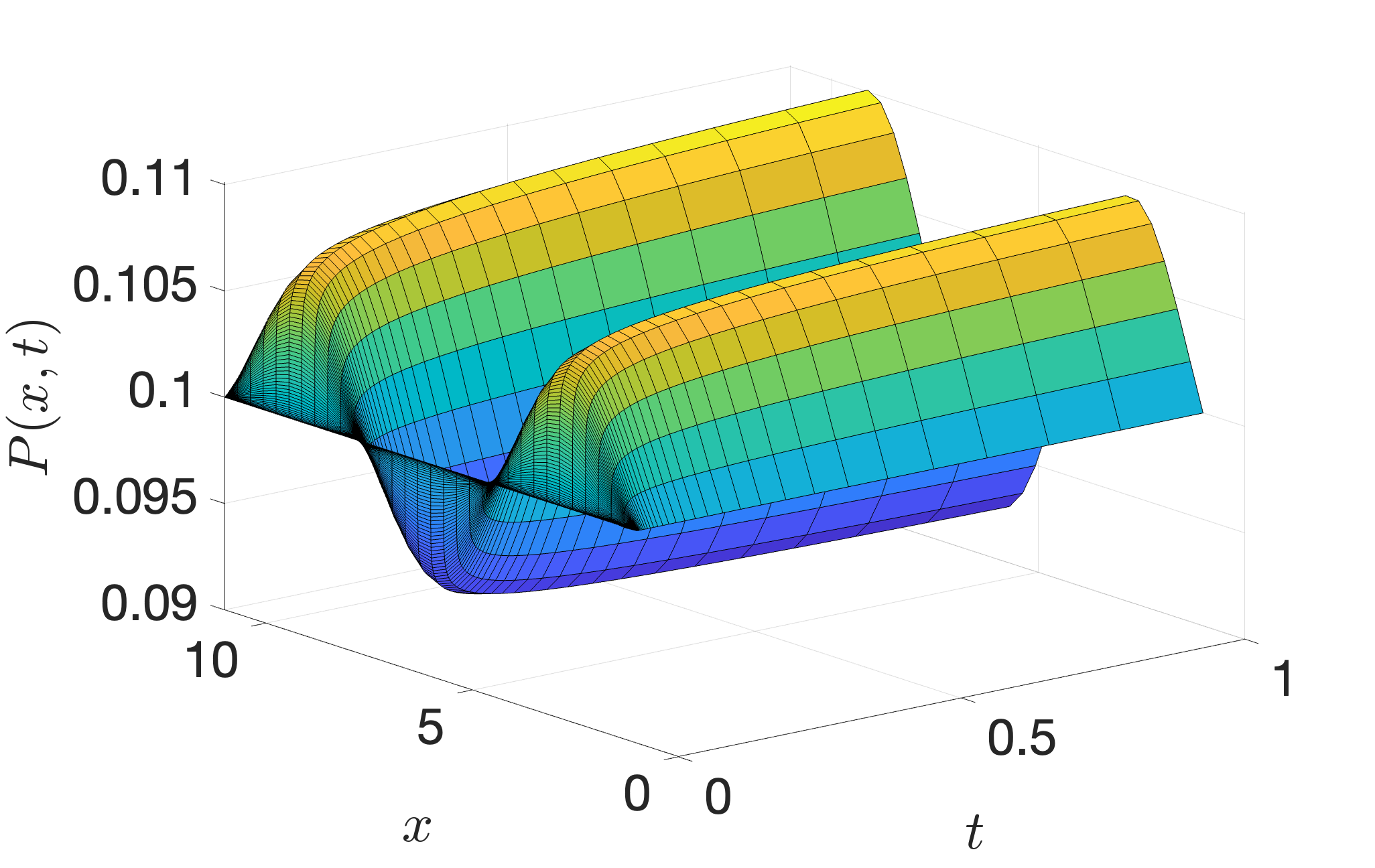}
		\caption{$M=32,~N=64$ and $\nu=2/\alpha$.}
		\label{Ex5_Fig1}
	\end{subfigure}
	\begin{subfigure}{.5\textwidth}
		\centering
		\includegraphics[width=0.95\linewidth]{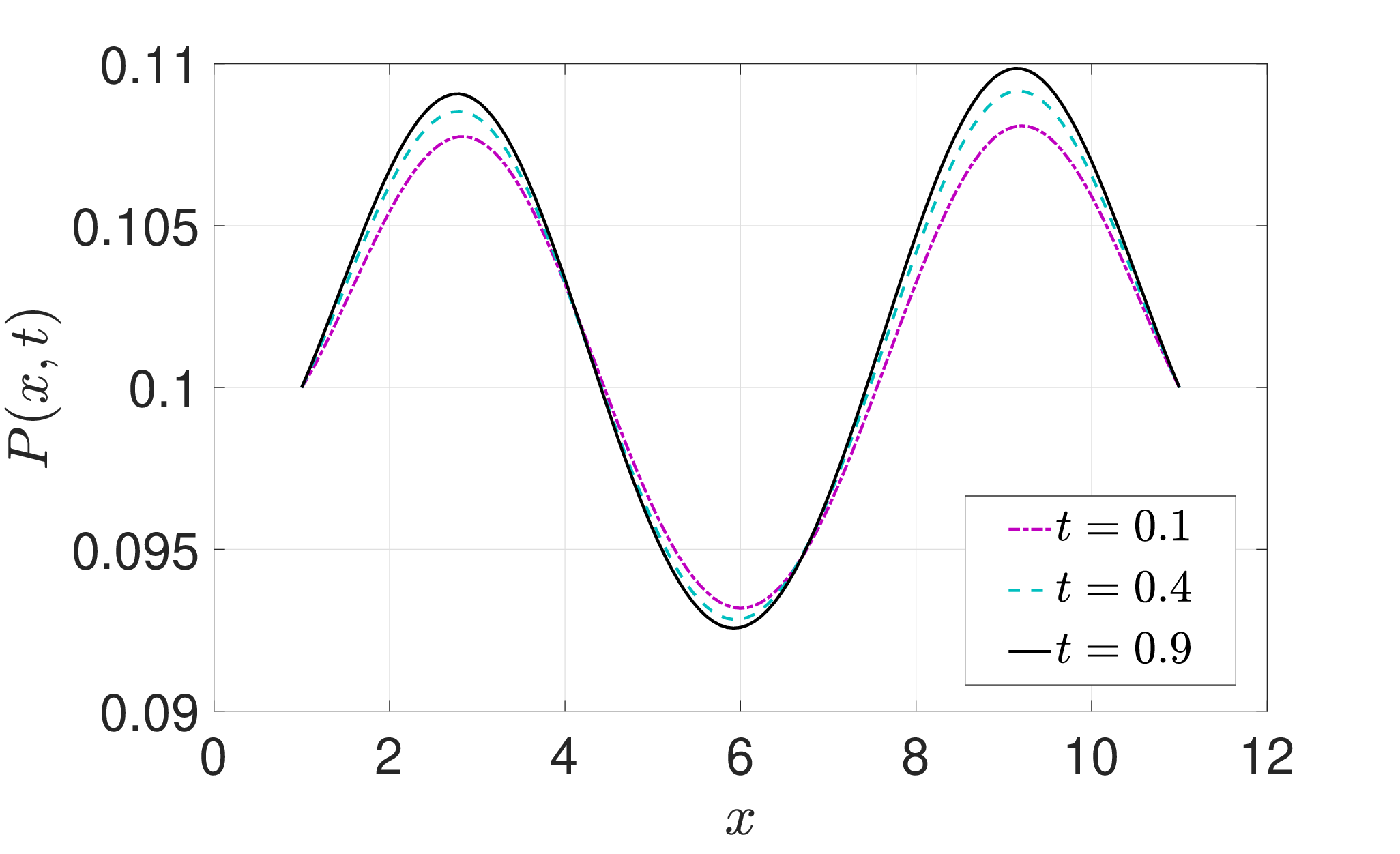}
		\caption{$M=64,~N=10$ and $\nu=1$.}
		\label{Ex5_Fig2}
	\end{subfigure}
	\caption{\small{Numerical solution to the time-fractional Fokker-Planck equation for $\alpha=0.2$.}}
	\label{Ex5_fig_fig1}
\end{figure}
\begin{table}[ht]
	\caption {Maximum error (Max-Error) and rate of convergence (RoC) with fixed $M=1000$ based on $L2$-$1_\sigma$ scheme for Example \ref{example5}.}\label{Ex5_table1}
	\begin{center}
		\begin{tabular}{llllllllllll}
			\hline
			$\alpha$ & $N$ & \multicolumn{6}{c}{Graded mesh} &&& \multicolumn{2}{c}{Uniform mesh} \\ \cline{3-8} \cline{11-12}
                         &     & \multicolumn{2}{c}{$\nu=2/\alpha$} &&& \multicolumn{2}{c}{$\nu=(3-\alpha)/\alpha$} &&& \multicolumn{2}{c}{$\nu=1$}\\\cline{3-8} \cline{11-12}
			&  & Max-Error & RoC &&& Max-Error & RoC &&& Max-Error & RoC \\ \hline
			\multirow{4}{*}{$0.3$}  
			&32  & 1.1919e-6  & 1.9478  &&& 7.5248e-7  & 1.8901 &&& 1.2523e-4  & 0.1230 \\ 	
			&64  & 3.0896e-7  & 1.9815  &&& 2.0301e-7  & 1.9163 &&& 1.1500e-4  & 0.1908 \\ 
			&128 & 7.8237e-8  & 1.9946  &&& 5.3781e-8  & 1.9395 &&& 1.0075e-4  & 0.2313 \\ 
                &256 & 1.9633e-8  &         &&& 1.4021e-8  &        &&& 8.5831e-5  &        \\ \hline
			\multirow{4}{*}{$0.5$} 
			&32  & 1.0893e-6  & 1.9537  &&& 7.0003e-7  & 1.8998 &&& 8.9586e-5  & 0.2786 \\ 	
			&64  & 2.8121e-7  & 1.9863  &&& 1.8759e-7  & 1.9129 &&& 7.3856e-5  & 0.3314 \\ 
			&128 & 7.0975e-8  & 1.9960  &&& 4.9817e-8  & 1.9253 &&& 5.8698e-5  & 0.3819 \\ 
                &256 & 1.7793e-8  &         &&& 1.3116e-8  &        &&& 4.5045e-5  &        \\ \hline
			\multirow{4}{*}{$0.7$} 
			&32  & 8.2781e-7  & 1.9345  &&& 6.3764e-7  & 1.8989 &&& 4.1895e-5  & 0.5951 \\ 	
			&64  & 2.1657e-7  & 1.9759  &&& 1.7098e-7  & 1.9148 &&& 2.7735e-5  & 0.6396 \\ 
			&128 & 5.5055e-8  & 1.9930  &&& 4.5344e-8  & 1.9301 &&& 1.7803e-5  & 0.6642 \\ 
                &256 & 1.3831e-8  &         &&& 1.1899e-8  &        &&& 1.1234e-5  &        \\ \hline
		\end{tabular}
	\end{center}
\end{table}
\subsection{Fractional-order Viscoelastic dynamics: An integral representation}
Let $\boldsymbol{\sigma}(x,t)$ and $\boldsymbol{\epsilon}$ denote the stress and strain, respectively, of a viscoelastic material. The temporal evolution of stress and strain in linear viscoelastic materials is commonly modeled by incorporating derivatives of both quantities. A standard viscoelastic model, as introduced by Bagley and Torvik \cite{BagleyFractional1983}, can be expressed as:
\[\boldsymbol{\sigma}(x,t)+\varrho_\epsilon\boldsymbol{\sigma}_t(x,t)=E_0\Big(\boldsymbol{\epsilon}(x,t) +\varrho_\sigma\boldsymbol{\epsilon}_t(x,t)\Big),\]
where $E_0$ is the prolonged modulus of the elasticity, $\varrho_\epsilon$ and $\varrho_\sigma$ represent the relaxation and retardation times, respectively. To capture more complex viscoelastic behaviors, particularly those exhibiting memory effects, fractional derivatives are introduced into the model. The general fractional-order viscoelastic model, as discussed in Konjik {\em et al.} \cite{KonjikWaves2010}, is given by:
\[\boldsymbol{\sigma}(x,t)+\varrho_\epsilon^\gamma\partial_t^\gamma\boldsymbol{\sigma}(x,t)=E_0\Big(\boldsymbol{\epsilon}(x,t) +\varrho_\sigma^\beta\partial_t^\beta\boldsymbol{\epsilon}(x,t)\Big),~\gamma<\beta.\]
$\gamma,~\beta$ are the orders of the fractional operators that lie between $0$ and $1$. It is important to note that the model represents instantaneous elasticity and describes wave processes if $\gamma<\beta$ whereas for $\gamma=\beta$, it lacks instantaneous elasticity and describes diffusion processes. Applying $\mathcal{J}^\gamma$ on both sides, one has the integral representation of the {\em fractional-order viscoelastic model} as:
\[\partial_t^{\alpha}\boldsymbol{\epsilon}(x,t)+\dfrac{1}{\varrho_\sigma^\beta\Gamma(\gamma)}\int_0^t(t-s)^{\gamma-1}\boldsymbol{\epsilon}(x,s)~ds=\mathcal{F}(x,t),\]
where $\alpha=\beta-\gamma\in(0,1)$, and $\mathcal{F}(x,t)=\dfrac{\varrho_\epsilon^\gamma}{E_0\varrho_\sigma^\beta}\boldsymbol{\sigma}(x,t)+\dfrac{1}{E_0\varrho_\sigma^\beta\Gamma(\gamma)}\displaystyle{\int_0^t}(t-s)^{\gamma-1}\boldsymbol{\sigma}(x,s)ds$. Here, the stress $\boldsymbol{\sigma}$ is considered to be known with $\boldsymbol{\sigma}(x,0)=0$, and the strain function $\boldsymbol{\epsilon}$ is unknown that satisfies the initial condition $\boldsymbol{\epsilon}(x,0)=0$. Hence, we have the more general version of the fractional order viscoelastic model described in the following example.
\begin{example}\label{example6}
Consider the following time-fractional viscoelastic model \cite{MohammadiComputational2024,AvazzadehLegendre2019} for $\Omega\subset\mathbb{R}^2$ involving weakly singular kernel:
\begin{equation*}
    \left\{
    \begin{array}{ll}
         \partial_t^\alpha \mathcal{U}(x,y,t)-\Delta\mathcal{U}(x,y,t)+\displaystyle{\int_0^t}(t-\xi)^{-\vartheta}\Delta\mathcal{U}(x,y,\xi)d\xi=f(x,y,t),~(x,y,t)\in\Omega\times\Omega_t,\\[6pt]
         \mathcal{U}(x,y,0)=0,~(x,y)\in\overline{\Omega},\\
         \mathcal{U}(x,y,t)=0,~(x,y,t)\in\partial\Omega\times\overline{\Omega}_t,
    \end{array}\right.
\end{equation*}
\end{example}
where $\alpha,\vartheta\in(0,1)$. The source function $f(x,y,t)$ is given by:
\[f(x,y,t)=\Big(\Gamma(\alpha+1)+8\pi^2t^\alpha-8\pi^2\boldsymbol{B}(\alpha+1,1-\vartheta)t^{1+\alpha-\vartheta}\Big)\sin2\pi x\sin2\pi y.\]
$\boldsymbol{B}$ denotes the well-known Euler beta function. The analytical solution for Example \ref{example6} is $\mathcal{U}=t^\alpha\sin2\pi x\sin2\pi y$. The proposed approximation to discretize the integral operator will not work straightforwardly as the operator involves a weakly singular kernel. It requires some advanced modification in the trapezoidal rule (see \cite{SantraAnalysis2024}) to deal with the weakly singular integral operator involving mixed derivatives. The $L2$-$1_\sigma$ scheme is used to discretize the time-fractional operator as well. Then, the multi-dimensional Haar wavelet approximation is used to proceed further. The computational output displayed in Table \ref{Ex6_table1} clearly indicates the high accuracy of the proposed approach on a non-uniform mesh. It also highlights the convergence on a uniformly distributed mesh.


\begin{table}[ht]
\caption {Error ($L^2$ \& $L^\infty$) and rate of convergence (RoC) based on $L2$-$1_\sigma$ scheme with $2M_1=2M_2=32$ for the viscoelastic model (Example \ref{example6}) having weakly singular kernel.}\label{Ex6_table1}
	\begin{center}
		\begin{tabular}{llllllllllllll}
			\hline
			$\alpha$ & $N$ & \multicolumn{5}{c}{$\nu=(3-\alpha)/\alpha$, $\vartheta=0.5$} &&& \multicolumn{5}{c}{$\nu=1$, $\vartheta=0.1$} \\ \cline{3-7}\cline{10-14}
                 &  & $L^2$-error & RoC    && $L^\infty$-error & RoC &&& $L^2$-error & RoC    && $L^\infty$-error & RoC\\ \hline
			\multirow{4}{*}{$0.5$}  
			&16  & 6.0343e-2 & 1.3273 && 4.3011e-1 & 1.1029 &&& 1.1695e-2  &  0.9105 && 3.7630e-2  & 0.7942 \\ 	
			&32  & 2.4047e-2 & 1.9259 && 2.0025e-1 & 1.8181 &&& 6.2217e-3  &  0.8806 && 2.1699e-2  & 0.5458 \\ 
			&64  & 6.3287e-3 & 2.3651 && 5.6789e-2 & 2.0359 &&& 3.3793e-3  &  0.7976 && 1.4864e-2  & 0.5996 \\ 
                &128 & 1.2284e-3 &        && 1.3848e-2 &        &&& 1.9441e-3  &         && 9.8091e-3  &        \\ \hline
			\multirow{4}{*}{$0.8$} 
			&16  & 6.6665e-2 & 1.1103 && 4.1806e-1 & 0.9735 &&& 1.3144e-2  & 0.8743  && 4.8259e-2  & 0.8706 \\ 	
			&32  & 3.0879e-2 & 1.4666 && 2.1291e-1 & 1.4013 &&& 7.1701e-3  & 0.8706  && 2.6394e-2  & 0.8781 \\ 
			&64  & 1.1173e-2 & 2.6109 && 8.0606e-2 & 2.8308 &&& 3.9213e-3  & 0.8187  && 1.4360e-2  & 0.8367 \\ 
                &128 & 1.8290e-3 &        && 1.1329e-2 &        &&& 2.2232e-3  &         && 8.0408e-3  &        \\ \hline
		\end{tabular}
	\end{center}
\end{table}

\begin{note}
It can be noticed that the models given in Examples \ref{example5} and \ref{example6} are the special cases of the given problem. By setting $\mu=0$ and appropriately choosing the coefficient functions, as well as the initial and boundary conditions, the proposed problem (\ref{Pdoc1_1}) reduces to the time-fractional Fokker-Planck equation without a source term for $\Omega\subset\mathbb{R}$. Similarly, introducing a weakly singular kernel within the Volterra integral operator with mixed derivatives reduces to a viscoelastic model in an integral representation. Consequently, the proposed problem represents a more generalized version of these models.
\end{note}

	
\section{Concluding remarks}\label{sec_concl}		
In this work, we have successfully employed the $L2$-$1_\sigma$ scheme combined with the multi-dimensional Haar wavelets to address the time-fractional integro-partial differential equations having unbounded time derivatives at $t=0$. The non-uniform mesh in time is proven to be more effective in addressing such time singularities compared to the uniform mesh. The analysis includes the stability of the proposed scheme on a non-uniform mesh, which is based on the $L^\infty$ norm. The convergence analysis leads to $\mathcal{O}(N^{-2}+M^{-2})$ accuracy (when $\Omega\subset\mathbb{R}$) and $\mathcal{O}(N^{-2}+\mathcal{M}^{-3})$ accuracy (when $\Omega\subset\mathbb{R}^2$) over the space-time domain for a suitable choice of the grading parameter. It also highlights the higher-order accuracy for a sufficiently smooth solution resides in $C^3(\overline{\Omega}_t)$. We apply the proposed method to numerous test examples. The experiments confirm the reliability of the proposed method. It also demonstrates the superiority of the proposed methodology in terms of accuracy compared to some existing methods available in the literature. The advancement of the proposed methodology is also demonstrated through the application of the time-fractional Fokker-Planck equation and the fractional-order viscoelastic dynamics having weakly singular kernels. The proposed method can be extended to analyze the given problem with multiple fractional operators. The delay in time may be considered for future investigations as it appears to be more challenging with multi-singularities.

\section*{Declarations}
\subsection*{Acknowledgement}
\noindent Sudarshan Santra wants to acknowledge the Axis Bank Centre for Mathematics and Computing, Indian Institute of Science, Bangalore for the financial assistance to carry out this research work in the Department of Computational and Data Sciences, Indian Institute of Science, Bangalore.

\subsection*{Ethical Approval}
In this manuscript, all the authors have agreed to authorship and approved the manuscript with consent for submission.
 
\subsection*{Funding Details}
\noindent No funding source is available for this research.

\subsection*{Data Availability Statements}
The data sets generated during and/or analyzed for the current study are available within this manuscript. No extra data is used for this research.

\subsection*{Conflicts of Interest and Declarations}	
The author declares that he does not have any conflicts of interest. In addition, he also declares that this work is not under consideration to anywhere.	
		
		

\bibliography{ReferencesSS}

\begin{thebibliography}{10}
\expandafter\ifx\csname url\endcsname\relax
  \def\url#1{\texttt{#1}}\fi
\expandafter\ifx\csname urlprefix\endcsname\relax\def\urlprefix{URL }\fi
\expandafter\ifx\csname href\endcsname\relax
  \def\href#1#2{#2} \def\path#1{#1}\fi

\bibitem{SinghEulerwavelets2024}
A.~Singh, A.~Kanaujiya, J.~Mohapatra, Euler wavelets method for optimal control problems of fractional integro-differential equations, J. Comput. Appl. Math. 454 (2025) Paper No. 116178.
\newblock \href {https://doi.org/10.1016/j.cam.2024.116178} {\path{doi:10.1016/j.cam.2024.116178}}.

\bibitem{RoohiAdaptive2020}
M.~Roohi, C.~Zhang, Y.~Chen, Adaptive model-free synchronization of different fractional-order neural networks with an application in cryptography, Nonlinear Dynamics 100~(4) (2020) 3979--4001.

\bibitem{AnAspace2021}
X.~An, F.~Liu, M.~Zheng, V.~V. Anh, I.~W. Turner, A space-time spectral method for time-fractional {B}lack-{S}choles equation, Appl. Numer. Math. 165 (2021) 152--166.
\newblock \href {https://doi.org/10.1016/j.apnum.2021.02.009} {\path{doi:10.1016/j.apnum.2021.02.009}}.

\bibitem{BjorklundErrorestimates2024}
M.~Bj\"{o}rklund, K.~Larsson, M.~G. Larson, Error estimates for finite element approximations of viscoelastic dynamics: the generalized {M}axwell model, Comput. Methods Appl. Mech. Engrg. 425 (2024) Paper No. 116933, 24.
\newblock \href {https://doi.org/10.1016/j.cma.2024.116933} {\path{doi:10.1016/j.cma.2024.116933}}.

\bibitem{DengNumerical2007}
W.~Deng, Numerical algorithm for the time fractional {F}okker-{P}lanck equation, J. Comput. Phys. 227~(2) (2007) 1510--1522.
\newblock \href {https://doi.org/10.1016/j.jcp.2007.09.015} {\path{doi:10.1016/j.jcp.2007.09.015}}.

\bibitem{MillerAnintroduction1993}
K.~S. Miller, B.~Ross, An introduction to the fractional calculus and fractional differential equations, A Wiley-Interscience Publication, John Wiley \& Sons, Inc., New York, 1993.

\bibitem{HanLinearized2023}
Y.~Han, X.~Huang, W.~Gu, B.~Zheng, Linearized transformed {$L1$} finite element methods for semi-linear time-fractional parabolic problems, Appl. Math. Comput. 458 (2023) Paper No. 128242, 14.
\newblock \href {https://doi.org/10.1016/j.amc.2023.128242} {\path{doi:10.1016/j.amc.2023.128242}}.

\bibitem{GraciaConvergence2018}
J.~L. Gracia, E.~O'Riordan, M.~Stynes, Convergence in positive time for a finite difference method applied to a fractional convection-diffusion problem, Comput. Methods Appl. Math. 18~(1) (2018) 33--42.
\newblock \href {https://doi.org/10.1515/cmam-2017-0019} {\path{doi:10.1515/cmam-2017-0019}}.

\bibitem{GaoAfinite2012}
G.-h. Gao, Z.-z. Sun, Y.-n. Zhang, A finite difference scheme for fractional sub-diffusion equations on an unbounded domain using artificial boundary conditions, J. Comput. Phys. 231~(7) (2012) 2865--2879.
\newblock \href {https://doi.org/10.1016/j.jcp.2011.12.028} {\path{doi:10.1016/j.jcp.2011.12.028}}.

\bibitem{LiFast2022}
M.~Li, Y.~Wei, B.~Niu, Y.-L. Zhao, Fast {L}2-1{$_\sigma $} {G}alerkin {FEM}s for generalized nonlinear coupled {S}chr\"{o}dinger equations with {C}aputo derivatives, Appl. Math. Comput. 416 (2022) Paper No. 126734, 22.
\newblock \href {https://doi.org/10.1016/j.amc.2021.126734} {\path{doi:10.1016/j.amc.2021.126734}}.

\bibitem{AlikhanovAhigh-order2021}
A.~A. Alikhanov, C.~Huang, A high-order {L}2 type difference scheme for the time-fractional diffusion equation, Appl. Math. Comput. 411 (2021) Paper No. 126545, 19.
\newblock \href {https://doi.org/10.1016/j.amc.2021.126545} {\path{doi:10.1016/j.amc.2021.126545}}.

\bibitem{ChenAnanalysis2019}
H.~Chen, F.~Holland, M.~Stynes, An analysis of the {G}r\"{u}nwald-{L}etnikov scheme for initial-value problems with weakly singular solutions, Appl. Numer. Math. 139 (2019) 52--61.
\newblock \href {https://doi.org/10.1016/j.apnum.2019.01.004} {\path{doi:10.1016/j.apnum.2019.01.004}}.

\bibitem{SantraHigherorder2023}
S.~Santra, J.~Mohapatra, P.~Das, D.~Choudhuri, Higher order approximations for fractional order integro-parabolic partial differential equations on an adaptive mesh with error analysis, Comput. Math. Appl. 150 (2023) 87--101.
\newblock \href {https://doi.org/10.1016/j.camwa.2023.09.008} {\path{doi:10.1016/j.camwa.2023.09.008}}.

\bibitem{BabaeiAnefficient2021}
A.~Babaei, S.~Banihashemi, C.~Cattani, An efficient numerical approach to solve a class of variable-order fractional integro-partial differential equations, Numer. Methods Partial Differential Equations 37~(1) (2021) 674--689.
\newblock \href {https://doi.org/10.1002/num.22546} {\path{doi:10.1002/num.22546}}.

\bibitem{DehghanALegendre2018}
M.~Dehghan, M.~Abbaszadeh, A {L}egendre spectral element method ({SEM}) based on the modified bases for solving neutral delay distributed-order fractional damped diffusion-wave equation, Math. Methods Appl. Sci. 41~(9) (2018) 3476--3494.
\newblock \href {https://doi.org/10.1002/mma.4839} {\path{doi:10.1002/mma.4839}}.

\bibitem{SantraAnovelapproach2022}
S.~Santra, A.~Panda, J.~Mohapatra, A novel approach for solving multi-term time fractional {V}olterra-{F}redholm partial integro-differential equations, J. Appl. Math. Comput. 68~(5) (2022) 3545--3563.
\newblock \href {https://doi.org/10.1007/s12190-021-01675-x} {\path{doi:10.1007/s12190-021-01675-x}}.

\bibitem{AbbasbandyOnconvergence2013}
S.~Abbasbandy, M.~S. Hashemi, I.~Hashim, On convergence of homotopy analysis method and its application to fractional integro-differential equations, Quaest. Math. 36~(1) (2013) 93--105.
\newblock \href {https://doi.org/10.2989/16073606.2013.780336} {\path{doi:10.2989/16073606.2013.780336}}.

\bibitem{HamoudModified2018}
A.~A. Hamoud, K.~P. Ghadle, Modified {L}aplace decomposition method for fractional {V}olterra-{F}redholm integro-differential equations, J. Math. Model. 6~(1) (2018) 91--104.
\newblock \href {https://doi.org/10.15393/j3.art.2018.4350} {\path{doi:10.15393/j3.art.2018.4350}}.

\bibitem{SantraAnovel2022}
S.~Santra, J.~Mohapatra, A novel finite difference technique with error estimate for time fractional partial integro-differential equation of {V}olterra type, J. Comput. Appl. Math. 400 (2022) Paper No. 113746, 13.
\newblock \href {https://doi.org/10.1016/j.cam.2021.113746} {\path{doi:10.1016/j.cam.2021.113746}}.

\bibitem{SantraAnalysis2022}
S.~Santra, J.~Mohapatra, Analysis of a finite difference method based on {L}1 discretization for solving multi-term fractional differential equation involving weak singularity, Math. Methods Appl. Sci. 45~(11) (2022) 6677--6690.
\newblock \href {https://doi.org/10.1002/mma.8199} {\path{doi:10.1002/mma.8199}}.

\bibitem{SantraAnalysis2021}
S.~Santra, J.~Mohapatra, Analysis of the {L}1 scheme for a time fractional parabolic-elliptic problem involving weak singularity, Math. Methods Appl. Sci. 44~(2) (2021) 1529--1541.
\newblock \href {https://doi.org/10.1002/mma.6850} {\path{doi:10.1002/mma.6850}}.

\bibitem{StynesErroranalysis2017}
M.~Stynes, E.~O'Riordan, J.~L. Gracia, Error analysis of a finite difference method on graded meshes for a time-fractional diffusion equation, SIAM J. Numer. Anal. 55~(2) (2017) 1057--1079.
\newblock \href {https://doi.org/10.1137/16M1082329} {\path{doi:10.1137/16M1082329}}.

\bibitem{GaoAnew2014}
G.-h. Gao, Z.-z. Sun, H.-w. Zhang, A new fractional numerical differentiation formula to approximate the {C}aputo fractional derivative and its applications, J. Comput. Phys. 259 (2014) 33--50.
\newblock \href {https://doi.org/10.1016/j.jcp.2013.11.017} {\path{doi:10.1016/j.jcp.2013.11.017}}.

\bibitem{AlikhanovAnew2015}
A.~A. Alikhanov, A new difference scheme for the time fractional diffusion equation, J. Comput. Phys. 280 (2015) 424--438.
\newblock \href {https://doi.org/10.1016/j.jcp.2014.09.031} {\path{doi:10.1016/j.jcp.2014.09.031}}.

\bibitem{BeheraMultilevel2015}
R.~Behera, M.~Mehra, N.~K.-R. Kevlahan, Multilevel approximation of the gradient operator on an adaptive spherical geodesic grid, Adv. Comput. Math. 41~(3) (2015) 663--689.
\newblock \href {https://doi.org/10.1007/s10444-014-9382-z} {\path{doi:10.1007/s10444-014-9382-z}}.

\bibitem{SchneiderWavelet2010}
K.~Schneider, O.~V. Vasilyev, Wavelet methods in computational fluid dynamics, in: Annual review of fluid mechanics. {V}ol. 42, Vol.~42 of Annu. Rev. Fluid Mech., Annual Reviews, Palo Alto, CA, 2010, pp. 473--503.
\newblock \href {https://doi.org/10.1146/annurev-fluid-121108-145637} {\path{doi:10.1146/annurev-fluid-121108-145637}}.

\bibitem{PervaizHaar2020}
N.~Pervaiz, I.~Aziz, Haar wavelet approximation for the solution of cubic nonlinear {S}chrodinger equations, Phys. A 545 (2020) 123738, 17.
\newblock \href {https://doi.org/10.1016/j.physa.2019.123738} {\path{doi:10.1016/j.physa.2019.123738}}.

\bibitem{Siraj-ul-IslamAcomparative2010}
S.~ul~Islam, I.~Aziz, F.~Haq, A comparative study of numerical integration based on {H}aar wavelets and hybrid functions, Comput. Math. Appl. 59~(6) (2010) 2026--2036.
\newblock \href {https://doi.org/10.1016/j.camwa.2009.12.005} {\path{doi:10.1016/j.camwa.2009.12.005}}.

\bibitem{FaheemAhigh2022}
M.~Faheem, A.~Khan, A.~Raza, A high resolution {H}ermite wavelet technique for solving space-time-fractional partial differential equations, Math. Comput. Simulation 194 (2022) 588--609.
\newblock \href {https://doi.org/10.1016/j.matcom.2021.12.012} {\path{doi:10.1016/j.matcom.2021.12.012}}.

\bibitem{KeshavarzBernoulli2014}
E.~Keshavarz, Y.~Ordokhani, M.~Razzaghi, Bernoulli wavelet operational matrix of fractional order integration and its applications in solving the fractional order differential equations, Appl. Math. Model. 38~(24) (2014) 6038--6051.
\newblock \href {https://doi.org/10.1016/j.apm.2014.04.064} {\path{doi:10.1016/j.apm.2014.04.064}}.

\bibitem{LiSolving2010}
Y.~Li, Solving a nonlinear fractional differential equation using {C}hebyshev wavelets, Commun. Nonlinear Sci. Numer. Simul. 15~(9) (2010) 2284--2292.
\newblock \href {https://doi.org/10.1016/j.cnsns.2009.09.020} {\path{doi:10.1016/j.cnsns.2009.09.020}}.

\bibitem{HaarZur1910}
A.~Haar, Zur {T}heorie der orthogonalen {F}unktionensysteme, Math. Ann. 69~(3) (1910) 331--371.

\bibitem{WichailukkanaAconvergence2016}
N.~Wichailukkana, B.~Novaprateep, C.~Boonyasiriwat, A convergence analysis of the numerical solution of boundary-value problems by using two-dimensional haar wavelets, Sci. Asia 42 (2016) 346--355.
\newblock \href {https://doi.org/10.2306/scienceasia1513-1874.2016.42.346} {\path{doi:10.2306/scienceasia1513-1874.2016.42.346}}.

\bibitem{ChenError2019}
H.~Chen, M.~Stynes, Error analysis of a second-order method on fitted meshes for a time-fractional diffusion problem, J. Sci. Comput. 79~(1) (2019) 624--647.
\newblock \href {https://doi.org/10.1007/s10915-018-0863-y} {\path{doi:10.1007/s10915-018-0863-y}}.

\bibitem{SinghAfully2024}
A.~Singh, S.~Kumar, J.~Vigo-Aguiar, A fully discrete scheme based on cubic splines and its analysis for time-fractional reaction-diffusion equations exhibiting weak initial singularity, J. Comput. Appl. Math. 434 (2023) Paper No. 115338, 19.
\newblock \href {https://doi.org/10.1016/j.cam.2023.115338} {\path{doi:10.1016/j.cam.2023.115338}}.

\bibitem{LinFinite2007}
Y.~Lin, C.~Xu, Finite difference/spectral approximations for the time fractional diffusion equation, J. Comput. Phys. 225~(2) (2007) 1533--1552.
\newblock \href {https://doi.org/10.1016/j.jcp.2007.02.001} {\path{doi:10.1016/j.jcp.2007.02.001}}.

\bibitem{HeinsaluFractional2006}
E.~Heinsalu, M.~Patriarca, I.~Goychuk, G.~Schmid, P.~H{\"a}nggi, Fractional fokker-planck dynamics: Numerical algorithm and simulations, Physical Review E 73~(4) (2006) 046133.
\newblock \href {https://doi.org/10.1103/PhysRevE.73.046133} {\path{doi:10.1103/PhysRevE.73.046133}}.

\bibitem{BagleyFractional1983}
R.~L. Bagley, P.~J. Torvik, Fractional calculus-a different approach to the analysis of viscoelastically damped structures, AIAA journal 21~(5) (1983) 741--748.

\bibitem{KonjikWaves2010}
S.~Konjik, L.~Oparnica, D.~Zorica, Waves in fractional {Z}ener type viscoelastic media, J. Math. Anal. Appl. 365~(1) (2010) 259--268.
\newblock \href {https://doi.org/10.1016/j.jmaa.2009.10.043} {\path{doi:10.1016/j.jmaa.2009.10.043}}.

\bibitem{MohammadiComputational2024}
H.~Mohammadi-Firouzjaei, H.~Adibi, M.~Dehghan, Computational study based on the {L}aplace transform and local discontinuous {G}alerkin methods for solving fourth-order time-fractional partial integro-differential equations with weakly singular kernels, Comput. Appl. Math. 43~(6) (2024) Paper No. 324, 19.
\newblock \href {https://doi.org/10.1007/s40314-024-02813-4} {\path{doi:10.1007/s40314-024-02813-4}}.

\bibitem{AvazzadehLegendre2019}
Z.~Avazzadeh, M.~Heydari, C.~Cattani, Legendre wavelets for fractional partial integro-differential viscoelastic equations with weakly singular kernels, The European Physical Journal Plus 134~(7) (2019) 368.
\newblock \href {https://doi.org/10.1140/epjp/i2019-12743-6} {\path{doi:10.1140/epjp/i2019-12743-6}}.

\bibitem{SantraAnalysis2024}
S.~Santra, Analysis of a higher-order scheme for multi-term time-fractional integro-partial differential equations with multi-term weakly singular kernels, Numer. Algorithms (2024) 1--47\href {https://doi.org/10.1007/s11075-024-01927-4} {\path{doi:10.1007/s11075-024-01927-4}}.

\end{thebibliography}

\end{document}